\documentclass[10pt,reqno]{amsart}
\usepackage[T1]{fontenc}
\usepackage[ngerman,french,english]{babel}
\usepackage[latin9,utf8]{inputenc}
\usepackage{a4wide}
\usepackage{amsthm}
\usepackage{amssymb}
\usepackage{setspace}
\usepackage{enumerate}
\usepackage{fancyhdr}
\usepackage{xcolor}
\usepackage{upgreek} 
\usepackage{mathabx}
\usepackage{lineno}
\usepackage[colorlinks=true]{hyperref}
\hypersetup{linkcolor=red,citecolor=blue,filecolor=dullmagenta,urlcolor=blue}

\numberwithin{equation}{section}
\theoremstyle{plain}
\newtheorem{thm}{Theorem}[section]
\newtheorem{lem}[thm]{Lemma}
\newtheorem{prop}[thm]{Proposition}
\newtheorem{cor}[thm]{Corollary}
\newtheorem{claim}[thm]{Claim}

\theoremstyle{definition}

\newtheorem{rem}{Remark}[section]

\theoremstyle{remark}

\newcommand{\R}{\mathbb{R}}

\newcommand{\I}{\mathcal{I}(t)}
\newcommand{\LL}{\mathcal{L}}

\newcommand{\PL}{\partial^2_x\LL}
\newcommand{\J}{\mathcal{J}}
\newcommand{\M}{\mathcal{M}}
\newcommand{\N}{\mathcal{N}}

\newcommand{\vA}{\varphi_A}

\newcommand{\IOpg}{(1-\gamma \partial_x^2)^{-1}}
\newcommand{\Int}[1]{\int_{#1}}
\newcommand{\Jap}[2]{\langle {#1}, {#2} \rangle}

\newcommand{\sgn}{\operatorname{sgn}}
\newcommand{\tv}{\tilde{v}}

\newcommand{\HH}{\boldsymbol{H}}
\newcommand{\bd}[1]{\boldsymbol{#1}}

\makeatother

\usepackage{babel}
\addto\shorthandsspanish{\spanishdeactivate{~<>}}

\addto\captionsenglish{}
\addto\captionsenglish{}
\addto\captionsenglish{}
\addto\captionsenglish{}
\addto\captionsenglish{}
\addto\captionsenglish{}
\addto\captionsenglish{}
\addto\captionsspanish{}
\addto\captionsspanish{}
\addto\captionsspanish{}
\addto\captionsspanish{}
\addto\captionsspanish{}
\addto\captionsspanish{}
\addto\captionsspanish{}

\newcommand{\sech}{\operatorname{sech}}

\begin{document}
	\title[Asymptotic stability of the fourth order $\phi^4$ kink]{Asymptotic stability of the fourth order $\phi^4$ kink for general perturbations in the energy space}
	\author[Christopher Maul\'en]{Christopher Maul\'en}  
	\address{Fakultat f\"ur Mathematik, Universit\"at Bielefeld,  Postfach 10 01 31, 33501 Bielefeld, Germany.}
	\email{cmaulen@math.uni-bielefeld.de}
	\thanks{Ch.Ma. was partially funded by the Deutsche Forschungsgemeinschaft (DFG, German Research Foundation) – Project-ID 317210226 – SFB 1283, and Chilean research grants FONDECYT 1191412.}
	\author[Claudio Mu\~noz]{Claudio Mu\~noz}  
	\address{Departamento de Ingenier\'{\i}a Matem\'atica and Centro
de Modelamiento Matem\'atico (UMI 2807 CNRS), Universidad de Chile, Casilla
170 Correo 3, Santiago, Chile.}
	\email{cmunoz@dim.uchile.cl}
	\thanks{Cl.Mu. was partially funded by Chilean research grants FONDECYT 1191412, 1231250, and Basal CMM FB210005.}


\maketitle

\begin{otherlanguage}{english} 
	\begin{abstract}
		The Fourth order $\phi^4$ model generalizes the classical $\phi^4$ model of quantum field theory, sharing the same kink solution. It is also the dispersive counterpart of the well-known parabolic Cahn-Hilliard equation. Mathematically speaking, the kink is characterized by a fourth order nonnegative linear operator with a simple kernel at the origin but no spectral gap. In this paper, we consider the kink of this theory, and prove orbital and asymptotic stability for any perturbation in the energy space. 
	\end{abstract}
\end{otherlanguage}
	
	\tableofcontents

	\section{Introduction}
	
		\subsection{Setting}
		In this work we consider the \emph{fourth order $\phi^4$ model}, or Wave-Cahn-Hilliard equation. This model reads in 1D as
		\begin{equation}\label{eq:WCH}
				  \partial_t^2 \phi+\partial^2_x \left( \partial^2_x \phi+ \phi-\phi^3 \right)=0, \quad (t,x)\in \R^2,
		\end{equation}
where $\phi=\phi(t,x)$ is assumed real-valued. Equation \eqref{eq:WCH} with quadratic nonlinearity and an additional linear term of order zero was considered by Bretherton \cite{Bre} as a model for studying weakly-nonlinear wave dispersion. It also appears with an additional parabolic term in \cite{FLM}.

\medskip

There are several interesting and profound motivations to study \eqref{eq:WCH}. The first and more obvious one is given by its similarity with a recognized model of degenerate phase transitions. Indeed, the ``parabolic'' version of \eqref{eq:WCH} is the well-known Cahn-Hilliard model
 \begin{equation}\label{eq:CH}
  \partial_t \phi + \partial^2_x \left( \partial^2_x \phi+ \phi-\phi^3\right)=0, \quad x\in\R, \ t\geq 0,
 \end{equation}
 that arises introduced by Cahn and Hilliard in the study of phase separation in cooling binary solutions such as alloys, glasses and polymer mixtures \cite{CH1958}, see additionally \cite{NCS1984,NC1985}. Major mathematical advances have been obtained for this model during past years, not only in 1D but also in several dimensions. In this setting, a key question is the long time behavior of kinks, usually referred in the community as fronts. Among other essential works, one finds the foundational result by Bricmont, Kupianen and Taskinen \cite{BKT}, who showed the stability and asymptotic stability of fronts in 1D for \eqref{eq:CH}. Many subsequent works have extended and improved this achievement. These results will be described in detail in Subsection \ref{previous}.

\medskip

Additionally, \eqref{eq:WCH} can be recast as the natural fourth order generalization of the well-known $\phi^4$ model of quantum field theory:
\begin{equation}\label{eq:phi4}
\partial_t^2 \phi - \left( \partial^2_x \phi+ \phi-\phi^3 \right)=0, \quad (t,x)\in \R^2,
\end{equation}
that has been extensively studied during past years. The $\phi^4$ kink
\begin{equation}\label{Kink}
H(x)=\tanh\left(\frac{x}{\sqrt{2}} \right),
\end{equation}
 is a notable example of simple solution of \eqref{eq:phi4} with a challenging dynamical behavior, and its general behavior is still unknown to us. In \cite{KMM2017} (see also \cite{KM2022}), the authors showed asymptotic stability for odd perturbations in the energy space. A proof for the general case is still missing.

\medskip

Another interesting motivation for the study of \eqref{eq:WCH} is given by its defocusing character, opposite to the focusing one present in the good-Boussinesq model. See \cite{M_GB} for a detailed introduction to the soliton problem in good-Boussinesq models. The latter model and the dynamics of their (unstable) solitons has been studied by many authors during past years. But, contrary to good-Boussinesq, the model studied in this work will have more physical interest. 

\medskip

Precisely, the fourth and final motivation to study \eqref{eq:WCH} comes from the fact that it possesses, as well as Cahn-Hilliard and $\phi^4$, stable kink solutions. In the fourth order $\phi^4$ case, the kink coincides with the one for $\phi^4$, and as far as we understand, for \eqref{eq:WCH} there is not even a single kink stability result. Consequently, in this paper our main objective is to initiate the study of the dynamics of kinks for \eqref{eq:WCH} dealing with the 1D case for general data in the energy space. 

\medskip 
		
	Following Linares \cite{Linares_93}, equation \eqref{eq:WCH} can be rewritten formally taken  
		\[
		\bd{\phi}=(\phi_1,\phi_2)=(\phi, \partial_x^{-1}\partial_t \phi),
		\]
		obtaining the following representation as a $2\times2$ system:
		\begin{align}\label{eq:WCH_syst} 
		\begin{cases}
		\partial_t \phi_1=\partial_x \phi_2 \\
		\partial_t \phi_2=-\partial_x (\partial_x^2 \phi_1 + \phi_1- \phi_1^3).
		\end{cases}
		\end{align}
This system admits kink solutions given by
\[
{\boldsymbol{H}_c}=(H_c,-cH_c)(x-ct-x_0), \quad c\in\R, \quad x_0\in\R,
\]
where $\pm H_c$ solves the ODE
\begin{equation}\label{eq:H_c}
 H''_c+(c^2+1)H_c-H_c^3 =0,\quad H_c(\pm\infty) =\pm \sigma= \pm\sqrt{1+c^{2}}. 
 \end{equation} 
One has that $H_c$ is a rescaled version of the $\phi^4$ kink:
\[
H_c(x)=\sigma H(\sigma x), \quad H(x) \; \hbox{given in \eqref{Kink}}, 
\]
that differs from the general kink from \eqref{eq:phi4} in the sense that the former only has the Lorentz boost as its scaling symmetry. This subtle difference will remain very important for the main results of this paper. 

System \eqref{eq:WCH_syst} will be the exact model worked in this paper. It has the following associated conserved quantities:	
\begin{equation}\label{eq:energy_momentum}
\begin{aligned}
E[\bd{\phi}]=E[\phi_1,\phi_2] =&~{}  \frac12 \int \big[(\partial_x \phi_1)^2+\phi_2^2+\frac12(\phi_1^2-1)^2\big]\ \ \  &(\mbox{Energy}) ,\quad \\
 P[\bd{\phi}]=P[\phi_1,\phi_2] =&~{}\int \phi_1\phi_2 \quad \quad  &(\mbox{Momentum}) .
\end{aligned}
\end{equation}	
(Here $\int$ means $\int_{\mathbb R} dx$.) 
The quantity $P$ has no good meaning around the kink solution if one works barely in the energy space, and will be loosely used in this general form. However, its second variation around the kink is perfectly well-defined and will be the key element to study the long time behavior of kinks. Following \cite{KMMV}, the set of functions $\boldsymbol{\phi}\in L^1_{loc}(\R)\times L^1_{loc}(\R)$ for which the energy is finite is
\begin{equation}\label{eq:E}
\boldsymbol{E}=\left\{(\phi_1,\phi_2)\in L^1_{loc}(\R)\times L^1_{loc}(\R){}~ \; \big| \; ~{} \partial_x \phi_1, \phi_2 \in L^2(\R), \; (\phi_1^2-1)\in L^2 \right\}.
\end{equation}
To study the stability of the static kink 
\begin{equation}\label{eq:kink}
\HH(x)=\boldsymbol{H}_0=(H,0)=\left(\tanh\left(\frac{x}{\sqrt{2}}\right),0\right),
\end{equation}
we introduce the subset of $\boldsymbol{E}$ in \eqref{eq:E} given by
\begin{equation}\label{eq:EH}
\boldsymbol{E_H}=\{\boldsymbol{\phi} \in \boldsymbol{E}~{} \big|~{} \boldsymbol{\phi-H}\in H^1(\R)\times L^2 (\R)\}.
\end{equation}
	Let us consider a shift modulation $\rho(t)$ induced by a perturbation $\boldsymbol{\phi}\in \boldsymbol{E_H}$  in  \eqref{eq:WCH_syst} of $\HH$ of the form 
	\[
		 \phi_1(t,x)=H(x-\rho(t))+u_1(t,x),\ \ \ \phi_2(t,x)=u_2(t,x),
	\]
	such that
	\begin{equation}\label{rem_data}
	 \langle \bd{u}(t) , {\bd H'}(\cdot- \rho(t)) \rangle = \langle u_1(t),H'(x-\rho(t)) \rangle =0, \quad  \langle a,b\rangle :=\int ab \, dx.
	\end{equation}
	Then from \eqref{eq:WCH_syst} one can see that the perturbation satisfies the following space-time, variable coefficients system:
		\begin{align} \label{eq:eq_lin}
		\begin{cases}
		\partial_t u_1= \partial_x u_2 + \rho' H'(\cdot-\rho) \\
		\partial_t u_2
				= \partial_x \LL u_1+\partial_x\left(3H(\cdot-\rho)u_1^2+u_1^3\right), 
		\end{cases}
		\end{align}
where
\begin{align}\label{eq:LL}
\LL=-\partial_x^2+V_0 (\cdot -\rho),
\quad \mbox{with} \quad V_0(x)=-1+3H^2(x)=2-3\sech^2\left( \frac{x}{\sqrt{2}}\right).
\end{align}
Note that $\mathcal L$ coincides with the unbounded linear operator appearing around the $\phi^4$ kink. Recalling \cite{Lohe,KMM2017,KMMV} (see also Lemma \ref{PropL} for further details), one has that $\mathcal L$ is nonnegative and it has absolutely continuous spectrum $[2,\infty)$. Additionally, the discrete spectrum consists of the simple eigenvalues $\lambda_0=0$ and $\lambda_1=\frac{\sqrt{3}}2$, with eigenfunctions
\begin{equation}\label{Y0}
Y_0(x)=H' \quad \mbox{ and }  \quad Y_1(x)=\sech\left(\frac{x}{\sqrt{2}}\right)\tanh\left(\frac{x}{\sqrt{2}}\right),
\end{equation}
respectively. Finally, $\lambda_2=2$ is a threshold resonance, in the sense that $\mathcal L \phi =2\phi$ possesses a smooth $L^\infty\backslash L^2$ solution with spatial derivative in $L^2$. 

\medskip

In the case of \eqref{eq:eq_lin} the situation has its own particularities, similar to the ones already present in \cite{M_GB}, in the sense that now the corresponding linear operator is $-\partial_x^2 \mathcal L$, which is of fourth order and the composition of two second order operators. The properties of this operator differ from $\mathcal L$ itself in the sense that (Lemma \ref{lem:PLL})  $-\partial_x^2 \mathcal L$ is nonnegative and it has absolutely continuous spectrum $[0,\infty)$, namely as well as in KdV, there is no spectral gap. There are no embedded eigenvalues in the continuous spectrum nor resonances but $\lambda_0=0$ is an eigenvalue with kernel $Y_0$ in \eqref{Y0}, generated by the invariance under shifts of the model. There is an additional linearly growing odd mode at the origin, which in 1D is usually harmless. Notice that any suitable scaling modulation of the kink costs infinite energy and by this reason will not be present in this work.

 \subsection{Main results} Our first result is the orbital stability of the fourth order $\phi^4$ kink:

\begin{thm}\label{thm:orbital}
Let $\HH$ be the kink introduced in \eqref{eq:kink}. There exists $\delta_{*}>0$ and $C^{*}>0$ such that for any $\bd{\phi^{in}}\in \bd{E_H}$, with
\begin{equation}\label{eq:Initial_data_a_kink0}
\|\bd{\phi^{in}-H}\|_{H^1\times L^2}\leq \delta,
\end{equation}
there exists a unique global solution $\bd{\phi}\in C(\R, \bd{E_{H}})$ of \eqref{eq:WCH_syst} with $\bd{\phi}(0)=\bd{\phi^{in}}$. Moreover, for some smooth $\rho(t)\in\R$, one has
\begin{equation}\label{OS}
\sup_{t\in \R} \|\bd{\phi}(t,\cdot +\rho(t))-\HH \|_{H^1\times L^2}\leq  C^{*} \|\bd{\phi^{in}-H}\|_{H^1\times L^2}.
\end{equation}
\end{thm}

Some previous results are needed for the proof of this theorem, among them a well-posednes theory for perturbations of the kink in the energy space whose proof relies on standard energy arguments; see e.g. \cite{KMMV}. The proof of this fact is given in Section \ref{sec:pre} and follows the ideas by Linares \cite{Linares_93}.

\medskip
In the case of scalar field models, orbital stability was proved by Henry, Perez and Wreszinski \cite{HPW} in the case of static kinks. The general case is contained in \cite{KMMV}. 

\medskip
Although the proof of Theorem \ref{thm:orbital} will follow standard ideas, it is worth to mention that no result of this type appeared in the literature for the case of the fourth order $\phi^4$ model. An interesting open question is to prove \eqref{OS} in the case of infinite energy perturbations that allow for scaling variations. In that case, the energy will be no longer useful at least in the classical sense. Our main results, orbital and asymptotic stability, stated for data in the energy space only, provide a satisfactory answer as in the parabolic setting \cite{BKT}. Indeed, in the ``hyperbolic'' case, the final state asymptotics will  also found in the case of general data perturbations.

\begin{thm}[Asymptotic stability]\label{thm:asymptotic}
Under the assumption \eqref{eq:Initial_data_a_kink0} in Theorem \ref{thm:orbital}, and by making $\delta_*>0$ smaller if necessary, one has 
that for any compact interval $I$ of $\R$ and $\gamma>0$ small enough,
\begin{equation}\label{eq:local_stability}
\lim_{t\to +\infty}  \left(  \left\| \phi(t, \cdot + \rho(t))-H \right\|_{L^\infty(I)} + \left\| \IOpg  \partial_t \phi(t, \cdot + \rho(t)) \right\|_{H^1(I)} \right)=0.
\end{equation}
\end{thm}

Recall that convergence on the whole line will imply that the data is the kink itself. In that sense, Theorem \ref{thm:asymptotic} is optimal if one considers the energy space topology only. Compared with the results in \cite{BKT} for the parabolic Cahn-Hilliard model and subsequent improvements, the Hamiltonian character of the dynamics forbids a better understanding of the exterior regions without the use of well-chosen weighted norms. Additionally, the parabolic dynamics possesses several key elements that are not present here, the most important being the presence of good decaying functionals to measure the long time dynamics.

\medskip

The parameter $\gamma>0$ in Theorem \ref{thm:asymptotic} depends on $\delta$, in particular, any fixed $0<\gamma\sim \delta^{2/5}$ suffices to prove \eqref{eq:local_stability}. Additionally, general data with no restriction naturally induce shifts in the dynamics. This shift parameter in \eqref{eq:local_stability} satisfies a particular equation that suggests that, unless one asks for additional space decay at time zero, there will not be convergence of $\rho(t)$ as time tends to infinity. This is in strong contrast with scalar field models as in \cite{KK1,KK2,KMMV}, where convergence was ensured by quadratic estimates on the shift parameters.

\medskip

It is also interesting to compare Theorem \ref{thm:asymptotic} with the foundational asymptotic stability result for the defocusing mKdV kink obtained by Merle and Vega \cite{MV}. The fourth order $phi^4$ kink requires more care because of its character as a nonlinear system of equations. They showed weak convergence of any (suitably shifted) perturbation of the kink towards the kink itself, once again thanks to the impossibility of performing scalings without spending an infinite amount of energy. This structural rigidity is present in other models, see e.g. the case of the Peregrine breather \cite{Mun_insNLS}. We believe that our techniques provide a locally stronger version of the results in \cite{MV}, and probably improvements of \cite{Mu} as well.

\subsection{Previous results}\label{previous}
We briefly comment now the main previous contributions to the kink asymptotic stability problem in our setting. We classify them in three different lines: the Boussinesq models, the Cahn-Hilliard model, and the $\phi^4$ and similar models.

\medskip

The fourth order $\phi^4$ model \eqref{eq:WCH} is part of the family of Boussinesq \cite{Bou1} models highly studied in the literature. One important aspect that is not completely understood is the behavior of solitary waves in the long time dynamics. For a complete review on this model in the ``focusing case'', the reader may consult the introduction in \cite{M_GB}. The literature is extensive and we shall concentrate ourselves in the soliton problem. Bona and Sachs \cite{BS} showed the stability of fast solitary waves for the so-called good Boussinesq model. Slow solitary waves are unstable and may develop blow up \cite{Liu}. Pego and Weinstein \cite{PW,PW2} addressed for the first time the asymptotic stability problem in the case of the so-called Improved Boussinesq model (see also \cite{MaMu}) revealing its difficulty compared with other fluid models. Precisely, compared with \eqref{eq:WCH}, the former has strongly unstable directions, specially in the case of the standing wave which was proved asymptotically stable in \cite{M_GB}, provided one works orthogonal to the unstable manifold. Our proofs follow in spirit the results in \cite{M_GB} (see also a previous work in the case of the Improved Boussinesq system \cite{MaMu}), however, unlike in previous works were parity was needed, here we are able to consider the problem in full generality. Previous decay results are available for the Bona-Chen-Saut $abcd$ model \cite{BCS1,BCS2}, see the references \cite{KMPP} and \cite{KM}.

\medskip

The long time behavior of kinks (or fronts) in the parabolic Cahn-Hilliard model \cite{CH1958} has been addressed by many authors during the past decades. Novick-Cohen and Segel \cite{NCS1984} and Novick-Cohen \cite{NC1985} provided foundational energy methods to describe the dynamics in the case of the originally motivated strongly degenerate Cahn-Hilliard model. In \cite{ES}, local and global well-posedness and long time behavior for Cahn-Hilliard on an interval were first obtained. Pego \cite{Pego} used matched asymptotics to describe the evolution in higher dimensions of the phase separation in the singular perturbative regime. Later, Cafferelli and Muler \cite{CM} provided rigorous $L^\infty$ bounds for this regime. Alikakos, Bates and Chen \cite{ABC} showed that level surfaces of solutions to the Cahn-Hilliard equation tend to solutions of the Hele-Shaw problem under the assumption that classical solutions of the latter exist.  As previously mentioned, Bricmont, Kupiainen and Taskinen \cite{BKT} provided a foundational result proving that the kink in 1D is asymptotically stable: if the perturbation is continuos and for some $p>2$,  $\|\langle x \rangle^p (\phi-H)\|_{L^\infty} $ is small enough, then
for some $x_0\in\R$, 
\[
\lim_{t\to +\infty}\|\phi(t) - H(\cdot -x_0)\|_{L^\infty} =0.
\]
A more precise asymptotics of the reminder term is also given. The method of proof involves but it is not limited to the renormalization group technique. It is interesting to mention that in the energy space, the convergence of the shift parameter in Theorem \ref{thm:asymptotic} towards a final state is probably not possible unless one adds additional information on the initial data of the problem. The reader may also consult a simplified proof \cite{CCO2} of the kink asymptotic stability by using free energy techniques.   
Later, Korvola, Kupiainen and Taskinen \cite{KKT} (see also \cite{Kor}) extended the 1D result to dimension $d\geq 3$ and any $p>d+1$. There is an interesting anomalous decay for the problem in higher dimensions, and the shift must take into account the transversal perturbations, however it converges to zero as time tends to infinity, unlike in the 1D case.  
\medskip

In the case of a generalized version of the 1D-Cahn-Hilliard equation, Howard \cite{How_noplanar_07} established that linear stability of fronts implies nonlinear stability. Nonlinear orbital stability is established for waves with initial perturbations of algebraic decay, under the spectral stability assumption, described in terms of the Evans function. Later, Howard \cite{Ho} established that the planar wave solutions are asymptotic stable, in the d-dimensional Cahn–Hilliard equation, with $d\geq 2$. In this result, it is required that the initial perturbations decay at an appropriate algebraic rate in an $L^1$ norm of the transverse variables; and in \cite{How}, the same author considers the multidimensional Cahn-Hilliard system showing for the planar transition front solutions
that spectral stability implies nonlinear stability.  
\medskip

Finally, Theorem \ref{thm:asymptotic} can be recast as an extension into the fourth order $\phi^4$ model of the stable \cite{HPW} kink asymptotic stability proved in \cite{KMM2017} using essentially virial identities (see \cite{KM2022} for a simplified proof and \cite{CucMae} for an extension of the previous result). As previously explained, the data considered in \cite{KMM2017} is odd and the general case is still open.  Earlier results in this direction were obtained by Cuccagna \cite{Cuc}, who considered the $\phi^4$ kink in 3D, and using vector field methods showed the asymptotic stability of the $\phi^4$ kink. Under higher order weighted norms, Komech and Kopylova \cite{KK1,KK2} showed asymptotic stability of kinks of highly degenerate scalar field theories. Delort and Masmoudi \cite{DM} applied Fourier analysis techniques and proved a detailed asymptotics for odd perturbations of the kink up to times $O(\varepsilon^{-4})$, where $\varepsilon$ is the size of the perturbation. In \cite{KMMV}, a sufficient condition is given to describe the long time dynamics of kink perturbations for any data in the energy space, including many models of interest in Quantum Field Theory \cite{Lohe} (except sine-Gordon and $\phi^4$). In this case kinks must be modulated in terms of scaling (the Lorentz boost) and shifts, which makes computations harder than usual. Cuccagna and Maeda provided a new sufficient condition to have asymptotic stability \cite{CucMae} in the odd data case. Snelson \cite{Snelson} considered the case of the $\phi^4$ kink in the presence of variable coefficients. The case of the integrable sine-Gordon kink has attracted attention during past years due to its complexity and lack of kink asymptotic stability in the energy space, see \cite{MP,AMP,CLL,LuS1} and references therein. The case of collision of kink structures was treated in \cite{JKL}.

\medskip

Another point of view, equivalent to the treatment of kinks under symmetry assumptions (essentially no shifts nor Lorentz boosts) is given by the study of 1D nonlinear Klein-Gordon models under variable coefficients. Foundational works in 3D were obtained by Soffer and Weinstein \cite{SW1,SW2,SW3}. In this direction we mention previous scattering results by Lindblad and Soffer \cite{LS1,LS2,LS3}, Hayashi and Naumkin \cite{HN2,HN3,HN4}, Sternbenz \cite{Ste}, Bambusi and Cuccagna \cite{Bam_Cucc}, Lindblad and Tao \cite{Lin_Tao}, and Lindblad et. al. \cite{LLS,LLS2,LLSS}. These results have been recently improved by considering quadratic nonlinearities, see the scattering results by Germain and Pusateri \cite{GP}, see also \cite{GPZ}. The dynamics of solitons in nonlinear Klein Gordon models has concentrated much efforts during past years. We mention 3D and 1D works on the description of the manifold manifold by Krieger-Nakanishi-Schlag \cite{KNS} and Nakanishi-Schlag \cite{NS}, and earlier results by Ibrahim, Masmoudi and Nakanishi  \cite{IMN}. Recently, 1D and 3D subcritical dynamics around the soliton has been addressed in great detail in \cite{BCS,KMM,KMM3b,GJ1,GJ2,LL,KP,LP,LP2,LuS2} in the presence of at least one unstable mode.

\subsection{Idea of proofs} We shall use localized virial estimates to show the asymptotic stability of the fourth order $\phi^4$ kink. Virials have been previously used in many complex dynamics, see e.g. the works \cite{ACKM,AM,KMM,KMM2017,Martel-Merle1,MM_solitonsKdV}. In this paper, we follow ideas from \cite{M_GB}, which are based in previous ones from Kowalczyk, Martel and the second author in \cite{KMM} and Kowalczyk, Martel, the second author and Van Den Bosch \cite{KMMV} to study the stability properties of kinks for (1+1)-dimensional nonlinear scalar field theories. 

\medskip

The first step is to decompose the solution close to the kink as follows: we choose $\rho(t)$ such that 
\begin{equation*}
\begin{cases}
\phi_1(t,x)= H(y)+u_{1}(t,x), \quad y= x-\rho(t),\\
\phi_2(t,x)= u_{2}(t,x),
\end{cases}
\end{equation*}
and $\langle u_1 ,H'(y)\rangle =0$ and $\|(u_1(t),u_2(t))\|_{H^1\times L^2} \leq \delta$. Notice that we have chosen not to follow the standard centering performed in \cite{KMMV}. There are several reasons to follow a different approach. The most relevant is that multi-kink structures do not center well, specially in the case of several kinks. Another reason is that centering and multiple derivatives as in \eqref{eq:WCH} do not cooperate well each other. 

\medskip

Then, we will focus on $(u_1,u_2)\in H^1\times L^2$, which satisfy the linearized equation \eqref{eq:eq_lin}. Following \cite{MPP,M_GB}, for an adequate weight function $\varphi_A$ placed at scale $A$ large, we obtain the virial estimate (see \eqref{eq:dI_w})
\begin{equation}\label{eq:I_intro}
\begin{aligned}
\dfrac{d}{dt}\int \varphi_{A}(y) u_1 u_2
\leq&-\dfrac{1}{2} \int \left[ w_2^2 +2(\partial_x w_1)^2
+\left(2- C_1\delta \right)w_1^2 \right] \\&  +  C_1\int \sech\left(y \right) u_1^2 +C_1\rho'^2,
\end{aligned}
\end{equation}
where $(w_1,w_2)$ is localized version of $(u_1,u_2)$ at $A$ scale, and $C_1$ denotes a fixed constant.  This virial estimate is standard now and has no good sign because of the term $C_1\int \sech\left(y \right) u_1^2$ and $\rho'^2$ (which is only of quadratic nature). Then we require to transform the system to a new one which has better virial estimates, in the spirit of  Martel \cite{Martel_linearKDV}. For any $\gamma>0$ small enough, we  define new variables $(v_1,v_2)\in H^{1}\times H^{2}$ by
\begin{equation*}
	\begin{cases}
	v_1= (1-\gamma \partial_x^2)^{-1}\LL  u_1,\\
	v_2= (1-\gamma \partial_x^2)^{-1}  u_2.
	\end{cases}
\end{equation*}
(see \eqref{eq:change_variable}). In \cite{M_GB}, this change of variable was enough to describe the stable manifold related to the unstable static soliton. Here we have additional complications since shifts are nontrivial perturbations and $(v_1,v_2)\in H^1\times H^2$ follow modified dynamics. In particular, $|\rho'|$ is only linear in $(u_1,u_2)$, unlike in scalar field models. However, a surprising miracle happens and the new system for $(v_1,v_2)$ (see \eqref{eq:syst_v}) satisfies, for an adequate weight function $\psi_{A,B}$, $B\ll A$, the new virial estimate (see \eqref{eq:dJ})
	\begin{equation}\label{eq:J_intro}
	\begin{aligned}
	\dfrac{d}{dt}\int \psi_{A,B} v_1 v_2 \leq & -\frac{1}{2} \int \left[z_1^2+(V_0-C_2 \delta^{1/10}) z_2^2 +2(\partial_x z_2)^2 \right] + \hbox{l.o.t.},
	\end{aligned}
	\end{equation}
where $(z_1,z_2)$ is a lozalized version of $(v_1,v_2)$,  at the smaller scale $B$, $V_0$ given by \eqref{eq:LL}, and $C_2$ denotes a fixed constant. An important point here is that the operator $-2\partial_y^2 +V_0$ is now positive, and $\rho'$ has a critical almost zero contribution in \eqref{eq:J_intro}, which is not the case in \eqref{eq:I_intro}. The cancelation of the contribution by $\rho'$ is part of a new idea that reveals that shift and scaling modulations should be treated as if they were additional internal modes (recall that they are not), just as it is done in scaler field models. This is the main new outcome of this paper, that allows us to treat shifts in a simple fashion in the case of fourth order $\phi^4$ kink perturbations.

\medskip

Following \cite{KMM}, in order to combine estimates \eqref{eq:I_intro} and \eqref{eq:J_intro} we need an estimate for the last terms in \eqref{eq:I_intro}. A coercive estimate is proved in terms of the variables $(w_1,w_2)$ and $(z_1,z_2)$ (see \eqref{eq:w_sech}):
	\begin{equation}\label{eq:sech_u1_intro}
	\begin{aligned}
	\int \sech (y)  u_1^2
	\lesssim & ~{}  \delta^{1/20} \left( \| w_1\|_{L^2}^2 +\|\partial_x w_1\|_{L^2}^2 \right)
		+ \delta^{-1/20} \|z_1\|_{L^2}^2 +\delta^{2/5}\| \partial_x z_1\|_{L^2}^2.
	\end{aligned}
	\end{equation}
We can directly observe that the term $\partial_x z_1$ does not appear in \eqref{eq:J_intro}, leading to the main obstruction present in this paper. This problem is deeply related to the fact that $(v_1,v_2)\in H^1\times H^2$, i.e., the new variables are in opposed order of regularity. 

\medskip

In order to overcome this problem, we introduce a series of modifications that will allow us to close estimates \eqref{eq:I_intro} and \eqref{eq:J_intro} properly. First, we must gain derivates. In a new virial estimate for the system of $(\partial_x v_1,\partial_x v_2)$ (see \eqref{eq:syst_vx}), we obtain the third virial estimate
		\begin{equation}\label{eq:M_intro}	
		\begin{aligned}
	 \frac{d}{dt} \int \psi_{A,B} \partial_x v_1 \partial_x v_2 
	 \leq & ~{}  
	-\dfrac{1}{2}\int \left( (\partial_x z_1)^2+\left(V_0 -  C_3\delta^{1/10}\right) (\partial_x z_2)^2+2(\partial_x^2 z_2)^2\right)+ \hbox{l.o.t.},
	\end{aligned}
	\end{equation}
with $C_3>0$ fixed.	This new estimate give us local $L^2$ control on $\partial_x z_1$ and $\partial_x^2 z_2$, which was not present before. Finally, our last contribution is a  transfer virial estimate that exchanges information between  $\partial_x z_1$, $\partial_x z_2$ and $\partial_x^2 z_2$, in the form of
		\begin{equation}\label{eq:Dintro}
		\begin{aligned}
		\frac12	\int (\partial_x z_1)^2 \le &~{} \frac{d}{dt} \int \rho_{A,B} \partial_x v_1 v_2
		+C_4 \int \left[ (\partial_x^2 z_2)^2  + (\partial_x z_2)^2 + z_2^2+z_1^2\right] + \hbox{l.o.t.}
		\end{aligned}
		\end{equation}
Here $C_4>0$ is fixed and $\rho_{A,B}$ is a suitable weight function. Finally, we consider a functional $\mathcal{H}$ being a well-chosen linear combination of \eqref{eq:I_intro}, \eqref{eq:J_intro}, \eqref{eq:M_intro}, \eqref{eq:sech_u1_intro} and \eqref{eq:Dintro}. We get
	\begin{equation*}
	\begin{aligned}
	\dfrac{d}{dt}\mathcal{H}(t)
			\leq &		
	- \delta^{1/10} \left(\|w_1\|_{L^2}^2+\|\partial_x w_1\|_{L^2}^2 
	+\| w_2\|_{L^2}^2  \right), \ \ \mbox{for all } t\geq 0.
	\end{aligned}
	\end{equation*}
This final estimate allows us to close estimates, and prove local decay for $u_1$ after some standard change of variables from $w_j$ to $u_j$.

\subsection*{Organization of this paper} This paper is organized as follows. In Section \ref{sec:pre} we provide several basic but not less important elements for the study of the fourth order $\phi^4$ model. Section \ref{sec:OS} is devoted to the proof of Theorem \ref{thm:orbital}. Finally, Sections \ref{sec:2}, \ref{sec:3}, \ref{sec:4}, \ref{sec:5}, \ref{sec:5bis} and \ref{sec:6} deal with the proof of the asymptotic stability of the kink, Theorem \ref{thm:asymptotic}.

\subsection*{Acknowledgments} Ch. Ma. thanks Sebastian Herr and Lars Eric Hientzsch for support and interesting comments and suggestions. Both authors thank to Banff Center (Canada), where part of this work was done.

\section{Preliminaries}\label{sec:pre}

\subsection{Linear operators} Recall $\mathcal L$ introduced in \eqref{eq:LL}. The following results are standard, see \cite{Lohe}.

\begin{lem}[Properties of $\LL$]\label{PropL}
The linear unbounded operator $\LL$, defined in $L^2$ with domain $H^2$, satisfies the following properties:
\begin{enumerate}
\item The absolutely continuous spectrum of $\LL$ is $[2,\infty)$.
\item  $\LL$ is self-adjoint and nonnegative.
\item The discrete spectrum consists of simple eigenvalues $\lambda_0=0$ and $\lambda_1=\frac{\sqrt{3}}{2}$, with eigenfunctions
\[
Y_0=H' \quad \mbox{ and }  \quad Y_1(x)= \sech\left(\frac{x}{\sqrt{2}}\right) \tanh\left(\frac{x}{\sqrt{2}}\right),
\]
respectively.
\end{enumerate}
\end{lem}
In a similar way, we have the following properties for the more involved operator $-\PL:$

\begin{lem}[Properties of $-\partial_x^2 \LL$]\label{lem:PLL}
The linear operator $-\PL$ defined in \eqref{eq:LL}, posed in $L^2$ with domain $H^4$, satisfies the following properties:
\begin{enumerate}
\item The absolutely continuous spectrum of $-\partial_x^2 \LL$ is $[0,\infty)$.
\item  $\LL$ is nonnegative.
\item $\hbox{ker} (-\partial_x^2\LL) =\hbox{span} \{ H' \}$.
\end{enumerate}
\end{lem}
The proof of the first two statements is direct. For the proof of the last result, see Appendix \ref{app:Null_spaces}.

\begin{rem}
Unlike good Boussinesq, the operator $\mathcal L$ appearing in the case of kinks has even kernel $H'$, implying that the equation $\mathcal L A= 1$ cannot have bounded solutions. Resonances are there excluded in this case. See \cite{M_GB} for the case where resonances but no shifts are allowed.
\end{rem}

\medskip
\subsection{Coercivity} 
The linearization of \eqref{eq:WCH_syst} around $\boldsymbol{H}$ involves the operator \eqref{eq:LL} and $-\partial_x^2\LL$. Here, we recall a few properties of the operator $\LL$.

Let the bilinear form
	\[
	H(u,v)=\Jap{\LL (u)}{v}=\int (\partial_x u \partial_x v+V_0uv) .
	\]
where $V_{0}$ is given in \eqref{eq:LL}.

\begin{lem}[Coercivity \cite{KMM2017}]\label{lem:coercivity}
	If $u\in H^1(\R)$ satisfies $\Jap{u}{Y_0}=0$, where $Y_0$ is the even eigenfunction associated to eigenvalue $\lambda_0=0$ of operator $\LL$, then 
	\begin{equation}\label{eq:coercivity_V}
H(u,u)\geq \frac37 \|u\|_{H^1}^2.
	\end{equation}
\end{lem}

We will need a weighted version of the previous result, which uses \eqref{eq:coercivity_V}. See \cite{M_GB} for a similar proof. 
\begin{lem}[Coercivity with weight function]\label{lem:coerc_weight}
	Consider the bilinear form
	\[
	H_{\phi_L} (u,v)=\Jap{\sqrt{\phi_L} \LL(u)}{\sqrt{\phi_L} v}=\int \phi_L (\partial_x u \partial_x v+V_0uv) .
	\]
	for $\phi_L$ such that $|\phi_L'|\leq CL \phi_L$, with $C$ not depending on $L$.
	Then, there exists $\lambda>0$ independent of $L$ small such that
	\[
	H_{\phi_L} (u,u)\geq \lambda \int \phi_L ((\partial_x u)^2+u^2) ,
	\]
	for all $u \in H^1(\R)$ satisfying $\Jap{u}{Y_0}=0$, and provided $L$ is taken small enough, independent of the size of $u$.
\end{lem}
\subsection{Technical identities related to the operator $\LL$}

The next lemma will be useful on the first computations related to virials identities, which allow us to prove the asymptotic stability of the static kink.

\begin{lem}\label{lem:LL}
Let $\eta$ be a smooth bounded function, and $(f,g)\in H^1(\R)\times  H^{1}(\R)$. The following identities hold
\begin{align}
\Jap{\eta \LL(f)}{g}=&~{} \int\eta [\partial_x f \partial_x g+V_0 fg]+\Jap{\eta' \partial_x f}{g},\label{eq:L}\\
\Jap{\eta \partial_x\LL(f)}{f}=&~{} -\frac12  \int\eta' [ 3 (\partial_x f)^2 +V_0 f^2 ]+\frac12 \int \eta V_0'  f^2+\frac12 \int \eta''' f^2,\label{eq:PL}
\end{align}
and
\begin{align}
\Jap{\eta \LL(\partial_x f)}{f}=& -\frac12  \int\eta' [ 3 (\partial_x f)^2 +V_0 f^2 ]-\frac12 \int \eta V_0'  f^2+\frac12 \int \eta''' f^2, \label{eq:LP}\\
\Jap{\eta \partial_x \LL(\partial_x f)}{f}=&-\int\eta [(\partial_x^2 f)^2 +V_0 (\partial_x f)^2]+2 \int \eta'' (\partial_x f)^2  \label{eq:PLP}
\\
&+\frac12  \int\eta''  V_0 f^2 +\frac12 \int \eta' V_0'  f^2-\frac12 \int \eta^{(4)} f^2.\nonumber
\end{align}
\end{lem}
For the proof see Appendix \ref{app:LL_idnt}.

\subsection{Ill-posedness}\label{IWP}

Before considering perturbations of the kink solution, we will make some remarks about the case of small data around the zero solution. It is expected, as in the case of $\phi^4$ or Allen-Cahn, that this state is unstable. However, some interesting remarks about local existence can be made in the fourth order $\phi^4$ case.

System \eqref{eq:WCH} can be written as
		\[
		\left\{
		\begin{aligned}
		\partial_{t}\phi +\partial_{x}^{2} \varphi=&~{}0\\
		\partial_{t} \varphi-\partial_{x}^{2}\phi-\phi=&-\phi^3.
		\end{aligned}
		\right.
		\]
		Let $\psi=\phi+i\varphi$, then
		\[
		\begin{aligned}
		i\partial_{t} \psi+\partial_{x}^{2}\psi+\mbox{Re}~{} \psi=(\mbox{Re}~{} \psi)^3.
		\end{aligned}
		\]

From this we can see that $w$ has interesting similarities with the equations appearing in the study of the Peregrine breather \cite{Mun_insNLS}. Indeed, the linear fourth order $\phi^4$ equation has the same instability issues as NLS around a nonzero background. From a standard frequency analysis, we get for a formal standing wave  $\phi=e^{i(kx-wt)}$ solution to the linear \eqref{eq:WCH}, one has
		\[
		w(k)=\pm |k|\sqrt{k^{2}-1},
		\]
		which reveals that for small wave numbers $(|k|<1)$ the linear equation behaves in an ``elliptic'' setting, and exponentially in time  growing modes are present from small perturbation of the vacuum solution. Unlike \cite{Mun_insNLS}, where one still has well-posedness in $H^s$, $s>\frac12$ in the nonlinear case, that approach fails here because of the two derivatives in the nonlinearity.

\subsection{Local well-posedness in a neighborhood of the kink}
Contrary to Subsection \ref{IWP}, perturbations of heteroclinic kinks do not suffer of the lack of well-posedness. To prove well-posedness around the kink solution we will adapt the well-posedness proof by Linares \cite{Linares_93} (see also \cite{Notes_linares}). Most of the concern is given by how the linear flow can recover the two derivatives present in the nonlinearity. Let $\delta>0$ small enough to be chosen later. We consider an initial data $\boldsymbol{u}^{in}\in \bd{E_{H}}$ (see \eqref{eq:EH}) such that \eqref{eq:Initial_data_a_kink0} is satisfied.

\medskip

In this section, we are looking for a solution $\bd{\phi}(t)$ 
of \eqref{eq:WCH_syst} in $\bd{E_H}$ for all time with initial data $\bd{\phi}^{in}$. Now, we will focus on the local well-posedness of the perturbed system around the static kink for the equation \eqref{eq:WCH_syst}. Here, it is remarkable that around the static kink the nature of the solution change drastically, here it is not present the exponential growth as in above section. 

We consider an initial data which is perturbation around the static kink, i.e., an initial data of system  \eqref{eq:WCH_syst} which has the form
\begin{equation}\label{eq:phi^in}
\begin{aligned}
\bd{\phi^{in}}=(H(x)+u_1^0(x), u_2^0(x)), 
\end{aligned}
\end{equation}
and setting $\bd{u}(t,x)=\bd{\phi}(t)-\HH$, we reduce our problem to solve the system 
\begin{equation}\label{eq:CP+1}
\begin{aligned}
		 \partial_t u_1&= \partial_ x u_2\\
		 \partial_t u_2 &= \partial_x (-\partial^2_x u_1+2u_1 +F(x,t,u_1)),
\end{aligned}
\end{equation}
in $H^1\times L^2$, where  
\begin{equation}\label{nueva_F}
F(t,x,u_1)=3u_1(H^2-1)+u_1^2(u_1+3H).
\end{equation} 
Noticing that $F$ is locally Lipschitz in the third variable, i.e. there exists $C>0$, such that for any $v,w$, if $\|u\|_{L^{\infty}}\leq 1$, $\|w\|_{L^{\infty}}\leq 1$, then
\begin{equation}\label{eq:Fu-Fv}
|F(t,x,u)-F(t,x,w)|
= |u-w||3(H^2-1)+u^2+(u+w)(w+3H)|
\leq C |u-w|.
\end{equation}
The fourth order $\phi^4$'s linearization around the static kink (see \eqref{eq:CP+1}) hides a slightly modified Good-Boussinesq equation's structure. That means that the linear part of \eqref{eq:CP+1} is essentially the same as the linear Good-Boussinesq equation. We can see that by Duhamel's principle, the associated solution is represented as
\[
u_1(t,x)= \mathcal{G}(t) u_1^0(x)+\mathcal{K}(t) u_2^0(x) +\int_{0}^t \mathcal{K}(t-s) \partial_x^2 F(x,s,u_1)ds,
\]
where
\[
\mathcal{G}(t)= \mathcal{F}^{-1} G(t,\xi) \mathcal{F}, \quad \mathcal{K}(t)= \mathcal{F}^{-1} K(t,\xi) \mathcal{F};
\]
  $\mathcal{F}$ and  $\mathcal{F}^{-1}$ represents the Fourier transform and its inverse, respectively. The Fourier multipliers are given by
\[
\begin{aligned}
G(t,\xi)=  \cos(\omega(\xi)t), 
\quad 
K(t,\xi)= \frac{\sin(\omega(\xi) t)}{\omega(\xi)}, 
\end{aligned}
\]
where $\omega(\xi)=|\xi|\sqrt{\xi^2+2}$. Then, following Linares' work \cite{Notes_linares,Linares_93} and noticing that  $F$ in \eqref{nueva_F} is locally Lipschitz and satisfies
\begin{equation}\label{eq:F_bound}
|F| \lesssim  |u_1| e^{-\sqrt{2}|x-\rho|} + u_1^2 + |u_1|^3,
\end{equation}
we conclude the equation \eqref{eq:CP+1} is locally well-posed in $H^1\times L^2(\R)$. Using the conserved quantities \eqref{eq:energy_momentum}, we obtain the global well-posedness of \eqref{eq:CP+1}.  To see the complete proof, the reader may consult Appendix \ref{app:WP_kink}.

\medskip
In this paper, we will only need the above notion of solution $\bd{\phi}=(\phi_1,\phi_2)$ of \eqref{eq:WCH_syst}. Now, we are ready to face the orbital stability of the static kink.

\section{Orbital stability}\label{sec:OS}

First, we will establish the energy and momentum of the static kink, as well as their perturbations.

\begin{lem}\label{lem:perturbation_energy_momentum}
Let $E$ and $P$ be defined as in \eqref{eq:energy_momentum}. The following are satisfied:
\begin{enumerate}
\item Conservation laws for the static kink
\begin{equation}\label{EP0}
E[\bd{H}]= \|H'\|_{L^2}^2
,\quad \quad 
P[\bd{H}]=0.
\end{equation}
\item Additionally,
\begin{equation}\label{EP1}
\begin{aligned}
E[\bd{H+u}]=&E[\bd{H}]+\frac12 \left[\|u_2\|^2_{L^2}+\Jap{\LL u_1}{u_1}\right]+R_{u_1}
\end{aligned}
\end{equation}
\end{enumerate}
\end{lem}
where $|R_{u_1}|\leq C \|u_1\|_{L^\infty} \|u_1\|_{L^2}^2$.
\begin{proof}[Proof of Lemma \ref{lem:perturbation_energy_momentum}]

\emph{Proof of \eqref{EP0}}. Recalling that $\HH=(H,0)$, one immediately obtains that $P[\HH]=0$. Furthermore, from \eqref{eq:H_c} for $c=0$, we have $(H')^2=(1-H^2)^2/2$ , concluding that $E[\HH]= \|H'\|_{L^2}^2$.

\medskip
\noindent 
\emph{Proof of \eqref{EP1}}. Expanding the energy, we get
\begin{equation*}
\begin{aligned}
E[\HH+\bd{u}]
=&~{} \frac12 \int \left( (\partial_x u_1)^2+u_2^2+ (\partial_x H)^2+2 H' \partial_x u_1  \right)
\\&
+\frac14  \int [u_1^2(u_1+2 H)^2+2u_1(u_1+2 H)(H^2-1)+(H^2-1)^2]\\
=&~{} \frac12 \int (\partial_x u_1)^2+u_2^2+ (\partial_x H)^2-\int  H'' u_1 
\\&
+\frac14 \int [u_1^2(u_1+2 H)^2+2u_1(u_1+2 H)(H^2-1)+(H^2-1)^2].
\end{aligned}
\end{equation*}
Using \eqref{eq:H_c}, we obtain 
\begin{equation*}
\begin{aligned}
E[\HH+\bd{u}]
=&~{} \frac12 \int (\partial_x u_1)^2+u_2^2+(-1+3H^2)u_1^2+ (\partial_x H)^2+\frac12 (H^2-1)^2
\\&
+\frac14\int  [u_1^2(u_1+2 H)^2+2u_1(u_1+2 H)(H^2-1)]+4u_1  H(1-H^2)+2(1-3H^2)u_1^2\\
=&~{} E[\HH]+\frac12 \left( \Jap{\LL u_1}{u_1}+\|u_2\|_{L^2}^2 \right)
+\frac14\int u_1^3  [u_1+4 H].
\end{aligned}
\end{equation*}
Letting $R_{u_1}=  \displaystyle{\frac14 \int} u_1^3  [u_1+4 H]$, we quickly obtain that $|R_{u_1}|\leq C \|u_1\|_{L^\infty} \|u_1\|_{L^2}^2$. This concludes the proof of lemma.
\end{proof}

\begin{proof}[Proof Theorem \ref{thm:orbital}]
Let $\bd{\phi}$ the local in time solution of \eqref{eq:WCH_syst} with initial data $\bd{\phi^{in}}$ given in \eqref{eq:phi^in}, satisfying \eqref{eq:Initial_data_a_kink0}.  For $C>1$ to be chosen later,  define
\[
T^{*}=
\sup \left\{
t\geq0 ~ \Big\vert ~ \bd{\phi} \mbox{ is well-defined on }[0,t] \mbox{ and } \sup_{s\in [0,t]} \inf_{\rho\in\R}\|\bd{\phi}(s)-\HH (\cdot -\rho)\|_{H^1\times L^2} \leq C^{*}\delta .
\right\}
\]
By continuity and \eqref{eq:Initial_data_a_kink0}, we get $T^{*}$ is well defined and $T^{*}>0$. Then, if $T^{*}<\infty$, using a continuity argument and the smallness of the initial data (see \eqref{eq:Initial_data_a_kink0}), $\bd{\phi}$ is well defined on $[0,T^1]$ for some $T^1>T^{*}$ and it would hold 
\begin{equation}\label{eq:time_T*}
\inf_{\rho\in\R}\|\bd{\phi}(T^{*})-\HH (\cdot -\rho) \|_{H^1\times L^2}=C^{*}\delta.
\end{equation}
Assuming that $T^{*}$ is finite, we work on the interval $[0,T^{*}]$. Considering that $\bd{\phi}$ has the form $\HH+\bd{u}$, we get from \eqref{EP1},
\[
E[\HH+\bd{u}](t)-E[\HH] = \frac12 \left[\|u_2\|^2_{L^2}+\Jap{\LL u_1}{u_1}\right]+R_{u_1},
\]
so that using the conservation of energy (Lemma \ref{lem:perturbation_energy_momentum}) and \eqref{eq:time_T*},
\[
 \frac12 \left[\|u_2\|^2_{L^2}+\Jap{\LL u_1}{u_1}\right] =  \frac12 \left[\|u_2(0)\|^2_{L^2}+\Jap{\LL u_1(0)}{u_1(0)}\right]+R_{u_1(0)}-R_{u_1} \leq C\delta^2 + C(C^*)^3 \delta^3.
\]
The coercivity of the operator $\LL$ (Lemma \ref{lem:coercivity} and \eqref{rem_data}) gives that for some fixed $\mu,C>0$,
\[
\mu \|\bd{u} \|_{H^1\times L^2}^2  \leq C\delta^2 + C(C^*)^3 \delta^3.
\]
Thus, for all $t\in [0,T^{*}]$, if $\delta$ is small enough,
\begin{equation}\label{eq:orbital}
\|\bd{u}\|_{H^1\times L^2}^2 \leq \frac{2C}{\mu} \delta^2,
\end{equation}
here $C$ and $\mu$ are independent of $C^{*}$. Fixing $(C^{*})^2>C/\mu$  we get a contradiction on \eqref{eq:time_T*}. Therefore, $\bd{\phi}$ is a global solution for $t\geq0$, $T^{*}$ and \eqref{eq:orbital} hold for any $t\geq0$. This conclude the proof of Theorem \ref{thm:orbital}.
\end{proof}

\section{Asymptotic stability: first estimates}\label{sec:2}

Consider a small perturbation of kink  solution $\HH=(H,0)$. In what follows we will describe this decomposition, introduce some notation, and develop a virial estimate for the fourth order $\phi^4$ system \eqref{eq:WCH_syst}.

\subsection{Decomposition of the solution in a vicinity of the kink}

Let $\bd{\phi}=(\phi,\partial_t\partial_{x}^{-1}\phi)=(\phi_1,\phi_2)$ be a solution of \eqref{eq:WCH_syst} satisfying 
\begin{equation}\label{eq:ineq_hip}
\Vert \bd{\phi}(t,\cdot+\rho(t))-\HH\Vert_{H^1\times L^2}=\|\bd{u}(t)\|_{H^1\times L^2} \leq C_0 \delta,
\end{equation}
 for all $t\geq0$, where $\delta$ is defined from the initial data $\bd{\phi}^{in}$  \eqref{eq:phi^in}, to be taken small enough. Recall the decomposition for $\bd{\phi}$
\begin{equation}\label{eq:decomposition}
\begin{aligned}
& \bd{\phi}(t,x)=\HH(x-\rho(t))+\bd{u}(t,x), \quad
\mbox{where }
  \bd{u}=(u_1(t,x),u_2(t,x)) 
 \\&\mbox{ satisfies } \quad \langle \bd{u}(t) , {\bd H'}(\cdot-\rho(t)) \rangle = \langle u_1(t),H' (\cdot-\rho(t)) \rangle =0. 
\end{aligned}
\end{equation}
Note that \eqref{eq:ineq_hip} and Theorem \ref{thm:orbital} reads now
\begin{equation}\label{new cota}
\|u_1(t)\|_{H^1} + \|u_2(t)\|_{L^2} \leq C_0 \delta.
\end{equation}
Additionally, from \eqref{eq:eq_lin} $\bd u$ satisfies the following equation
\begin{equation}\label{eq:u_vector}
\partial_t {\bd u}=\partial_x {\bd J }{\bd L}{\bd u}+ \partial_x{\bd F}(x,\bd{u}) +\rho' \HH',
\end{equation}
where
\begin{equation*}
{\bd J}=
\left(\begin{matrix}
0& &1\\1& & 0
\end{matrix}\right),
\quad \quad
{\bd L}=
\left(\begin{matrix}
 \LL & & 0\\0& &1
\end{matrix}\right),
\quad \mbox{ and }
{\bd F} (x,{\bd u})=
\left(\begin{matrix}
0 \\ u_1^3+3H (\cdot -\rho) u_1^2
\end{matrix}\right),
\end{equation*}
with $\LL$ given by \eqref{eq:LL}. The equation \eqref{eq:u_vector} is equivalent to system
\begin{equation}\label{eq:u_linear}
	\begin{cases} 
	\dot{u}_1= \partial_x u_2  +\rho' H' \\
	\dot{u}_2=\partial_x \LL u_1+\partial_x(u_1^3+3H u_1^2).
	\end{cases}
\end{equation}
Notice that from \eqref{eq:decomposition} one has
\[
 \langle u_1,H' \rangle =0 \implies  \langle \dot u_1,H' \rangle = \rho'  \langle u_1,H'' \rangle, 
\]
and from \eqref{eq:u_linear}
\begin{equation}\label{eq:rho_p}
 \langle \dot{u}_1,H' \rangle = -\langle u_2, H'' \rangle  +\rho' \| H' \|^2 \implies |\rho'| \lesssim \left| \langle u_2, H'' \rangle\right| \lesssim \left( \int e^{-\sqrt{2}|x-\rho|}u_2^2\right)^{1/2}.
\end{equation}

\subsection{ Weighted functions} 
In this section we recall all the necessary auxiliary results that will be needed in forthcoming sections. The notation is taken essentially from \cite{KMMV}, with some particular choices from \cite{M_GB}. We start by describing the weighted functions used to define our local norms.

We consider a smooth even function $\chi:\R\to \R$ satisfying
\begin{equation}\label{chichi}
\chi=1 \mbox{ on } [-1,1], \quad \chi=0 \mbox{ on } (-\infty,2]\cup [2,\infty), \quad \chi'\leq 0 \mbox{ on } [0,\infty).
\end{equation}
For any scale $K>0$, we define the functions $\zeta_K$ and $\varphi_K$ as follows
\begin{equation}\label{eq:bound_phiA} 
\zeta_K(x)=\exp\left( -\dfrac{1}{K}(1-\chi(x))|x|\right),
\quad \varphi_K(x)=\int_{0}^{x} \zeta_K^2(y)dy, \ \ x\in \R.
\end{equation}
We consider the function $\psi_{A,B}$ defined as
\begin{equation}\label{eq:psi_chiA}
\psi_{A,B}(x)=\chi^2_A(x)\varphi_B(x) \mbox{ where }\ \ \chi_A(x)=\chi\left(\dfrac{x}{A}\right), \ \  x \in \R.
\end{equation}
We recall that $\psi_{A,B}$ and $\varphi_{K}$ are {\it odd functions}. 
\medskip

These functions will be used several times in arguments of virial type, and with different scales, $A$ and $B$, under the following constraint:
\begin{equation}\label{eq:scales}
	1\ll B \ll B^{8}  \ll A.
\end{equation}

\noindent
{\bf Basic  set of technical estimates  related to weighted functions}. 
 The following technical estimates on functions $\zeta_K$ and $\chi_{A}$ will be useful troughthout all this work.  These estimates has been used on similar context   (for the proofs, see \cite{KMM2017,M_GB}).

\begin{lem}
	Let $\zeta_K$ and $\chi$ be defined by \eqref{eq:bound_phiA} and \eqref{chichi}, respectively. Then one has
	\[
	\begin{aligned}
	&\frac{\zeta_K'}{\zeta_K}= -\frac1K[-\chi'(x)|x|+(1-\chi(x))\sgn(x)],\\
	& \frac{\zeta_K'' }{\zeta_K}= \left(\frac{\zeta_K'}{\zeta_K}\right)^2+\frac1K[\chi''(x)|x|+2\chi'(x)\sgn(x)],
	\end{aligned}
	\]
	and
	\[
		\begin{aligned}
					\frac{\zeta_K'''}{\zeta_K}=&~{} 3\frac{\zeta_K''}{\zeta_K} \frac{\zeta_K'}{\zeta_K}
					-2 \left(\frac{\zeta_K'}{\zeta_K}\right)^3 
					+\frac1K\left[ \chi'''(x)|x|+3\chi''(x)\sgn(x)\right] ,\\
				\frac{\zeta_K^{(4)}}{\zeta_K}=&~{}
				4\frac{\zeta_K'''}{\zeta_K} \frac{\zeta_K'}{\zeta_K}
				+3\left(\frac{\zeta_K''}{\zeta_K}\right)^2 
				-12 \frac{\zeta_K''}{\zeta_K} \left(\frac{\zeta_K'}{\zeta_K}\right)^2 
				+6\left(\frac{\zeta_K'}{\zeta_K}\right)^4 
				\\&+\frac1K\left[ \chi^{(4)}(x)|x|+4\chi'''(x)\sgn(x)\right]
				 .
		\end{aligned}
	\]
\end{lem}

\medskip

The previous identities are key through the paper. From the previous lemma we observe that for any $K>0$ sufficiently large,
	\begin{equation}\label{eq:zA''-zA2}
	\left|\dfrac{ \zeta_K''}{\zeta_K} -2\left(\dfrac{\zeta_K'}{\zeta_K} \right)^{2}\right|\lesssim \frac{1}{K}.
	\end{equation}	
		Furthermore,
		\begin{equation}\label{eq:z11}
		\left| \frac{\zeta_K'}{\zeta_K}\right|\lesssim K^{-1} \textbf{1}_{\{|x|>1\}}(x),
		\quad 
				\left| \frac{\zeta_K''}{\zeta_K}\right|
				\lesssim   K^{-2}+K^{-1} \sech(x),
		\end{equation}
		and
		\[
		\begin{aligned}
			\left| \frac{\zeta_K'''}{\zeta_K}\right|
			\lesssim & ~{}  K^{-3}+K^{-1} \sech(x),\quad 
			\left| \frac{\zeta_K^{(4)}}{\zeta_K}\right|
			\lesssim   K^{-4}+K^{-1} \sech(x).
		\end{aligned}
	\]
	Then,
		\begin{equation}\label{eq:z'/z}
	\begin{aligned}
	\left| \frac{\zeta_K''}{\zeta_K}\right| +
	\left| \frac{\zeta_K'''}{\zeta_K}\right|+
	\left|  \frac{\zeta_K^{(4)}}{\zeta_K}\right|\lesssim & ~{}  K^{-1}.
	\end{aligned}
	\end{equation}
	In particular, for $A$ large enough, the following estimate holds:	 
	\begin{equation}\label{eq:C*z'/z}
	\begin{aligned}
			\left| \textbf{1}_{\{A<|x|<2A\}} \frac{\zeta_K^{(n)}}{\zeta_K}\right|\lesssim   \frac1{K^{n}}, \mbox{ for } n\in\mathbb{N}.
	\end{aligned}
	\end{equation}

\medskip

Now we will focus on $\chi_{A}$. Considering that $\chi_{A}$ is a cut-off function at $A$ scale, we notice the following relation between $\chi_{A}$ and $\zeta_{A}$: for each function $v$
	\begin{equation}\label{rem:chiA_zetaA4}
		\int \chi_A^2 v^2\leq\int_{|x|\leq 2A} v^2 \leq C\int_{|x|\leq 2A} e^{-4|x|/A}v^2 \lesssim 	\int v^2 \zeta_A^4 \leq \| \zeta_A^2 v\|^2_{L^2} .
	\end{equation}
This estimate will be essential for the well-boundedness of some nonlinear terms in what follow (see Subsections \ref{control_tJ1} and \ref{control_J4}). Notice that this estimate will be used with the translated variable $y=x-\rho(t)$.	

\subsection{A first virial estimate} Following the ideas in \cite{M_GB}, we set
\begin{equation}\label{eq:I}
\I=\Int{\R} \varphi_A(y) u_1(t,x) u_2(t,x)dx  , \qquad y=x-\rho(t),
\end{equation}
and
\begin{equation} \label{eq:wi}
w_i(t,x)=\zeta_A (y)u_i(t,x), \quad  i=1,2.
\end{equation}

Here, $\bd{w}=(w_1,w_2)$ represents a localized version of $\bd{u}=(u_1,u_2)$ at scale $A$. The following virial argument has been used in \cite{KMM2017,KMMV,M_GB} in a similar context.
 \begin{prop}\label{prop:virial_I}
There exist $C_1>0$ and $\delta_1>0$ such that for any $0<\delta\leq\delta_1$, the following holds. Fix 
\begin{equation}\label{Fijacion1}
A=\delta^{-1}.
\end{equation}
Assume that for all $t\geq 0$,  \eqref{eq:ineq_hip} holds. Then
\begin{equation}\label{eq:dI_w}
\begin{aligned}
\dfrac{d}{dt}\mathcal I
			 \leq &-\frac12 \int  \left( w_2^2  + 3(\partial_x w_1)^2  
		+\left(V_0 -4C_0\delta \right)w_1^2 \right)
		+ \int \vA  \left(  3H + u_1\right) H' u_1^2 \\
		&+C_0 \delta |\rho'|^2 +\rho' \int \vA H' u_2.
\end{aligned}
\end{equation}
\end{prop}

The proof of this result requires several computations. We start with a first identity.

\begin{lem}
Let $(u_1,u_2)\in H^1(\R)\times L^2(\R)$ a solution of \eqref{eq:u_linear}. Consider $\varphi_A =\varphi_A (y)$ a smooth bounded function to be chosen later. Then
\begin{equation}\label{lem4p3}
	\begin{aligned}
   \frac{d}{dt}\mathcal I
   	=&- \dfrac{1}{2}\int \vA'  \left(u_2^2+3(\partial_x u_1)^2 +V_0u_1^2\right) + \frac12 \int \vA''' u_1^2 \\
	&- \dfrac{1}{2}\int \vA' u_1^3 \left(\frac32  u_1 +4H \right)  
	 +\frac12 \int \vA u_1^2 \left(  V'_0  +2H' u_1 \right) \\
	 & -\rho' \int \vA' u_1u_2 +\rho' \int \vA H' u_2.
   \end{aligned}
\end{equation}
\end{lem}
\begin{proof}
Taking derivative in \eqref{eq:I} and using \eqref{eq:u_linear},
	\begin{equation}\label{eq:I'_RHS}
	\begin{aligned}
	\dfrac{d}{dt}\mathcal I
	=&\int \vA(\dot{u_1}u_2+u_1\dot{u_2}) -\rho' \int \vA' u_1u_2\\
	= &\int \vA(\partial_x u_2 u_2+ \rho' H' u_2+u_1 (\partial_x\LL(u_1)+\partial_x(u_1^3+3Hu_1^2) ) \\
	=& -\frac12 \int \vA' u_2^2 
	+\int \vA u_1 \partial_x\LL(u_1)  
	-\int \vA  u_1  \partial_x (u_1^3+3Hu_1^2) \\
	&~{}-\rho' \int \vA' u_1u_2 +\rho' \int \vA H' u_2.
	\end{aligned}
	\end{equation}
For the second integral in the RHS of the above equation, by \eqref{eq:PL} in Lemma \ref{lem:LL}, we have
	\begin{equation*}
	\begin{aligned}
	\int \vA u_1 \partial_x\LL(u_1)  
	=&-\frac12  \int \vA' [ 3 (\partial_x u_1)^2 +V_0 u_1^2 ]+\frac12 \int  \vA V_0'  u_1^2+\frac12 \int  \vA''' u_1^2.
	\end{aligned}
	\end{equation*}

For the last integral in the RHS of \eqref{eq:I'_RHS}, separating terms and integrating by parts we obtain
	\begin{equation*}
	\begin{aligned}
	\int \vA u_1 \partial_x(u_1^3+3Hu_1^2) 
			=&-\int \vA \partial_x u_1 (u_1^3+3Hu_1^2)-\int \vA'  (u_1^4+3Hu_1^3)\\
			=&{}~ \frac14\int \vA'  u_1^4 +\int (\vA' H+\vA H') u_1^3-\int \vA'  (u_1^4+3Hu_1^3)\\
			=& -\frac34\int \vA'  u_1^4 -2\int \vA'  Hu_1^3+\int \vA H' u_1^3. 
		\end{aligned}
	\end{equation*}
Finally, by collecting, noticing that $V_{0}=6HH'$ and regrouping terms, we obtain
\[
	\begin{aligned}
	\dfrac{d}{dt}\mathcal I
	=&- \dfrac{1}{2}\int \vA'  \left(u_2^2+3(\partial_x u_1)^2 +V_0u_1^2\right) + \frac12 \int \vA''' u_1^2 \\
	&- \dfrac{1}{2}\int \vA' u_1^2 \left(\frac32  u_1^2 +4Hu_1 \right)  
	 + \int \vA  u_1^2 H' \left(  3H + u_1\right)\\
	 & -\rho' \int \vA' u_1u_2 +\rho' \int \vA H' u_2.
	\end{aligned}
\]
This concludes the proof of \eqref{lem4p3}.
\end{proof}
\subsection*{Proof Proposition \ref{prop:virial_I}}
	Recalling that $\vA' =\zeta_A^2$, we have in \eqref{lem4p3}
	\begin{equation}\label{eq:vA_u1x0}
	\int \vA' \left( u_2^2 + V_0 u_1^2\right) 
	=\int \left( w_2^2 +V_0 w_1^2\right) .
	\end{equation}
	Also,
	\begin{equation}\label{eq:vA_u1x}
	\int \vA' (\partial_x u_1)^2 
	= \int (\partial_x w_1)^2  
	+\int w_1^2\dfrac{ \zeta_A''}{\zeta_A} .
	\end{equation}
	Additionally,
	\begin{equation}\label{eq:vAxxx_u1}
	\int  \vA''' u_1^2  =\int \dfrac{(\zeta_A^2)''}{\zeta_A^2}w_1^2 
	= 2\int \left(\dfrac{\zeta_A''}{\zeta_A}+\dfrac{(\zeta_A')^2}{\zeta_A^2}\right)w_1^2 .
	\end{equation}
Next, we deal with the nonlinear terms. Using that $w_1=\zeta_A u_1$, we obtain
\[
\begin{aligned}
\int \vA' &  \left(\frac32  u_1^4 +4Hu_1^3\right)  
	 = 
	 \int w_1^2  \left(\frac32  u_1^2 +4Hu_1\right). 
\end{aligned}
\]
Finally, gathering the previous computation and \eqref{eq:vA_u1x0}, \eqref{eq:vA_u1x}, and \eqref{eq:vAxxx_u1}, we obtain 
\[
\begin{aligned}
\dfrac{d}{dt} \mathcal I
=&~{} -\frac12 \int  \left( w_2^2  + 3(\partial_x w_1)^2  
		+V w_1^2 \right)
 + \int \vA u_1^2 H' \left(  3H+ u_1\right) \\
 &~{}  -\rho' \int \vA' u_1u_2 +\rho' \int \vA H' u_2 ,
\end{aligned}
\]
where (with variable $x$, $y$ and $t$)
\begin{equation}\label{VV}
V= V_0+\dfrac{ \zeta_A''}{\zeta_A} -2\dfrac{(\zeta_A')^2}{\zeta_A^2}
+u_1\left( \frac32  u_1^2 +4H\right) .
\end{equation}

Now we consider the range of parameter $A\gg 1$. Then from \eqref{eq:zA''-zA2} in \eqref{VV} one has
\[
V
\geq V_0-C(A^{-1}+\delta).
\]
This last estimate follows from \eqref{eq:zA''-zA2} and  $\|u_1\|_{L^{\infty}} \leq C_0\delta$.

Finally, consider penultimate term in \eqref{lem4p3}. Using \eqref{new cota} and \eqref{Fijacion1} one gets $\delta=A^{-1}$, and $\|u_2\|_{L^2},\| u_1\|_{L^{\infty}}\leq C_0 \delta $. Therefore,
\[
\bigg|\rho' \int \vA' u_1u_2\bigg| \lesssim |\rho'| \int |w_1\zeta_A u_2| \lesssim  |\rho'| \|w_1\|_{L^2} \| \zeta_A u_2\|_{L^2}
\leq C_0\delta   |\rho'| \|w_1\|_{L^2}.
\]
Replacing this bound in \eqref{lem4p3}, and using the Cauchy's inequality, we obtain
\[
\begin{aligned}
\dfrac{d}{dt}\mathcal I
=&
-\frac12 \int  \left( w_2^2  + 3(\partial_x w_1)^2  +Vw_1^2\right)
 + \dfrac{1}{2}\int    \vA u_1^2 \left(  V'_0 +2H' u_1\right)\\
& -\rho' \int \vA' u_1u_2 +\rho' \int \vA H' u_2\\
		 \leq &-\frac12 \int  \left( w_2^2  + 3(\partial_x w_1)^2  
		+\left(V_0 -4C A^{-1}\right)w_1^2 \right)
		+ \int \vA  \left(  3H + u_1\right)H' u_1^2 \\
		& +C_0 \delta |\rho'|^2  +\rho' \int \vA H' u_2.
\end{aligned}
\]
This concludes the proof of \eqref{eq:dI_w}.

\section{Duality and second virial estimates}\label{sec:3}

Following Martel \cite{Martel_linearKDV}, we consider the function $v_1=\LL u_1$ instead $u_1$ to obtain a transformed problem with better virial properties. Unlike generalized KdV, in our case we have a system of unknowns and some care is needed with the dual transformation. In the case of the Good Boussinesq equation this was done in \cite{M_GB} with notable success. Since our original variables $(u_1,u_2)$ belong to $H^1(\R)\times L^2(\R)$, by using $\LL$, the new variables are not well-defined. Therefore, we need a regularization procedure, as in many other works \cite{KMM2017,KMMV, M_GB}. 

\subsection{The transformed problem}
Let $\gamma >0$ small, to be determined later. Set $${\bd v}=(1-\gamma \partial_x^2)^{-1} \bd{L u},$$
 or equivalently,
\begin{equation}\label{eq:change_variable}
	\begin{cases}
	v_1= (1-\gamma \partial_x^2)^{-1}\LL  u_1,\\
	v_2= (1-\gamma \partial_x^2)^{-1}  u_2.
	\end{cases}
\end{equation}
From the system \eqref{eq:u_linear} we have  ${\bd v}=(v_1,v_2)\in H^{1}(\R)\times H^{2}(\R)$. Furthermore,
from the system \eqref{eq:u_vector}, follows that
\begin{equation*}
\begin{aligned}
\partial_t {\bd v}=&{\bd L}{\bd J }\partial_x{\bd v}
	+(1-\gamma \partial_x^2)^{-1} \tilde{\bf F},
\end{aligned}
\end{equation*}
where
\[
\tilde{\bf F}
=\left( \begin{matrix} \gamma \left[  V''_0 \partial_x v_2+2 V_0'\partial_x^2 v_2 \right]  -\rho'  V_0' u_1 \\ \partial_x\left[ u_1^3+3H u_1^2 \right]\end{matrix}\right).
\]

The above system is equivalent to 
\begin{equation}\label{eq:syst_v}
	\begin{cases}
	\dot{v}_1=\LL(\partial_x v_2)
	+G,\\
	\dot{v}_2= \partial_x v_1
	+F,
	\end{cases}
\end{equation}
where
\begin{equation}\label{eq:G_H}
\begin{aligned}
 F=&~{}(1-\gamma\partial_x ^2)^{-1} \partial_x \left[ u_1^3+3H u_1^2 \right], \\
 G=&~{} \gamma (1-\gamma\partial_x ^2)^{-1}\left[  V_0'' \partial_x v_2+2 V_0'\partial_x^2 v_2 \right] 
 -\rho' (1-\gamma \partial_x^2)^{-1} (V_0' u_1).
 \end{aligned}
\end{equation}
Now we compute a second virial estimate, this time on $(v_1,v_2)$.

\subsection{Virial functional for the transformed problem}

Recall that $y=x-\rho(t)$. Set now
\begin{equation}\label{eq:virial_J}
\J(t)=\int   \psi_{A,B}(y) v_1 (t,x)v_2(t,x)dx  ,
\end{equation}
with
\begin{equation}\label{eq:def_psiB}
\begin{aligned}
\psi_{A,B}=&~{} \chi_A^2\varphi_B,  \\
z_i (t,x)=&~{} \chi_A (y) \zeta_B (y)v_i (t,x),\ \  i= 1,2. 
\end{aligned}
\end{equation}
Here, $\bd{z}=(z_1,z_2)$ represents a localized version of the variables $\bd{v}=(v_1,v_2)$ at the scale $B$. This scale is intermediate, and $\J$ involves a cut-off at scale $A$, which is needed to bound some bad  error and nonlinear terms; see \cite{KMMV,M_GB} for a similar procedure. 

\begin{prop}\label{prop:virial_J}
Under \eqref{eq:ineq_hip}, \eqref{eq:scales} and \eqref{Fijacion1}, the following is satisfied. There exist $C_2>0$ and $\delta_2>0$ such that for any $0<\delta\leq\delta_2$, the following holds. Fix 
\begin{equation}\label{Fijacion2}
B=A^{1/10}= \delta^{-1/10}
,\quad  \gamma=B^{-4} = \delta^{2/5}.
\end{equation}
Then for all $t\geq 0$, 
	\begin{equation}\label{eq:dJ}
	\begin{aligned}
	\dfrac{d}{dt}\J
	\leq & ~{} 
	-\frac12 \int  \left( z_1^2+(V_0-C_2 \delta^{1/10}) z_2^2 +2(\partial_x z_2)^2 \right)+C_2\delta \|z_1\|_{L^2}^2 \\
&+C_2 \delta^{1/10} \left(\| w_1\|_{L^2}^2+\| \partial_{x} w_1\|_{L^2}^2+\| w_2\|_{L^2}^2 \right)+C_2 \delta |\rho'|^2,
	\end{aligned}
	\end{equation}
	where $V_0$ is given by \eqref{eq:LL}.
\end{prop}

The rest of this section is devoted to the proof of this proposition, which has been divided in several subsections.  
\subsection{Change of variable}
The following identities were proved in \cite{M_GB}; they are useful in the next subsection and in the proof of Claim \ref{claim:R_CZ_vi_xx}. 

\begin{claim} \label{claim:P_CZ_vi_x}
	Let $P\in W^{1,\infty}(\R)$, $P=P(y)$, $v_i$ be as in \eqref{eq:change_variable}, and $z_i$ be as in \eqref{eq:def_psiB}. Then 
		\begin{equation}\label{eq:claimP_CZ_B_vix_final}
	\begin{aligned}
	\int P\chi_A^2\zeta_B^2(\partial_x v_i)^2 
	=&	\int P (\partial_x z_i)^2 + \int   \left[ P' \frac{\zeta_B'}{\zeta_B}+P\frac{\zeta_B''}{\zeta_B}\right] z_i^2 +\int  \mathcal{E}_1\zeta_B^2 v_i^2,
	\end{aligned}
	\end{equation}
	where
	\begin{equation}\label{eq:claimE1}
	\begin{aligned}
	\mathcal{E}_1(P)
	= P  \left[ \chi_A''\chi_A+(\chi_A^2)' \frac{\zeta_B'}{\zeta_B}\right] 
+ \frac12 P'(\chi_A^2)' ,
	\end{aligned}
	\end{equation}
	and
	\begin{equation}\label{cota_final}
	|\mathcal{E}_1 (P) |\lesssim 
	A^{-1}\|P' \|_{L^\infty(A\leq|y|\leq2A)} 
	+ (AB)^{-1} \|P\|_{L^\infty (A\leq|y|\leq2A)}.
	\end{equation}
\end{claim}

\begin{rem}\label{rem:CZ_v2x} For further purposes, we need the following easy consequences: 
	for $P\equiv 1$, we get
	\begin{equation}\label{eq:1CZ_B_vix}
	\begin{aligned}
	\int  \chi_A^2\zeta_B^2(\partial_x v_i)^2 
	=&	\int  (\partial_x z_i)^2 +
	\int     
	\frac{\zeta_B''}{\zeta_B} z_i^2
	+\int  \mathcal{E}_1 \zeta_B^2 v_i^2,
	\end{aligned}
	\end{equation}
	where
	\begin{equation}\label{eq:E1_1}
	\begin{aligned}
	\mathcal{E}_1 =\mathcal{E}_1(1) = \chi_A''\chi_A+(\chi_A^2)' \frac{\zeta_B'}{\zeta_B}.
	\end{aligned}
	 \end{equation}
	Additionally, from  \eqref{eq:1CZ_B_vix}, \eqref{eq:z11} and \eqref{cota_final}, one has the following estimate:
	\begin{equation}\label{eq:CZvx}
	\begin{aligned}
	\| \chi_A\zeta_B \partial_x v_i\| ^2
	\lesssim & ~{}  \| \partial_x z_i\|_{L^2}^2+B^{-1}\| z_i\|_{L^2}^2
	+ (AB)^{-1}\| \zeta_{B}v_i\|_{L^2}^2.
	\end{aligned}
	\end{equation}
This estimate will be particularly useful later in the paper, see e.g. \eqref{eq:M3} and \eqref{eq:cotaN3}.
\end{rem}

\subsection{Proof of Proposition \ref{prop:virial_J}: first computations}

Recall the definition of $\J$ in \eqref{eq:virial_J}. We have from \eqref{eq:virial_J} and \eqref{eq:syst_v},
\begin{equation*}
\begin{aligned}
	\dfrac{d}{dt}\J=&\int  \psi_{A,B} \left[(\LL\partial_{x} v_2)v_2+\frac12\partial_{x}(v_1^2)+Fv_1+Gv_2 \right] -\rho' \int  \psi_{A,B}' v_1v_2\\
	=&\int  \psi_{A,B} (\LL\partial_{x} v_2)v_2  -\frac12 \int  \psi_{A,B}' v_1^2 +\int  \psi_{A,B} \left[Gv_2+Fv_1 \right]-\rho' \int  \psi_{A,B}' v_1v_2.
\end{aligned}
\end{equation*}
Applying \eqref{eq:LP} in Lemma \ref{lem:LL}, we obtain
\[
\begin{aligned}
\int  \psi_{A,B} (\LL\partial_{x} v_2)v_2 
=&-\dfrac{1}{2}\int   \psi_{A,B}'  \left(3(\partial_x v_2)^2+V_0 v_2^2\right) +\dfrac{1}{2}\int  \psi_{A,B}''' v_2^2 -\dfrac{1}{2}\int \psi_{A,B} V_0' v_2^2 .
\end{aligned}
\]
Now, we consider the following decomposition
\begin{equation}\label{eq:J'_v}
\begin{aligned}
	\dfrac{d}{dt}\J
	=&-\dfrac{1}{2}\int   \psi_{A,B}'  \left(v_1^2+V_0 v_2^2+3(\partial_x v_2)^2\right) 
	+\dfrac{1}{2}\int  \psi_{A,B}''' v_2^2-\dfrac{1}{2}\int \psi_{A,B} V_0' v_2^2  \\
	 &
	+\int  \psi_{A,B} Gv_2 +\int  \psi_{A,B} Fv_1  -\rho' \int  \psi_{A,B}' v_1v_2\\
	=: & ~{} (J_1+J_2+ J_3)+(J_4+J_5 +J_6).
\end{aligned}
\end{equation}
By definition of $\psi_{A,B}$ (see \eqref{eq:def_psiB}), it follows that
\begin{equation}\label{eq:psi_deriv}
\begin{aligned}
\psi_{A,B}' =&~{}\chi_A^2 \zeta_B^2+(\chi_A^2)'\varphi_B,\\
\psi_{A,B}''' =& ~{}\chi_A^2 (\zeta_B^2)''+3(\chi_A^2)' (\zeta_B^2)'+3(\chi_A^2)'' \zeta_B^2+(\chi_A^2)'''\varphi_B.
\end{aligned}
\end{equation}
Also, by the definition of $\bd{z}=(z_1,z_2)$ in \eqref{eq:def_psiB}, we have:
\begin{equation*}
\begin{aligned}
-2J_1 =& \int   \psi_{A,B}'  \left(v_1^2+V_0 v_2^2+3(\partial_x v_2)^2\right)\\
=& \int  (z_1^2+V_0 z_2^2)  +\int   (\chi_A^2)'\varphi_B\left(v_1^2+V_0 v_2^2+3(\partial_x v_2)^2\right) +3\int   \chi_A^2 \zeta_B^2(\partial_x v_2)^2.
\end{aligned}
\end{equation*}
For the last term on the above equation, applying Remark \ref{rem:CZ_v2x}, we obtain
\[
\begin{aligned}
\int \chi_A^2\zeta_B^2(\partial_x v_2)^2 =&\int  (\partial_x z_2)^2+ \int  \dfrac{\zeta_B''}{\zeta_B}z_2^2
+ \int   \left[ \chi_A'' +2\chi_A' \frac{\zeta_B'}{\zeta_B}\right]\chi_A \zeta_B^2 v_2^2.
\end{aligned}
\]
Then, for $J_1$ in \eqref{eq:J'_v} we obtain
\begin{equation}\label{J_1_intermedio}
\begin{aligned}
J_1
=& -\frac12 \int  (z_1^2+3 (\partial_x z_2)^2+V_0 z_2^2) 
-\frac32\int  \dfrac{\zeta_B''}{\zeta_B}z_2^2\\
&-\frac32\int   \left[ \chi_A'' +2 \chi_A' \frac{\zeta_B'}{\zeta_B}\right]\chi_A \zeta_B^2 v_2^2
-\frac12 \int   (\chi_A^2)'\varphi_B\left(v_1^2+V_0 v_2^2 +3(\partial_x v_2)^2\right).
\end{aligned}
\end{equation}
Now we turn into $J_2$. By \eqref{eq:psi_deriv}, $J_2$  in \eqref{eq:J'_v} satisfies the following decomposition
\begin{equation}\label{J_2_intermedio}
\begin{aligned}
J_2
=&~{} \frac12 \int \left( \chi_A^2 (\zeta_B^2)''+3(\chi_A^2)' (\zeta_B^2)'+3(\chi_A^2)'' \zeta_B^2+(\chi_A^2)'''\varphi_B \right) v_2^2 \\
=& ~{} \int \left[ \left(\dfrac{\zeta_B'}{\zeta_B}\right)^2+\dfrac{\zeta_B''}{\zeta_B} \right]z_2^2
+\frac12 \int \left( 3(\chi_A^2)' (\zeta_B^2)'+3(\chi_A^2)'' \zeta_B^2+(\chi_A^2)'''\varphi_B \right) v_2^2 .
\end{aligned}
\end{equation}
As for $J_3$  in \eqref{eq:J'_v}, using the definition of $z_2$ in \eqref{eq:def_psiB}, we obtain
\begin{equation}\label{J_3_intermedio}
\begin{aligned}
J_3
=-\frac12 \int \psi_{A,B} V_0' v_2^2  = -\frac12\int  \frac{\varphi_B}{\zeta_B^2} V_0'  z_2^2  .
\end{aligned}
\end{equation}
Finally, gathering \eqref{J_1_intermedio}, \eqref{J_2_intermedio} and \eqref{J_3_intermedio}, we obtain that the first part in \eqref{eq:J'_v} can be written as
\[
J_1+J_2+J_3= -\frac12 \int  \left[z_1^2+Vz_2^2 +3(\partial_x z_2)^2 \right] 
+\tilde{J_1},
\]
where
\begin{equation}\label{VVV}
V=V_0+\frac{\zeta_B''}{\zeta_B}-2\dfrac{(\zeta_B')^2}{\zeta_B^2}+V_0'\dfrac{\varphi_B}{\zeta_B^2
},
\end{equation}
and the error term is given by
\begin{equation}\label{eq:tildeJ1}
\begin{aligned}
\tilde{J_1}=&
-\frac12 \int   \left[ 3\big( \chi''_A \chi_A-(\chi_A^2)''\big)  \zeta_B^2-\frac32(\chi_A^2)'( \zeta_B^2)'+\big((\chi_A^2)' V_0-(\chi_A^2)'''\big)\varphi_B\right] v_2^2\\
&
-\frac12\int   (\chi_A^2)'\varphi_B v_1^2   -\frac32\int   (\chi_A^2)'\varphi_B(\partial_x v_2)^2.
\end{aligned}
\end{equation}
In order to control the main part of the virial term, a lower bound for the potential $V$ is necessary. We have the following result:
\begin{lem}
	There are  $C>0$ and $B_0>0$ such that for all $B\geq B_0$, one has
	\[
	V\geq V_0- CB^{-1}, \ \ \mbox{where} \ \ V_0=-1+3H^2.
	\]
\end{lem}
\begin{proof}
	First,  noticing that $V_0'=6HH'>0$ for $y>0$ 
	and  using that for $y\in [0,\infty)\mapsto \zeta_B(y)$ is a non-increasing function, we have for $y>0,$
	\[
	\dfrac{\varphi_B}{\zeta_B^2}=\dfrac{\int_{0}^y \zeta_B^2 }{\zeta_B^2} \geq y>0.
	\]
	 Then, from this estimate, \eqref{VVV} and \eqref{eq:zA''-zA2} with $K=B$,
	\begin{equation*}
		\begin{aligned}
		V(y)
		\geq &~{} V_0 (y)- C B^{-1}+|V_0'(y)y|\geq V_0(y)- CB^{-1}. 
		\end{aligned}
	\end{equation*}
The case $y\leq 0$ is similar. These estimates hold for any $y\in\R$. This concludes the proof.
\end{proof}

\noindent
{\bf First conclusion}. Using this lemma, and $\tilde{J}_1$ in \eqref{eq:tildeJ1}, we conclude
\begin{equation}\label{eq:estimates_J}
\frac{d}{dt}\J\leq -\frac12 \int  \left[z_1^2+(V_0-CB^{-1}) z_2^2 +3(\partial_x z_2)^2 \right] 
+\tilde{J_1}+J_4+J_5+J_6,
\end{equation}
where $J_4$ and $J_5$ are related with the nonlinear term in \eqref{eq:J'_v} and $J_6$ is related with the shift considered on the kink.
To control the terms $\tilde{J_1}, J_4, J_5$ and $J_6$, and the terms that will appear in the sections below, some technical estimates will be needed.

\subsection{First technical estimates}

The following estimaste have been used to estabilsh the asymptotic stability for the Goiod Boussinesq equation, as well as the $1+1$ scalar fields equation (see \cite{M_GB,KMMV} for proof). 
For $\gamma>0$, let $(1-\gamma \partial_x^2)^{-1}$
	be the bounded operator from $L^2$ to $H^2$ defined by its Fourier transform as
	\[
	\mathcal F((1-\gamma \partial_x^2)^{-1} g)(\xi)=\frac{\widehat{g}(\xi)}{1+\gamma \xi^2},\quad \mbox{for any }  g\in L^2. 
	\]
We start with a basic but essential result, in the spirit of \cite{KMMV}.

\begin{lem}\label{lem:estimates_IOp}
	Let $f\in L^2(\R)$ and $0<\gamma<1$ fixed. We have the following estimates:
	\vspace{0.1cm}
	\begin{enumerate}[$(i)$]
		\item $\| (1-\gamma \partial_x^2)^{-1}f\|_{L^2(\R)}\leq \|f \|_{L^2(\R)}$,
		\vspace{0.1cm}
		\item $\| (1-\gamma \partial_x^2)^{-1}\partial_x f\|_{L^2(\R)}\leq \gamma^{-1/2}\|f \|_{L^2(\R)}$,
		\vspace{0.1cm}
		\item $\| (1-\gamma \partial_x^2)^{-1}f\|_{H^2(\R)}\leq \gamma^{-1}\|f \|_{L^2(\R)}$.
	\end{enumerate}
\end{lem}
We also enunciate the following results that appear in \cite{KMMV,KMM2017,M_GB}. Notice that now the variable in the weights is $y=x-\rho(t)$, but the shift does not affect the final outcome.
\begin{lem}\label{lem:cosh_kink}
	There exist $\gamma_1>0$ and $C>0$ such that for any $\gamma\in (0,\gamma_1)$, $0<K\leq 1$ and $g\in L^2$, the following estimates holds
	\vspace{0.1cm}
	\begin{align}
	\label{eq:lem_sec_Iop g}
	\left\| \sech\left( K y\right) (1-\gamma\partial_x^2)^{-1} g\right\|_{L^2}
	\leq{}~ C &\left\| (1-\gamma\partial_x^2)^{-1}\left[ \sech\left( Ky\right)  g\right]\right\|_{L^2},
	\\ \vspace{0.1cm}
	\label{eq:cosh_kink}
	\left\| \cosh\left( Ky\right) (1-\gamma\partial_x^2)^{-1} \left[\sech\left( Ky\right) g\right]\right\|_{L^2}
	\leq& {}~ C \left\| (1-\gamma\partial_x^2)^{-1} g\right\|_{L^2},\\
	\vspace{0.1cm}
\label{cor:estimates_Sech_Iop_partial}
	\left\| \sech(Ky)(1-\gamma \partial_x^2)^{-1}\partial_x g \right\|_{L^2} 
	\leq&{}~ C \gamma^{-1/2} \left\| \sech(Ky) g \right\|_{L^2},
	\end{align}
and	
\begin{equation}\label{lem:estimates_Sech_Iop_Op}
	\left\| \sech(Ky)(1-\gamma \partial_x^2)^{-1}(1-\partial_x^2)g \right\|_{L^2} \leq {}~ C \gamma^{-1} \| \sech(Ky) g \|_{L^2},
	\end{equation}
\medskip
where the explicit constant $C$ is independent of $\gamma$ and $K$.
\end{lem}

The following results are also contained in \cite{M_GB}, and are essentially obtained from Lemma \ref{lem:estimates_IOp}.
\begin{lem}
Recall $v_i$, $w_i$ and $z_i$, $i=1,2$ defined in \eqref{eq:change_variable}, \eqref{eq:wi} and \eqref{eq:def_psiB}, respectively. Then one has:
	\begin{enumerate}[(a)]
		
		\item Estimates on $v_1$ and $u_1$:
		\begin{equation}\label{eq:estimates_v1}
		\begin{aligned}
		\|v_1\|_{L^2} \lesssim &~{} \gamma^{-1} \|u_1 \|_{L^2},\\
		\|\partial_x v_1\|_{L^2}\lesssim &~{} \gamma^{-1/2} \| u_1 \|_{L^2}+\gamma^{-1} \|\partial_x u_1\|_{L^2} .
		\end{aligned}
		\end{equation} 
				
		\item Estimates on $v_2$ and $u_2$:
		\begin{equation}\label{eq:estimates_v2}
		\begin{gathered}
		\|v_2\|_{L^2}\lesssim  \| u_2\|_{L^2}, \quad 
			\|\partial_x v_2\|_{L^2} \lesssim \gamma^{-1/2}\| u_2\|_{L^2},\\
				\|\partial_x^2 v_2\|_{L^2}\lesssim \gamma^{-1} \| u_2\|_{L^2}.
		\end{gathered}
		\end{equation}
	\end{enumerate}
\end{lem}

\begin{lem}\label{lem:v_w}
Let  $1\leq K\leq A$ fixed. Then
	
	\begin{enumerate}[(a)]
		
		\item  Estimates on $v_1$ and $w_1$:
		\begin{equation}\label{eq:estimates_v1_w1}
		\begin{aligned}
		\|\zeta_K v_1\|\lesssim &~{}\gamma^{-1} \|w_1 \|_{L^2},\\
		\|\zeta_K \partial_x v_1\|\lesssim &~{}\gamma^{-1} \left(  \| w_1 \|_{L^2} + \|\partial_x w_1 \|_{L^2} \right).
		\end{aligned}
		\end{equation} 
		
		\item
		Estimates on $v_2$ and $w_2$:
		\begin{equation}\label{eq:estimates_v2_w2}
		\begin{gathered}
		\| \zeta_K v_2\|_{L^2}\lesssim  \| w_2\|_{L^2}, \quad \quad 
		\|\zeta_K \partial_x v_2\|_{L^2} \lesssim \gamma^{-1/2}\| w_2\|_{L^2},\\
		\|\zeta_K \partial_x^2 v_2\|_{L^2}\lesssim \gamma^{-1} \| w_2\|_{L^2}.
		\end{gathered}
		\end{equation}

		\item 
		Estimates on $u_1$ and $w_1$:
		\begin{equation}\label{eq:estimates_sech_u1_w1}
		\begin{aligned}
		\left\|\sech \left( \frac{y}{K}\right) u_1\right\|_{L^2} \lesssim &~{} \|w_1 \|_{L^2},\\
		\left\|\sech \left( \frac{y}{K}\right) \partial_x u_1\right\|_{L^2}	\lesssim &~{} \|\partial_x w_1 \|_{L^2}+ \| w_1 \|_{L^2}.
		\end{aligned}
		\end{equation} 
		\item 
		Estimates on $u_2$ and $w_2$:
		\begin{equation}\label{eq:sech_u2_w2}
		\begin{gathered}
		\left\| \sech \left( \frac{y}{K}\right) u_2\right\|_{L^2}\lesssim  \| w_2\|_{L^2}.
		\end{gathered}
		\end{equation}
	\end{enumerate}
\end{lem}

For the purposes of this work, we also include a refined version of Lemma \ref{lem:v_w}. 
\begin{lem}
Let $1\leq K\leq A$ fixed. Then,
\begin{equation}\label{v1_mejorada}
\begin{aligned}
\|(1-\gamma\partial_{x}^{2})^{-1}\LL f \|_{L^2}\lesssim &~{}\|f\|_{L^{2}} +\gamma^{-1/2}\left\| \partial_{x} f\right\|_{L^{2}};\\
\|\zeta_{K} v_{1}\|_{L^{2}}\lesssim &~{}  \left(1+\frac{1}{K\gamma^{1/2}}\right)  \left\|\zeta_{K} u_{1}\right\|_{L^{2}}
		  +\gamma^{-1/2}\left\| \zeta_{K} \partial_{x} u_{1}\right\|_{L^{2}}.
\end{aligned}
\end{equation}
\end{lem}
\begin{proof}
We havet
		\begin{equation*}
		\begin{aligned}
		\|(1-\gamma\partial_{x}^{2})^{-1}\LL f\|_{L^{2}}
		=&~{} \left\|(1-\gamma\partial_{x}^{2})^{-1}[(2-\partial_{x}^{2}) f-3 \sech^{2}\left( \frac{y}{\sqrt{2}}\right)f]\right\|_{L^{2}}\\
		\lesssim& ~{} \left\|(1-\gamma\partial_{x}^{2})^{-1}(2-\partial_{x}^{2}) f\right\|_{L^{2}}+\left\|  \sech^{2}\left( \frac{\cdot-\rho}{\sqrt{2}}\right)f\right\|_{L^{2}}\\
		\lesssim& ~{} \left\|(1-\gamma\partial_{x}^{2})^{-1}(2-\partial_{x}^{2}) f \right\|_{L^{2}}+\left\|  f\right\|_{L^{2}}.
		\end{aligned}
		\end{equation*}
Focus on the first term on RHS. Using the Plancherel's Theorem, we get
		\[
		\begin{aligned}
		  \left\|(1-\gamma\partial_{x}^{2})^{-1}(2-\partial_{x}^{2}) f\right\|_{L^{2}}
		  =& ~{} \left\| \frac{(2+\xi^{2})}{(1+\gamma\xi^{2})} \widehat{f}\right\|_{L^{2}}\\
		  =&~{} \left\| \left(\frac{(2+\xi^{2})}{(1+\gamma\xi^{2})} 1_{[-1,1]}+\frac{(2+\xi^{2})}{(1+\gamma\xi^{2})} 1_{[-1,1]^{c}}\right)\widehat{f}\right\|_{L^{2}}\\
		  \leq& ~{}\left\| \frac{3}{(1+\gamma\xi^{2})} 1_{[-1,1]}\widehat{f}\right\|_{L^{2}}
		  +\left\|\frac{3\sgn(\xi)|\xi|}{(1+\gamma\xi^{2})} 1_{[-1,1]^{c}} \xi\widehat{f}\right\|_{L^{2}}\\
		    \leq&~{} \left\| \frac{3}{(1+\gamma\xi^{2})} 1_{[-1,1]}\widehat{f}\right\|_{L^{2}}
		  +\left\|3\gamma^{-1/2}\frac{\sgn(\xi)\sqrt{1+\gamma|\xi|^{2}}}{(1+\gamma\xi^{2})} 1_{[-1,1]^{c}} \xi\widehat{f}\right\|_{L^{2}}\\
		    \lesssim&~{} \left\|f\right\|_{L^{2}}
		  +\gamma^{-1/2}\left\| \partial_{x} f\right\|_{L^{2}}.
		\end{aligned}
		\]
This concludes the proof of the first estimate in \eqref{v1_mejorada}.

\medskip

For the second inequality, 	by \eqref{eq:lem_sec_Iop g} and using that $\zeta_{K}\LL(u_{1})=\LL(\zeta_{K} u_{1})+2\zeta_{K}'\partial_{x} u_{1}+\zeta_{K}'' u_{1}$, we obtain
		\[
		\begin{aligned}
		\|\zeta_{K} v_{1}\|_{L^{2}}\lesssim& \| (1-\gamma \partial_{x}^{2})^{-1} (\zeta_{K} \LL u_{1})\|_{L^{2}}\\
		\lesssim& \| (1-\gamma \partial_{x}^{2})^{-1}(\LL(\zeta_{K} u_{1})+2\zeta_{K}'\partial_{x} u_{1}+\zeta_{K}'' u_{1})\|_{L^{2}}\\
		\lesssim& \| (1-\gamma \partial_{x}^{2})^{-1}\LL(\zeta_{K} u_{1})\|_{L^{2}}+ \| 2\zeta_{K}'\partial_{x} u_{1}\|_{L^{2}}+ \| \zeta_{K}'' u_{1}\|_{L^{2}}.
		\end{aligned}
		\]
Therefore, from the first identity in \eqref{v1_mejorada}, \eqref{eq:z11} and \eqref{eq:z'/z},
		\[
		\begin{aligned}
		\|\zeta_{K} v_{1}\|_{L^{2}}
		\lesssim&~{} \left\|\zeta_{K} u_{1}\right\|_{L^{2}}
		  +\gamma^{-1/2}\left\|  \partial_{x} (\zeta_{K} u_{1})\right\|_{L^{2}} + \| \zeta_{K}'\partial_{x} u_{1}\|_{L^{2}}+ \| \zeta_{K}'' u_{1}\|_{L^{2}}\\
		\lesssim&~{}\left(1+\frac{1}{K\gamma^{1/2}}\right) \left\|\zeta_{K} u_{1}\right\|_{L^{2}}
		  + \left( \gamma^{-1/2} + \frac1K\right)\left\|  \zeta_{K} \partial_{x} u_{1}\right\|_{L^{2}}\\
		  \lesssim&~{}\left(1+\frac{1}{K\gamma^{1/2}}\right) \left\|\zeta_{K} u_{1}\right\|_{L^{2}} +  \gamma^{-1/2} \left\|  \zeta_{K} \partial_{x} u_{1}\right\|_{L^{2}}.
		\end{aligned}
		\]
		This ends the proof.
		
\end{proof}
\begin{rem}
Using \eqref{eq:vA_u1x}, for $K\in[1,A]$, and $\zeta_{A}\partial_{x}u_{1}=\partial_{x} w_{1}-\frac{\zeta_{A}'}{\zeta_{A}} w_{1}$, we obtain
\[
\begin{aligned}
	\int \zeta_{K}^{2} (\partial_x u_1)^2 
		=&\int \frac{\zeta_{K}^{2}}{\zeta_{A}^{2}}(\partial_{x} w_{1})^{2}
		+\int  w_{1}^{2} \left( 
						\frac{\zeta_{K}^{2}}{\zeta_{A}^{2}} \frac{\zeta_{A}''}{\zeta_{A}}
						+2 \frac{\zeta_{K}}{\zeta_{A}} \frac{\zeta_{A}'}{\zeta_{A}}\left(\frac{\zeta_{K}'}{\zeta_{A}}-\frac{\zeta_{K}}{\zeta_{A}}\frac{\zeta_{A}'}{\zeta_{A}} \right)
						\right).
\end{aligned}
\]
Hence, a crude estimate gives
\[
\|\zeta_{K} \partial_{x} u_{1}\|^{2}_{L^{2}} \lesssim \|\partial_{x} w_{1}\|_{L^{2}}^{2}+A^{-1}\|w_{1}\|^{2}_{L^{2}}.
\]
Taking $K=A$ we obtain
\begin{equation}\label{eq:ineq_zK_u1x}
\|\zeta_{A} \partial_{x} u_{1}\|_{L^{2}} \lesssim \|\partial_{x} w_{1}\|_{L^{2}} +A^{-1/2}\|w_{1}\|_{L^{2}}.  
\end{equation}
and from \eqref{v1_mejorada} and $K=A$,
\[
\begin{aligned}
\|\zeta_{A} v_{1}\|_{L^{2}}\lesssim &~{}  \left(1+\frac{1}{(A\gamma)^{1/2}}\right)  \left\| w_{1}\right\|_{L^{2}}
		  +\gamma^{-1/2} \|\partial_{x} w_{1}\|_{L^{2}}.
  \end{aligned}
\]
Finally, using \eqref{Fijacion2}, $A\gamma = \delta^{-1+2/5} \gg 1$, and 
\begin{equation}\label{eq:kv1_w1}
\begin{aligned}
\|\zeta_{A} v_{1}\|_{L^{2}}\lesssim &~{} \left\| w_{1}\right\|_{L^{2}}
		  +\gamma^{-1/2} \|\partial_{x} w_{1}\|_{L^{2}}.
  \end{aligned}
\end{equation}
\end{rem}

\subsection{Controlling  error and nonlinear terms}

 By the definition of $\zeta_B$ and $\chi_A$ in \eqref{eq:bound_phiA} and \eqref{eq:psi_chiA}, it holds that
\begin{equation}
\begin{aligned}\label{eq:bound_ZCV_B}
& \zeta_B(y)\lesssim e^{-\frac{|y|}{B}}, \quad |\zeta_B'(y)|\lesssim \frac{1}{B} e^{-\frac{|y|}{B}}, \ \  |\varphi_B| \lesssim B, \\
& |\chi_{A}'|\lesssim A^{-1}, \ \ |(\chi_{A}^2)'|\lesssim A^{-1},\ \ |(\chi_{A}^2)''|\lesssim A^{-2},\ \ \ |(\chi_{A}^2)'''|\lesssim A^{-3}.
\end{aligned}
\end{equation}

\subsubsection{Control of $\tilde{J_1}$.}\label{control_tJ1}
Considering the following decomposition $\tilde{J_1}$:
\[
\begin{aligned}
\tilde{J_1}=&
-\frac12\int   (\chi_A^2)'\varphi_B (v_1^2 +3(\partial_x v_2)^2)
-\frac12 \int   \big( (\chi_A^2)' V_0-(\chi_A^2)''' \big)\varphi_B v_2^2
\\&
-\frac12 \int   \left[ 3\big( \chi''_A \chi_A-(\chi_A^2)''\big) -3(\chi_A^2)' \zeta_B' \right]  \zeta_B^2 v_2^2\\
=: &~{} H_1+H_2+H_3.
\end{aligned}
\]
In the case of $H_1$ and $H_2$, using $|(\chi_A^2)'\varphi_B|\lesssim A^{-1}B$ and \eqref{rem:chiA_zetaA4}, we obtain
\begin{equation*}
\begin{aligned}
|H_1|
\lesssim &~{}   A^{-1}B (\|v_1 \|_{L^2(|y|\leq 2A)}^2+ \|\partial_x v_2\|_{L^2(|y|\leq 2A)}^2)\\
\lesssim &~{} A^{-1}B (\|\zeta_A^2 v_1 \|_{L^2}^2+ \|\zeta_A^2\partial_x v_2\|_{L^2}^2),
\end{aligned}
\end{equation*}
and
\[
\begin{aligned}
|H_2|\lesssim & ~{}   A^{-1}B \| v_2\|_{L^2(A \leq |y|\leq 2A)}^2
\lesssim  A^{-1}B \| v_2\|_{L^2(|y|\leq 2A)}^2 \lesssim A^{-1}B\|\zeta_A^2 v_2\|_{L^2}^2.
\end{aligned}
\]
Finally, in the case of $H_3$, using
\eqref{eq:bound_ZCV_B}, we have
\[
\begin{aligned}
|H_3|
\lesssim &~{} 
(AB)^{-1}\|\zeta_B v_2\|_{L^2(|y|\leq 2A)}^2
\lesssim  (AB)^{-1} \|\zeta_B v_2\|_{L^2}^2.
\end{aligned}
\]
We obtain, using that $\zeta_B\lesssim \zeta_A$, 
\[
\begin{aligned}
|\tilde{J_1}|
\lesssim & ~{}  A^{-1}B\left( \|\zeta_A v_1\|_{L^2}^2 +  \|\zeta_A v_2 \|_{L^2}^2  +  \| \zeta_A \partial_x v_2\|_{L^2}^2 \right).
\end{aligned}
\]
Applying \eqref{eq:kv1_w1} and \eqref{eq:estimates_v2_w2} with $K=A$, and finally using \eqref{Fijacion2}, we get
\begin{equation}\label{eq:estimates_tildeJ1}
\begin{aligned}
|\tilde{J_1}|
\lesssim & ~{}  A^{-1}B\left( \| w_1\|_{L^2}^2 +  \gamma^{-1} \| \partial_x w_1\|_{L^2}^2 +   \gamma^{ -1} \| w_2\|_{L^2}^2 \right) \\
\lesssim & ~{} A^{-1}B \gamma^{-1}\left( \| w_1\|_{L^2}^2 +   \| \partial_x w_1\|_{L^2}^2 +   \| w_2\|_{L^2}^2 \right) \\
\lesssim & ~{} \delta^{1/2}\left( \| w_1\|_{L^2}^2 +   \| \partial_x w_1\|_{L^2}^2 +   \| w_2\|_{L^2}^2 \right) .
\end{aligned}
\end{equation}

\subsubsection{Control of $J_4$.}\label{control_J4}

Using the value of $G$ in \eqref{eq:G_H}, $J_4$ is bounded as follows:
\begin{equation}\label{decoJ4}
\begin{aligned}
J_4= &~{}
\gamma  \int  \psi_{A,B} v_2(1-\gamma\partial_x ^2)^{-1}\left[  2\partial_x (V_0' \partial_x v_2)-  V_0'' \partial_x v_2 \right]
 \\
 &~{}  -\rho'  \int \psi_{A ,B} v_2 (1-\gamma \partial_x^2)^{-1} (V_0' u_1)\\
 =: &~{} J_{41}+J_{42}.
\end{aligned}
\end{equation}
Bound on $J_{41}.$ Using H\"older's inequality,
\[
\begin{aligned}
|J_{41}|
\lesssim & ~{}  \gamma \|\psi_{A,B} v_2\|_{L^2}
	\left\|(1-\gamma\partial_x ^2)^{-1}[  2\partial_x (V_0' \partial_x v_2)-  V_0'' \partial_x v_2 ]\right\|_{L^2}.
\end{aligned}
\]
Using $\psi_{A,B}=\chi_A^2\varphi_B$ in \eqref{eq:def_psiB}, and \eqref{rem:chiA_zetaA4}, one can see that
\begin{equation}\label{eq:necesaria}
\|\psi_{A,B}v_2\|_{L^2}
 \lesssim B \|\chi_A v_2\|_{L^2}
 \lesssim  B \| v_2\|_{L^2(|y|<2A)}
  \lesssim  B \| \zeta_A^2 v_2\|_{L^2}.
\end{equation}
On the other hand, using Lemma \ref{lem:estimates_IOp},
\[
\left\| (1-\gamma\partial_x ^2)^{-1}[  2\partial_x (V_0' \partial_x v_2)-  V_0'' \partial_x v_2 ]\right\|_{L^2}\lesssim \gamma^{-1/2}\|V_0' \partial_x v_2\|_{L^2}+
\|V_0'' \partial_x v_2 \|_{L^2}.
\]
\noindent Recall that $|V_0'|, |V_0''| \sim \sech^2(y/\sqrt{2}) \lesssim e^{-\sqrt{2}|y|}$.
Therefore, we are led to the estimate of
\[
\| e^{-\sqrt{2}|y|} \partial_x v_2\|_{L^2}.
\]
First of all, write
\[
e^{-\sqrt{2}|y|} \partial_x v_2 = (e^{-\sqrt{2}|y|/2} \zeta_B^{-1}) e^{-\sqrt{2}|y|/2}  \zeta_B (1-\chi_A)\partial_x v_2 + (e^{-\sqrt{2}|y|/2}   \zeta_B^{-1} ) e^{-\sqrt{2}|y|/2} \chi_A   \zeta_B\partial_x v_2.
\]
Noticing that since $ e^{-\sqrt{2}|y|/2}\zeta_B^{-1}\lesssim 1$, 
\begin{equation}\label{eq:exp_vx}
\begin{aligned}
\|  e^{-\sqrt{2}|y|} \partial_x v_2\|_{L^2}
\lesssim &~{}  \|  e^{-\sqrt{2}|y|/2}  \zeta_B (1-\chi_A) \partial_x v_2\|_{L^2} +\| e^{-\sqrt{2}|y|/2}  \zeta_B \chi_A \partial_x v_2\|_{L^2} \\
\lesssim &~{} e^{-\sqrt{2}A/2} \|  \zeta_B \partial_x v_2\|_{L^2} +  \| e^{-\sqrt{2}|y|/2}   \chi_A\zeta_B \partial_x v_2\|_{L^2}.
\end{aligned}
\end{equation}
Using \eqref{eq:claimP_CZ_B_vix_final} and \eqref{cota_final} with $P=e^{-\sqrt{2}|y|/2}$ we get
\[
\| e^{-\sqrt{2}|y|/2}   \chi_A\zeta_B \partial_x v_2\|_{L^2}\lesssim  \| \partial_x z_2\|_{L^2}  +B^{-1/2}\| z_2\|_{L^2} +\frac1{A^{1/2}} e^{-\sqrt{2}A/4}\|\zeta_B v_2\|_{L^2}.
\] 
Coming back to the bound on $\| e^{-\sqrt{2}|y|} \partial_x v_2\|_{L^2}$, we get
\begin{equation*}
\begin{aligned}
\|  e^{-\sqrt{2}|y|}  \partial_x v_2\|_{L^2}
\leq &~{} e^{-\sqrt{2}A/2} \|  \zeta_B \partial_x v_2\|_{L^2} + \| \partial_x z_2\|_{L^2}  +B^{-1/2}\| z_2\|_{L^2}
\\&
+ (Ae^{\sqrt{2}A/2})^{-1/2}\| \zeta_B v_2\|_{L^2}.
\end{aligned}
\end{equation*}
The previous inequality allows us to conclude that 
\begin{equation}\label{eq:bound_f'x_v2x}
\begin{aligned}
\left\| (1-\gamma\partial_x ^2)^{-1}[  2\partial_x (V_0' \partial_x v_2)-  V_0'' \partial_x v_2 ] \right\|_{L^2}
\lesssim & ~{} 
\gamma^{-1/2} \left(\| \partial_x z_2\|_{L^2}+ B^{-1/2}\|z_2\|_{L^2}\right)\\
  & +e^{-A/4}\gamma^{-1/2} \left(    \| \zeta_B v_2\|_{L^2} + \|\zeta_B \partial_x v_2\|_{L^2}\right).
\end{aligned}
\end{equation}
Gathering \eqref{eq:necesaria} and this last estimate, we conclude first
\[
|J_{41}|
\lesssim  
\gamma^{1/2} B \|\zeta_A^2 v_2\|_{L^2}
\left(
\| \partial_x z_2\|_{L^2}
+ \|z_2\|_{L^2}
+e^{-A/4}  (  \|\zeta_{B} v_2\|_{L^2} + \|\zeta_B\partial_x v_2\|_{L^2})
\right).
\]
Using \eqref{eq:estimates_v2_w2} with $K=A,B$, Cauchy's inequality and \eqref{Fijacion2}, we obtain $e^{-A/4}  \gamma^{-1/2}\ll 1$ and
\[
\begin{aligned}
|J_{41}|
\lesssim & ~{} 
\gamma^{1/2} B \|w_2\|_{L^2}
\left(
\| \partial_x z_2\|_{L^2}
+ \|z_2\|_{L^2}
+e^{-A/2}  \gamma^{-1/2} \|w_2\|_{L^2}
\right)\\
\lesssim & ~{} 
\gamma^{1/2} B \left(  \|w_2\|_{L^2}^2 + \| \partial_x z_2\|_{L^2}^2 + \|z_2\|_{L^2}^2 \right).
\end{aligned}
\]
We conclude that $\gamma^{1/2} B = B^{-1} = \delta^{1/10}$ and
\begin{equation}\label{eq:estimates_J41}
\begin{aligned}
|J_{41}|
\lesssim \delta^{1/10} \left[  \|w_2\|_{L^2}^2 + \| \partial_x z_2\|_{L^2}^2 + \|z_2\|_{L^2}^2 \right].
\end{aligned}
\end{equation}
Now, we focus on $J_{42}$ in \eqref{decoJ4}. By H\"older's inequality, we get
\[
\begin{aligned}
|J_{42}|
=&~{} | \rho'| \bigg| \int \psi_{A ,B} v_2 (1-\gamma \partial_x^2)^{-1} (V_0' u_1)\bigg|\\
\lesssim&~{} | \rho'| \|  \chi_{A}\varphi_B  v_2\|_{L^2} \left\| \chi_{A} (1-\gamma \partial_x^2)^{-1} (V_0' u_1)\right\|_{L^2}.
\end{aligned}
\]
Now, using \eqref{eq:estimates_v2} and $| \chi_{A}\varphi_B |\lesssim B$, and \eqref{new cota},
\[
\|  \chi_{A}\varphi_B  v_2\|_{L^2} \lesssim B\|u_2\|_{L^2} \lesssim B\delta.  
\]
Additionally, from Lemma \ref{lem:estimates_IOp}, the fact that 
\[
|V_0'\zeta_A^{-1}|=6| \zeta_A^{-1} H'(y)H(y)|\lesssim 1,
\]
and \eqref{eq:wi}, we obtain
\[
 \left\| \chi_{A} (1-\gamma \partial_x^2)^{-1} (V_0' u_1)\right\|_{L^2} \lesssim \| \zeta_A u_1\|_{L^2} =\|w_1\|_{L^2}.
\]
Since $\delta B = \delta^{9/10}$, we conclude:
\begin{equation}\label{eq:estimates_J42}
|J_{42}| \lesssim \delta B | \rho'| \|  w_1\|_{L^2} \lesssim \delta^{9/10}  \left( | \rho'|^2 + \|  w_1\|_{L^2}^2 \right).
\end{equation}
Gathering \eqref{eq:estimates_J41} and \eqref{eq:estimates_J42}, we finally arrive to the estimate
\begin{equation}\label{eq:estimates_J4}
\begin{aligned}
|J_{4}| \lesssim \delta^{1/10} \left[ \|  w_1\|_{L^2}^2+  \|w_2\|_{L^2}^2 + \|z_2\|_{L^2}^2+ \| \partial_x z_2\|_{L^2}^2  \right] +\delta^{9/10} \rho'^2.
\end{aligned}
\end{equation}

\subsubsection{Control of $J_5$.}

Recalling that $\psi_{A,B}=\chi_A^2\varphi_B$, using the H\"older inequality and Remark \ref{rem:chiA_zetaA4}, we get
\[
\begin{aligned}
|J_5|=& \left|\int  \psi_{A,B} F v_1 \right|\lesssim \| \chi_A\varphi_B v_1\|_{L^2}
\|\chi_A F \|_{L^2}\\
\lesssim & ~{} 
\| \chi_A\varphi_B v_1\|_{L^2}
\|\zeta_A^2 (1-\gamma\partial_x ^2)^{-1} \partial_x[u_1^3+3Hu_1^2] \|_{L^2}.
\end{aligned}
\]
First of all, since $|\varphi_{B}|\lesssim B$, we have
\begin{equation}\label{eq:estimates_psiB_v1}
\| \chi_A\varphi_{B} v_1\|_{L^2}\lesssim B \|\chi_A v_1\|_{L^2} \lesssim B \|v_1\|_{L^2(|y|<2A)}\lesssim B \|\zeta_A^2 v_1\|_{L^2}.
\end{equation}
Furthermore, 
using \eqref{eq:lem_sec_Iop g} and Lemma \ref{lem:estimates_IOp}, we have
\[
\begin{aligned}
\left\|\zeta_A^2 (1-\gamma\partial_x ^2)^{-1} \partial_x[u_1^3+3Hu_1^2] \right\|_{L^2}
& \lesssim \left\| \zeta_A^2 \partial_{x}[u_1^3+3Hu_1^2] \right\|_{L^2},
\end{aligned}
\]
and  expanding the above term, one gets
\begin{equation}\label{eq:LambaN}
\begin{aligned}
\left\| \zeta_A^2 \partial_{x}[u_1^3+3Hu_1^2] \right\|_{L^2}
&~{} = \left\| \zeta_A^2 [3u_1^2\partial_{x}u_{1}+3H'u_1^2+6Hu_1\partial_{x} u_{1}]\right\|_{L^2}\\
&~{}\leq  \|u_{1}\|_{L^{\infty}} \left\| \zeta_A^2 [3u_1 \partial_{x}u_{1}+3H'u_1 +6H\partial_{x} u_{1}] \right\|_{L^2}\\
 & ~{} \lesssim \| u_1\|_{L^\infty} \left[ \| w_1\|_{L^2} \| \zeta_A\partial_{x}u_{1}\|_{L^2}+\|w_1\|_{L^2} +\|\zeta_A\partial_{x} u_{1} \|_{L^2} \right]\\
  & ~{} \lesssim \left\| u_1\right\|_{L^\infty} \left[ \left\|w_1\right\|_{L^2} +\left\|\zeta_A\partial_{x} u_{1}\right\|_{L^2} \right].
\end{aligned}
\end{equation}
Finally, by \eqref{eq:LambaN}, \eqref{eq:estimates_psiB_v1} and \eqref{eq:kv1_w1}, we conclude
\begin{equation}\label{eq:estimates_J5}
\begin{aligned}
	|J_5| \lesssim &~{}
	B \|u_1\|_{L^\infty} \| \zeta_A^2 v_1\|_{L^2}
	\left[ \left\|w_1\right\|_{L^2}+\left\|\zeta_{A}\partial_{x} u_{1}\right\|_{L^{2}} \right] \\
	\lesssim &~{} B\delta \left[ \left\| w_{1}\right\|_{L^{2}}  +\gamma^{-1/2} \|\partial_{x} w_{1}\|_{L^{2}} \right]\big[ \|w_1\|_{L^2}+\|\zeta_{A}\partial_{x} u_{1}\|_{L^{2}} \big]\\
		\lesssim &~{} 
	\delta^{7/10}
	\left[ \left\|w_1\right\|_{L^2}^2+\left\|\partial_{x} w_{1}\right\|_{L^{2}}^2 \right].
\end{aligned}	
\end{equation}
In the last inequality we have used \eqref{eq:estimates_sech_u1_w1} combined with \eqref{eq:bound_ZCV_B}, and $B\delta \gamma^{-1/2} = \delta^{1-1/10 -2/10}=\delta^{7/10}$. 

\subsubsection{ Control of $J_6$ in \eqref{eq:J'_v}}
Replacing $\psi_{A,B}'$ from \eqref{eq:psi_deriv}, we have
\[
\begin{aligned}
|J_6| = &~{}  \left|\rho' \int \psi_{A,B}' v_1 v_2 \right|\\
 \leq &~{}  |\rho' | \left|\int (\chi_A^2\zeta_B^2+2\chi_A' \chi_A \varphi_B) v_1 v_2 \right| \\
=&~{} \rho' \int z_1 \chi_A\zeta_B v_2 +2 \rho' \int \chi_A' \chi_A \varphi_B v_1 v_2
=: J_{6,1}+J_{6,2}.
\end{aligned}
\]
Bound on $J_{6,1}$. Applying H\"older's inequality, estimate \eqref{eq:estimates_v2} and \eqref{new cota}, we get
\[
|J_{6,1}|\leq | \rho'|  \| z_1\|_{L^2} \|u_2\|_{L^2}\leq C_0\delta  | \rho'|   \| z_1\|_{L^2}.
\]
Bound on $J_{6,2}$. Here, after invoking \eqref{eq:bound_ZCV_B} and H\"older's inequality we get
\[
|J_{6,2}|\lesssim A^{-1}B |\rho'| \| \chi_A  v_1\|_{L^2} \|v_2\|_{L^2}.
\]
Using \eqref{eq:change_variable}, Lemma \ref{lem:estimates_IOp} and \eqref{new cota}, we easily have $\|v_2\|_{L^2} \lesssim \|u_2\|_{L^2} \lesssim \delta$. From \eqref{rem:chiA_zetaA4}, we also get
\[
\| \chi_A  v_1\|_{L^2} \lesssim  \| \zeta_A^2 v_1\|_{L^2} \lesssim  \| \zeta_A v_1\|_{L^2}.
\]
Hence, by \eqref{eq:kv1_w1},
\[
|J_{6,2}|\lesssim \delta  A^{-1}B  |\rho'|  \left(\|  w_1\|_{L^2} +\gamma^{-1/2}\|\partial_x w_1\|_{L^2} \right).
\]
Since $ A^{-1}B\gamma^{-1/2}=\delta^{7/10}$, after applying the Cauchy-Schwarz inequality we conclude that
\begin{equation}\label{eq:estimates_J6}
\begin{aligned}
|J_6|\lesssim &~{}  \delta   A^{-1}B  |\rho'| \left(\|  w_1\|_{L^2} +\gamma^{-1/2}\|\partial_x w_1\|_{L^2} \right) + C \delta |\rho'|   \| z_1\|_{L^2}\\
\lesssim &~{}  \delta^{17/10}  |\rho'| \left(\|  w_1\|_{L^2} + \|\partial_x w_1\|_{L^2} \right) + C \delta |\rho'|   \| z_1\|_{L^2}\\
\lesssim &~{}  \delta^{17/10}  \left(\|  w_1\|_{L^2}^2 +\|\partial_x w_1\|_{L^2}^2 \right) + C \delta  \| z_1\|_{L^2}^2
+C\delta |\rho'|^2 .
\end{aligned}
\end{equation}

\subsection{End of proof of Proposition \ref{prop:virial_J}}

From
 \eqref{eq:estimates_tildeJ1}, \eqref{eq:estimates_J4},  \eqref{eq:estimates_J5} and \eqref{eq:estimates_J6}, we obtain
 \[
 \begin{aligned}
 \big| \tilde{J_1}+J_4&+J_5+J_6 \big| \\
&~{} \lesssim
\delta^{1/2}\left( \| w_1\|_{L^2}^2 +   \| \partial_x w_1\|_{L^2}^2 +   \| w_2\|_{L^2}^2 \right)
+
\delta^{7/10} \left( \|w_1\|_{L^2}^2+\|\partial_{x} w_{1}\|_{L^{2}}^2 \right)
\\& ~{}\quad  
+ \delta^{17/10}  \left(\|  w_1\|_{L^2}^2 +\|\partial_x w_1\|_{L^2}^2 \right)
+
\delta^{1/10} \left( \|  w_1\|_{L^2}^2+  \|w_2\|_{L^2}^2 \right)
\\& ~{}\quad + \delta^{1/10} \left(  \| \partial_x z_2\|_{L^2}^2 + \|z_2\|_{L^2}^2 \right)
 +  \delta   \| z_1\|^2_{L^2}+\delta^{9/10} \rho'^2+ \delta \rho'^2 .
  \end{aligned}
 \]
Finally, simplifying we obtain that \eqref{eq:J1+J4+J5} is bounded as follows:
\begin{equation}\label{eq:J1+J4+J5}
	\begin{aligned}
	\left| \tilde{J}_1+J_4+J_5+J_6 \right|
		\lesssim & ~{} 
	\delta^{1/10} \left(\| w_1\|_{L^2}^2
	+\| \partial_{x} w_1\|_{L^2}^2
	+\| w_2\|_{L^2}^2
	+ \| z_2\|_{L^2}^2 + \| \partial_x z_2\|_{L^2}^2 \right)
	\\&
	+\delta \|z_1\|_{L^2}^2+\delta^{9/10} |\rho'|^2.
	\end{aligned}
\end{equation}
\medskip
Finally, the virial estimate is concluded as follows: for some $C_2>0$ independent of $\delta$ small, \eqref{eq:estimates_J} becomes
\[
\begin{aligned}
\frac{d}{dt}\mathcal J
\leq & ~{} -\frac12 \int  \left[z_1^2+(V_0-C_2 \delta^{1/10}) z_2^2 +2(\partial_x z_2)^2 \right]+C_2\delta \|z_1\|_{L^2}^2 \\
&+C_2 \delta^{1/10} \left(\| w_1\|_{L^2}^2+\| \partial_{x} w_1\|_{L^2}^2+\| w_2\|_{L^2}^2 \right)+C_2 \delta^{9/10} |\rho'|^2.
\end{aligned}
\]
This ends the proof of \eqref{eq:dJ} and Proposition \ref{prop:virial_J}.

\section{Gain of regularity}\label{sec:4}

We will focus now on a new virial estimate obtained from the new system of equations involving the variables 
\begin{equation}\label{tilde vi}
\tilde{v}_i=\partial_x v_i, \quad i=1,2.
\end{equation}
Formally taking derivatives in \eqref{eq:syst_v}, we have
\begin{equation}\label{eq:syst_vx}
\begin{cases}
\dot{\tv}_1=\partial_x \LL  \tv_2+\tilde{G}, \qquad \tilde{G}=\partial_x G,\\
\dot{\tv}_2= \partial_x \tv_1
+\tilde{F},\qquad \tilde{F}=\partial_x F,
\end{cases}
\end{equation}
where $G$ and $F$ are given in \eqref{eq:G_H} and $\tilde v_i$ given in \eqref{tilde vi}. For this new system, we consider the virial functional
\begin{equation}\label{eq:new_virials}
\mathcal{M}(t)=\int \phi_{A,B} (y)\tv_1 (t,x)\tv_2(t,x)dx= \int \phi_{A,B}(y) \partial_x v_1 (t,x)\partial_x  v_2(t,x)dx.
\end{equation}
Later we will choose $\phi_{A,B}=\psi_{A,B}=\chi_A^2 \varphi_B $ (see \eqref{eq:def_psiB}). 

\subsection{A virial estimate related to $\M$}
\begin{lem}\label{lem:identity_dtM}
	Let $(v_1,v_2)\in H^1(\R)\times H^2(\R)$ be a solution of \eqref{eq:syst_v}. Consider $\phi_{A,B}$ an odd smooth bounded function to be a choose later. Then
	\begin{equation}
	\begin{aligned}\label{eq:M'}
	\frac{d}{dt}\mathcal{M}
	=&-\dfrac{1}{2}\int  \phi_{A,B}'  \left[(\partial_x v_1)^2+V_0 (\partial_x v_2)^2+3(\partial_x^2 v_2)^2\right] 
	+\dfrac{1}{2}\int  \phi_{A,B}''' (\partial_x v_2)^2 \\
	&+\dfrac{1}{2}\int \phi_{A,B} V_0' (\partial_x v_2)^2
	+\int  \phi_{A,B}\big[ \partial_x {G} \partial_x v_2
	+\tilde{F} \partial_x v_1\big] - \rho'  \int \phi_{A,B}' \partial_x v_1\partial_x v_2.
	\end{aligned}
	\end{equation}
\end{lem}

\begin{proof}[Proof of Lemma \ref{lem:identity_dtM}]
We compute from \eqref{eq:new_virials}:
\[
\frac{d}{dt}\mathcal{M} =\int \phi_{A,B} \tv_{1,t}\tv_2+\int \phi_{A,B}\tv_1 \tv_{2,t}  -\rho'\int \phi_{A,B}'\tv_1 \tv_2.
\]
	From \eqref{eq:syst_vx} and \eqref{eq:PL}, we have
	\[
	\begin{aligned}
	\frac{d}{dt}\mathcal{M} =&\int \phi_{A,B} \left( \partial_x \LL  \tv_2+\tilde{G} \right) \tv_2+\int \phi_{A,B}\tv_1 \left(  \partial_x \tv_1+\tilde{F}\right)  -\rho'\int \phi_{A,B}' \partial_x v_1\partial_x v_2 \\
	=&-\dfrac{1}{2}\int  \phi_{A,B}'  \left[\tv_1^2+V_0 \tv_2^2+3(\partial_x \tv_2)^2\right] 
	+\dfrac{1}{2}\int  \phi_{A,B}''' \tv_2^2 \\
	&+\dfrac{1}{2}\int \phi_{A,B} V_0' \tv_2^2
	+\int  \phi_{A,B}\big[ \tilde{G} \tv_2
	+\tilde{F} \tv_1\big]- \rho'  \int \phi_{A,B}' \partial_x v_1\partial_x v_2 .
	\end{aligned}
	\]
	Rewriting the above identity in term of the variables $(v_1,v_2)$ and using the definition of $\tilde{G}$, we have
	\[
	\begin{aligned}
	\frac{d}{dt}\mathcal{M}
	=&-\dfrac{1}{2}\int  \phi_{A,B}'  \left[(\partial_x v_1)^2+V_0 (\partial_x v_2)^2+3(\partial_x^2 v_2)^2\right] 
	+\dfrac{1}{2}\int  \phi_{A,B}''' (\partial_x v_2)^2 \\
	&+\dfrac{1}{2}\int \phi_{A,B} V_0' (\partial_x v_2)^2
	+\int  \phi_{A,B}\big[ \partial_x {G} \partial_x v_2
	+\tilde{F} \partial_x v_1\big]  - \rho'  \int \phi_{A,B}' \partial_x v_1\partial_x v_2.
	\end{aligned}
	\]
	This ends the proof of Lemma \ref{lem:identity_dtM}.
\end{proof}

The following proposition connects two virial identities in the variable $(z_1,z_2)$. 
Finally, let $z_i$, $i=1,2$ be as in \eqref{eq:def_psiB}.

\begin{prop}\label{prop:ineq_dtM}
	There exist $C_3>0$ and $\delta_3>0$ such that  for any $0<\delta\leq\delta_3$, the following holds.  Assume that for all $t\geq 0$, \eqref{eq:ineq_hip} and \eqref{Fijacion2} hold. Then for all $t\geq 0$, 
	\begin{equation}\label{eq:dM_z}
	\begin{aligned}
	 \frac{d}{dt}\M 
\leq&
	-\dfrac{1}{2}\int  \left[ (\partial_x z_1)^2 
	+2 (\partial_x^2 z_2)^2
	\right]+C_3 (\| z_1\|_{L^2}^2 +\| z_2\|^2_{L^2}+\|\partial_x z_2\|^2_{L^2}) \\
	&+C_3\delta^{1/10} \big(\|  w_1\|^2_{L^2} + \|   \partial_x w_1\|_{L^2}^2+\| w_2\|^2_{L^2} \big)
+C_3\delta^{7/10}  |\rho'|^2.
	\end{aligned}
	\end{equation}
	\end{prop}
The proof of the above result requires some technical estimates. We first state them, to then prove Proposition \ref{prop:ineq_dtM} (Subsection \ref{ahorasi}). 	
{\color{black}
\subsection{Second set of technical estimates} Recall the following technical estimates on the variables $\zeta_B$ and other related error terms. These estimates had been proved and used in \cite{M_GB},
 therefore we only enunciate the main results.

\begin{claim}\label{claim:R_CZ_vi_xx}
	Let  $R$ be a $W^{2,\infty}(\R)$ function, $R=R(y)$, $v_i$ be as in \eqref{eq:change_variable}, and $z_i$ be as in \eqref{eq:def_psiB}. Then
	\begin{equation*}
	\begin{aligned}
	\int R \chi_A^2\zeta_B^2 (\partial_x^2 v_i)^2
	=&
	\int R(\partial_x^2 z_i)^2+\int \tilde{R} z_i^2
	+\int P_R (\partial_x z_i)^2\\ 
	&+
	\int    \left[ P_R' \frac{\zeta_B'}{\zeta_B}+P_R
	\frac{\zeta_B''}{\zeta_B}\right] z_i^2 +\int \mathcal{E}_2 \zeta_B^2v_i^2\\
	&+\int  \mathcal{E}_1\zeta_B^2 v_i^2
	+\int \mathcal{E}_3 \zeta_B^2(\partial_x v_i)^2,
	\end{aligned}
	\end{equation*}
	where
	\begin{equation}\label{eq:claimtildeR}	
	\begin{aligned}
	\tilde{R}=\tilde{R}_R=&	  
-2R\left[\frac{\zeta_B^{(4)}}{\zeta_B}+\frac{\zeta_B'''}{\zeta_B}\frac{\zeta_B'}{\zeta_B}\right]-2	R'\frac{\zeta_B'''}{\zeta_B}
- R''\frac{\zeta_B''}{\zeta_B},
	\end{aligned}
	\end{equation}
	\begin{equation}\label{eq:claimPR}
	P_R=
	R 
	\bigg[
	4  \frac{\zeta_B''}{ \zeta_B}  
	-2 \left(\frac{\zeta_B'}{\zeta_B} \right)^2\bigg]
	+2R' \frac{\zeta_B'}{\zeta_B},
	\end{equation}
	$\mathcal{E}_1$ is defined in \eqref{eq:claimE1}, 
		\begin{equation}\label{eq:claimE2}	
	\begin{aligned}
	\mathcal{E}_2 =\mathcal{E}_2(R)
=&
-R
\left(
\chi_A^{(4)} \chi_A 
+4\chi_A''' \chi_A \frac{\zeta_B'}{\zeta_B^2}
+6 \chi_A''\chi_A\frac{\zeta_B''}{\zeta_B}
+2(\chi_A^2)' \frac{\zeta_B'''}{\zeta_B}  
\right)\\
&
-R'
\left(
2\chi_A''' \chi_A 
+6\chi_A'' \chi_A  \frac{\zeta_B'}{\zeta_B}
+6
\chi_A' \chi_A \frac{\zeta_B''}{\zeta_B}
\right) \\&
- R''
\left(
\chi_A'' \chi_A +\frac12 (\chi_A^2)' \frac{\zeta_B'}{\zeta_B}
\right), 
\end{aligned}
\end{equation}
and
\begin{equation}\label{eq:claimE3}
\begin{aligned}
\mathcal{E}_3 =\mathcal{E}_3(R)
= R\bigg[ 
4\chi_A'' \chi_A -2(\chi_A' )^2+ 2\frac{\zeta_B'}{\zeta_B}(\chi_A^2)'  
\bigg]
+R'(\chi_A^2)' .
	\end{aligned}
	\end{equation}
Finally, $P_R$, $\mathcal{E}_2$ and $\mathcal{E}_3$ satisfy the following bounds
\begin{equation}\label{eq:PR_E23}
\begin{aligned}
	|P_R| \lesssim & ~{}  		B^{-1}\|R' \|_{L^{\infty}}  + B^{-1}\| R\|_{L^\infty},   \\
	|P_R'| \lesssim & ~{}  	B^{-1}\|R'' \|_{L^{\infty}}	+B^{-1}\|R' \|_{L^{\infty}}  + B^{-1}\| R\|_{L^\infty},   \\
	|\mathcal{E}_2| \lesssim & ~{} (AB)^{-1}\|R'' \|_{L^{\infty}(A\leq |y|\leq 2A)}+(AB^2)^{-1}\|R' \|_{L^{\infty}(A\leq |y|\leq 2A)} \\
	&~{} +(AB^3)^{-1}\|R \|_{L^{\infty}(A\leq |y|\leq 2A)},\\
	|\mathcal{E}_3| \lesssim & ~{}  A^{-1}\|R' \|_{L^{\infty}(A\leq |y|\leq 2A)}+(AB)^{-1}\|R \|_{L^{\infty}(A\leq |y|\leq 2A)}.
\end{aligned}
\end{equation}
\end{claim}

\begin{rem}
Estimates \eqref{eq:PR_E23} are of technical type, needed in some particular stage of the proof of a transfer of regularity estimate \eqref{eq:n2}. A proof of these estimates is present in \cite{M_GB}, but without the localization terms. Here we slightly improve these estimates by considering the region of space where these functions are supported.
\end{rem}

\begin{rem}\label{rem:CZ_v2xx}
	For further purposes, we shall need the previous identities in the simplest case $R=1$.  We obtain
		\begin{equation}\label{eq:1_CZ_v2_xx}
	\begin{aligned}
	\int \chi_A^2\zeta_B^2 (\partial_x^2 v_i)^2
	=&
	\int (\partial_x^2 z_i)^2+\int \tilde{R}_1z_i^2
	+\int P_1 (\partial_x z_i)^2 
	+
	\int    \left[ P_1' \frac{\zeta_B'}{\zeta_B}+P_1 
	\frac{\zeta_B''}{\zeta_B}\right] z_i^2 \\
	&+\int \mathcal{E}_2 \zeta_B^2v_i^2
	+\int  \mathcal{E}_1\zeta_B^2 v_i^2
	+\int \mathcal{E}_3 \zeta_B^2(\partial_x v_i)^2,
	\end{aligned}
	\end{equation}
	where (see \eqref{eq:claimtildeR} and \eqref{eq:claimPR}),
	\begin{equation}\label{eq:tildeR1_P1}	
		\tilde{R}_1=	  
			-2\left[\frac{\zeta_B^{(4)}}{\zeta_B}+\frac{\zeta_B'''}{\zeta_B}\frac{\zeta_B'}{\zeta_B}\right]
	,\ \quad 
	P_1=
	4  \frac{\zeta_B''}{ \zeta_B}  
	-2 \left(\frac{\zeta_B'}{\zeta_B} \right)^2,
		\end{equation}
	$\mathcal{E}_1$ is defined in \eqref{eq:claimE1}, $\mathcal E_2$ in \eqref{eq:claimE2} becomes
\begin{equation}\label{eq:claimE2_1}	
\begin{aligned}
\mathcal{E}_2(1)
=&
-
\left[
\chi_A^{(4)} \chi_A 
+4\chi_A''' \chi_A \frac{\zeta_B'}{\zeta_B^2}
+6 \chi_A''\chi_A\frac{\zeta_B''}{\zeta_B}
+2(\chi_A^2)' \frac{\zeta_B'''}{\zeta_B}  
\right],
	\end{aligned}
	\end{equation}
	and \eqref{eq:claimE3} reads now
	\begin{equation}\label{eq:claimE3_1}
	\begin{aligned}
	\mathcal{E}_3 (1)
=
4\chi_A'' \chi_A -2(\chi_A' )^2+ 2\frac{\zeta_B'}{\zeta_B}(\chi_A^2)'  .
	\end{aligned}
	\end{equation}
	Finally, by \eqref{eq:z'/z} and \eqref{eq:C*z'/z}, we obtain the simplified estimate
	\[
	\begin{aligned}
	\| \chi_A\zeta_B \partial_x^2 v_i\| 
	\lesssim & ~{}  \| \partial_x^2 z_i\|_{L^2}^2+B^{-1}\| \partial_x z_i\|_{L^2}^2+B^{-1}\| z_i\|_{L^2}^2 \\
	&~{} 
	+(AB)^{-1} \left( \| \zeta_{B}  v_i\|_{L^2}^2 + \| \zeta_{B} \partial_{x} v_i\|_{L^2}^2 \right).
	\end{aligned}
	\]
\end{rem}

\subsection{Start of proof of Proposition \ref{prop:ineq_dtM}}\label{ahorasi}
The proof of this result needs the following computation:
\begin{lem}\label{lem:virial_M}
	Let $(v_1,v_2)\in H^1(\R)\times H^2(\R)$ a solution of \eqref{eq:syst_v}. Consider $\phi_{A,B}=\psi_{A,B}=\chi_A^2 \varphi_B $. Then
	\begin{equation}\label{dtM}
	\begin{aligned}
	\frac{d}{dt}\mathcal{M}
	=&
	-\dfrac{1}{2}\int  \left[ (\partial_x z_1)^2 
	+ V_0 (\partial_x z_2)^2 
	+3 (\partial_x^2 z_2)^2
	\right]
	+\mathcal{R}_{z}
	+\mathcal{R}_{v}
	+\mathcal{DR}_{v}\\
		&+\dfrac{1}{2}\int \phi_{A,B} V_0' (\partial_x v_2)^2
	+\int  \phi_{A,B}\big[ \partial_x{G} \partial_x v_2
	+\tilde{F} \partial_x v_1\big]
	~{} -\rho'  \int \phi_{A,B}' \partial_x v_1\partial_x v_2,
			\end{aligned}
	\end{equation}
	where $\mathcal{R}_{z}(t)$, $\mathcal{R}_{v}(t)$ and $\mathcal{DR}_{v}(t)$ are error terms that, under the constrain \eqref{Fijacion2}, satisfy the following bound
	\begin{equation}\label{eq:Rz+Rv+DRv_M}
		\begin{aligned}
		\left|\mathcal{R}_{z}+\mathcal{R}_{v}+\mathcal{DR}_{v} \right| \lesssim & ~{}  
 \delta^{1/10} \big(\|  w_1\|^2_{L^2} + \|   \partial_x w_1\|_{L^2}^2+\| w_2\|^2_{L^2} +\| z_1\|_{L^2}^2 +\| z_2\|_{L^2}^2+\| \partial_x z_2\|_{L^2}^2\big).
		\end{aligned}
	\end{equation}
\end{lem}
\begin{proof}
	First, we recall that $z_i=\chi_A \zeta_B v_i$, and
	by \eqref{eq:psi_deriv} and Remark \ref{rem:CZ_v2x}, one gets
	\begin{equation}\label{eq:m1}
	\begin{aligned}
	\int  \phi_{A,B}'  ( \partial_x v_1)^2
	=&
	\int  (\partial_x z_1)^2
	+\int  \dfrac{\zeta_B''}{\zeta_B} z_1^2 
	+ \int  \mathcal{E}_1\zeta_B^2 v_1^2
	+\int  (\chi_A^2)'\varphi_B  ( \partial_x v_1)^2,
	\end{aligned}
	\end{equation}
	where $ \mathcal{E}_1$ is given by  \eqref{eq:E1_1}. For the second term of \eqref{eq:M'}, applying \eqref{eq:claimP_CZ_B_vix_final},  we obtain
	
	\[
	\begin{aligned}
	\int \phi_{A,B}' V_0 (\partial_x v_2)^2 
	=&	\int V_0 (\partial_x z_2)^2 +
	\int    \left[ V_0' \frac{\zeta_B'}{\zeta_B}+V_0 
	\frac{\zeta_B''}{\zeta_B}\right] z_2^2\\
	&+\int  \mathcal{E}_1(V_0)\zeta_B^2 v_2^2+\int (\chi_A^2)'\varphi_B V_0 (\partial_x v_2)^2 ,
	\end{aligned}
	\]
	where following \eqref{eq:claimE1},
	\[
	\begin{aligned}
	\mathcal{E}_1(V_0)
	= V_0   \left[ \chi_A''\chi_A+(\chi_A^2)' \frac{\zeta_B'}{\zeta_B}\right] 
+ \frac12 V_0'(\chi_A^2)' .
	\end{aligned}
	\]
	Now, using \eqref{eq:1_CZ_v2_xx} in Remark \ref{rem:CZ_v2xx}, we get
	\begin{equation}\label{eq:m2}
	\begin{aligned}
	\int  \phi_{A,B}'  (\partial_x^2 v_2)^2  
	=&
	\int (\partial_x^2 z_2)^2
	+\int P_1 (\partial_x z_2)^2
	+
	\int    \left[\tilde{R}_1+ P_1' \frac{\zeta_B'}{\zeta_B}+P_1
	\frac{\zeta_B''}{\zeta_B}\right] z_2^2 \\
	&+\int \mathcal{E}_2(1) \zeta_B^2v_2^2
	+\int  \mathcal{E}_1(P_1)\zeta_B^2 v_2^2
	+\int \mathcal{E}_3(1) \zeta_B^2(\partial_x v_2)^2\\
	&+\int  (\chi_A^2)'\varphi_B (\partial_x^2 v_2)^2 ,
		\end{aligned}
	\end{equation}
		where $\tilde{R}_1, P_1, \mathcal{E}_2(1), \mathcal{E}_3(1) $ are given by  \eqref{eq:tildeR1_P1}, \eqref{eq:claimE2_1},  \eqref{eq:claimE3_1}, and $ \mathcal{E}_1$ is given by \eqref{eq:claimE1}.\\
	Now, continuing with the second integral in the RHS of \eqref{eq:M'}, we have
	\begin{equation*}
	\begin{aligned}
	\int  \phi_{A,B}''' (\partial_x v_2)^2 
	= \int   \frac{(\zeta_B^2)''}{\zeta_B^2} \chi_A^2 \zeta_B^2(\partial_x v_2)^2 + \int  \left[6(\chi_A^2)' \frac{\zeta_B'}{\zeta_B}+3(\chi_A^2)'' +(\chi_A^2)'''\frac{\varphi_B}{\zeta_B^2} \right]\zeta_B^2 (\partial_x v_2)^2, 
	\end{aligned}
	\end{equation*}
	and using Claim  \ref{claim:P_CZ_vi_x},
	\begin{equation}\label{eq:m3}
	\begin{aligned}
	\int  \phi_{A,B}''' (\partial_x v_2)^2  
	=&
	\int   \frac{(\zeta_B^2)''}{\zeta_B^2}  (\partial_x z_2)^2 
	+ \int   \frac{(\zeta_B^2)''}{\zeta_B^2}   \frac{\zeta_B''}{\zeta_B}  z_2^2
	+ \int  \left( \frac{(\zeta_B^2)''}{\zeta_B^2} \right)'  \frac{\zeta_B'}{\zeta_B}  z_2^2\\
	&+ \int   \frac{(\zeta_B^2)''}{\zeta_B^2}  \bigg[\chi_A''\chi_A +(\chi_A^2)' \frac{\zeta_B'}{\zeta_B}\bigg] \zeta_B^2 v_2^2
	+\frac12 \int  \left( \frac{(\zeta_B^2)''}{\zeta_B^2} \right)' (\chi_A^2)'  \zeta_B^2   v_2^2\\
	&+ \int  \left[ 6(\chi_A^2)' \frac{\zeta_B'}{\zeta_B}+3(\chi_A^2)'' +(\chi_A^2)'''\frac{\varphi_B}{\zeta_B^2} \right] \zeta_B^2 (\partial_x v_2)^2. 
	\end{aligned}
	\end{equation}
Collecting \eqref{eq:m1}, \eqref{eq:m2} and \eqref{eq:m3}, we obtain 
	\begin{equation*}
	\begin{aligned}
	\frac{d}{dt}\mathcal{M}
	=&
	-\dfrac{1}{2}\int  \left[ (\partial_x z_1)^2 
	+ V_0 (\partial_x z_2)^2 
	+3 (\partial_x^2 z_2)^2
	\right]
	+\mathcal{R}_{z}
	+\mathcal{R}_{v}
	+\mathcal{DR}_{v}\\
		&+\dfrac{1}{2}\int \phi_{A,B} V_0' (\partial_x v_2)^2
	+\int  \phi_{A,B}\big[ \partial_x{G} \partial_x v_2
	+\tilde{F} \partial_x v_1\big] -\rho'  \int \phi_{A,B}' \partial_x v_1\partial_x v_2,
			\end{aligned}
	\end{equation*}
		where the error terms are the following: associated to $(z_1,z_2)$ is
		\begin{equation}\label{R_z}
	\begin{aligned}
	\mathcal{R}_{z}=&-\dfrac{1}{2}\int  \dfrac{\zeta_B''}{\zeta_B} z_1^2 
	-\dfrac{1}{2}\int    \left[ V_0' \frac{\zeta_B'}{\zeta_B}+V_0 
	\frac{\zeta_B''}{\zeta_B}\right] z_2^2
	-\dfrac{3}{2}
	\int    \left[\tilde{R}_1+ P_1' \frac{\zeta_B'}{\zeta_B}+P_1 
	\frac{\zeta_B''}{\zeta_B}\right] z_2^2\\
	&
	+ \frac{1}{2} \int   \frac{(\zeta_B^2)''}{\zeta_B^2}   \frac{\zeta_B''}{\zeta_B}  z_2^2
	+ \frac{1}{2}\int  \left[ \frac{(\zeta_B^2)''}{\zeta_B^2} \right]'  \frac{\zeta_B'}{\zeta_B}  z_2^2
	+\frac12 \int \left[ \frac{(\zeta_B^2)''}{\zeta_B^2}  -3P_1 \right] (\partial_x z_2)^2,
			\end{aligned}
	\end{equation}
	 associated to $(v_1,v_2)$ is
			\begin{equation}\label{eq:Rv}
	\begin{aligned}
	\mathcal{R}_{v}=&-\dfrac{1}{2} \int  \mathcal{E}_1(1)\zeta_B^2 v_1^2
	-\dfrac{1}{2}\int  \mathcal{E}_1(V_0)\zeta_B^2 v_2^2
		+\frac14 \int   \frac{(\zeta_B^2)''}{\zeta_B^2}  \bigg[\chi_A''\chi_A +2(\chi_A^2)' \frac{\zeta_B'}{\zeta_B}\bigg] \zeta_B^2 v_2^2\\
	&+\frac14 \int  \left[ \frac{(\zeta_B^2)''}{\zeta_B^2} \right]' (\chi_A^2)'  \zeta_B^2   v_2^2
	-\dfrac{3}{2}\int \mathcal{E}_2(1) \zeta_B^2 v_2^2
	-\dfrac{3}{2}\int  \mathcal{E}_1(P_1)\zeta_B^2 v_2^2,
				\end{aligned}
	\end{equation}
	 	and  associated to $(\partial_x v_1, \partial_x v_2)$ is
			\begin{equation}\label{eq:DRv}
	\begin{aligned}
	\mathcal{DR}_{v}=	&-\dfrac{3}{2}\int \mathcal{E}_3(1) \zeta_B^2(\partial_x v_2)^2
	-\dfrac{1}{2}\int  (\chi_A^2)'\varphi_B  ( \partial_x v_1)^2
	-\dfrac{1}{2}\int (\chi_A^2)'\varphi_B V_0 (\partial_x v_2)^2 \\
	&+\frac{1}{2} \int  \left[ 6(\chi_A^2)' \frac{\zeta_B'}{\zeta_B}+3(\chi_A^2)'' +(\chi_A^2)'''\frac{\varphi_B}{\zeta_B^2} \right] \zeta_B^2 (\partial_x v_2)^2
		-\dfrac{3}{2}\int  (\chi_A^2)'\varphi_B (\partial_x^2 v_2)^2 .
		\end{aligned}
	\end{equation}

	We have obtained the identity \eqref{dtM}.
To conclude the proof of Lemma \ref{lem:virial_M}, we must estimate the error terms.

\subsection{Controlling error terms}
Recall $\mathcal{R}_{z}(t)$ from \eqref{R_z}. Decompose
\[
\mathcal{R}_{z}(t) =\mathcal{R}^{1}_{z}(t) + \mathcal{R}^{2}_{z}(t) +\mathcal{R}^{3}_{z}(t),
\]
where
	\begin{equation*}
	\begin{aligned}
	\mathcal{R}^1_{z}=& -\dfrac{1}{2}\int  \dfrac{\zeta_B''}{\zeta_B} \left[ z_1^2 +V_0 z_2^2 \right]-\dfrac{1}{2}\int     V_0' \frac{\zeta_B'}{\zeta_B} z_2^2, \\
	\mathcal{R}^2_{z}=&-\dfrac{3}{2} \int    \left[\tilde{R}_1 + P_1' \frac{\zeta_B'}{\zeta_B}+P_1 \frac{\zeta_B''}{\zeta_B}\right] z_2^2 +\frac12 \int \left[ \frac{(\zeta_B^2)''}{\zeta_B^2}  -3P_1\right] (\partial_x z_2)^2, \\
	\mathcal{R}^3_{z}=& \frac{1}{2} \int   \left[ \frac{(\zeta_B^2)''}{\zeta_B^2}   \frac{\zeta_B''}{\zeta_B}  
	+   \left( \frac{(\zeta_B^2)''}{\zeta_B^2} \right)'  \frac{\zeta_B'}{\zeta_B}\right]  z_2^2.
	\end{aligned}
	\end{equation*}
In the case of $\mathcal{R}^1_{z}$, recalling the estimate \eqref{eq:z'/z}, we obtain
\begin{equation} \label{eq:R1z}
\begin{aligned} 
\left|\mathcal{R}^1_{z} \right|
\lesssim B^{-1}( \| z_1\|_{L^2}^2 +  \| z_2\|_{L^2}^2).
\end{aligned}
\end{equation}
As for $\mathcal{R}^2_{z}$, we  recall the form of  $P_1$ (see  \eqref{eq:tildeR1_P1}),  $\tilde{R}_1$ (see \eqref{eq:tildeR1_P1}), 
and 
by  \eqref{eq:z'/z}, we conclude
\begin{equation} \label{eq:R2z}
\begin{aligned}
|\mathcal{R}^2_{z}|
\lesssim B^{-1} (\| z_2\|_{L^2}^2 +\| \partial_x z_2\|_{L^2}^2).
\end{aligned}
\end{equation}

\medskip

\noindent
Now we consider the term $\mathcal{R}^3_{z}$. From \eqref{eq:C*z'/z} we obtain
\begin{equation} \label{eq:R3z}
\left| \mathcal{R}^3_{z} \right| \lesssim B^{-1} \| z_2\|_{L^2}^2.
\end{equation}
Collecting \eqref{eq:R1z}, \eqref{eq:R2z} and \eqref{eq:R3z}, and considering \eqref{Fijacion2}, we finally get
\begin{equation}\label{eq:Bound_Rz}
\left| \mathcal{R}_z \right|\lesssim  \delta^{1/10} \left( \|z_1\|_{L^2}^2+  \|z_2\|_{L^2}^2+ \|\partial_x z_2\|_{L^2}^2\right).
\end{equation}

\subsubsection*{Controlling $\mathcal{R}_{v}$} For the term $\mathcal{R}_{v}$ given by \eqref{eq:Rv}, we consider the following decomposition
	\begin{equation*}
	\begin{aligned}	
	\mathcal{R}^1_{v}=& -\dfrac{1}{2} \int  \mathcal{E}_1(1)\zeta_B^2 v_1^2
	-\dfrac{1}{2}\int  \left( \mathcal{E}_1(V_0)
	+3 \mathcal{E}_2(1) 
	+3\mathcal{E}_1(P_1) \right)\zeta_B^2 v_2^2,\\
	\mathcal{R}^2_{v}=& \frac14 \int   \frac{(\zeta_B^2)''}{\zeta_B^2}  \bigg[\chi_A''\chi_A +2(\chi_A^2)' \frac{\zeta_B'}{\zeta_B}\bigg] \zeta_B^2 v_2^2
	+\frac14 \int  \left( \frac{(\zeta_B^2)''}{\zeta_B^2} \right)' (\chi_A^2)'  \zeta_B^2   v_2^2.
	\end{aligned}
	\end{equation*}
We note that the terms $\mathcal{E}_{1}(P_1)$ and $\mathcal{E}_{2}(1)$  (see  \eqref{eq:claimE1}, \eqref{eq:tildeR1_P1} and \eqref{eq:claimE2_1}), by \eqref{eq:C*z'/z} and \eqref{eq:PR_E23}, are bounded and satisfy the following estimates:
\[
|\mathcal{E}_{2}(1)|\lesssim (AB^3)^{-1},  \ \mbox{and} \ \
|\mathcal{E}_{1}(P_1)| \lesssim (AB^3)^{-1},
\]
and for $\mathcal{E}_1(1)$ in \eqref{eq:E1_1} and $\mathcal{E}_1(V_0)$ in \eqref{cota_final} (replacing $P$ by $V_0$),
\[
|\mathcal{E}_1(1)|\lesssim (AB)^{-1}, \quad \quad |\mathcal{E}_1(V_0)|\lesssim (AB)^{-1}.
\]
Consequently,
\[
|\mathcal{R}^1_{v}| \lesssim \frac1{AB} \left( \|\zeta_B v_1 \|_{L^2}^2 + \|\zeta_B v_2\|_{L^2}^2 \right).
\]
Also,
\[
|\mathcal{R}^2_{v}| \lesssim \frac1{AB} \left( \|\zeta_B v_1 \|_{L^2}^2 + \|\zeta_B v_2\|_{L^2}^2 \right).
\]
Then,  applying \eqref{eq:kv1_w1} and \eqref{eq:estimates_v2_w2}, we have
\[
\begin{aligned}
|\mathcal{R}_{v}| \lesssim&~{} (AB)^{-1} \left( \| \zeta_B v_1\|_{L^2}^2+\| \zeta_B v_2\|_{L^2}^2 \right)\\
\lesssim&~{} (AB)^{-1}\gamma^{-1} \left( \| w_1\|_{L^2}^2+ \| \partial_x w_1\|_{L^2}^2+\| w_2\|_{L^2}^2 \right),
\end{aligned}
\]
By \eqref{Fijacion2}, we have $(AB)^{-1}\gamma^{-1} =\delta^{7/10}$. Then, we get
\begin{equation}\label{eq:Bound_Rv}
|\mathcal{R}_{v}| \lesssim~{}  \delta^{7/10}   \left( \| w_1\|^2_{L^2} +\| \partial_x w_1\|^2_{L^2} +\|w_2\|^2_{L^2}\right).
\end{equation}

\subsubsection*{Controlling $\mathcal{DR}_v$}  In the case of the term $\mathcal{DR}_v$ given by \eqref{eq:DRv}, first of all we have from \eqref{eq:claimE3_1} and \eqref{eq:bound_ZCV_B}:
\[
\left| \mathcal{E}_3(1) \right| \lesssim \frac1{AB}.
\]
Using again \eqref{eq:bound_ZCV_B}  and \eqref{rem:chiA_zetaA4}, we have  
\[
\begin{aligned}
|\mathcal{DR}_v|\lesssim &~{} \int \left| \mathcal{E}_3(1) \right| \zeta_B^2(\partial_x v_2)^2 + \int  (\chi_A^2)' |\varphi_B|  ( \partial_x v_1)^2 + \int (\chi_A^2)' |\varphi_B V_0 | (\partial_x v_2)^2 \\
	&~{} + \int  \left|  6(\chi_A^2)' \frac{\zeta_B'}{\zeta_B}+3(\chi_A^2)'' +(\chi_A^2)'''\frac{\varphi_B}{\zeta_B^2} \right| \zeta_B^2 (\partial_x v_2)^2 + \int  \left| (\chi_A^2)'\varphi_B \right| (\partial_x^2 v_2)^2 \\
\lesssim &~{} BA^{-1}\left(\| \zeta_A^2 \partial_x v_1\|_{L^2}^2+\|\zeta_A^2\partial_x v_2\|^2_{L^2}
+\|\zeta_A^2\partial_x^2 v_2\|_{L^2}^2 \right).
\end{aligned}
\]
Now, applying \eqref{eq:estimates_v1_w1} and \eqref{eq:estimates_v2_w2} with $K=A$; and by \eqref{Fijacion2}, we get
\begin{equation}\label{eq:Bound_DRv}
\begin{aligned}
\left| \mathcal{DR}_v \right|\lesssim 
&~{} \delta^{1/10} \left(\| w_1\|_{L^2}^2+\| \partial_x w_1\|_{L^2}^2+\| w_2\|^2_{L^2}  \right).
\end{aligned}
\end{equation}
And, by \eqref{eq:Bound_Rz}, \eqref{eq:Bound_Rv} and \eqref{eq:Bound_DRv}, we obtain
\begin{equation}\label{eq:bound_errorM}
\begin{aligned}
& |\mathcal{R}_z +\mathcal{R}_v +\mathcal{DR}_v|\\
 &~{}\lesssim
 \delta^{1/10}\left(\| w_1\|_{L^2}^2+\| \partial_x w_1\|_{L^2}^2+\| w_2\|^2_{L^2}  
+ \| z_1\|_{L^2}^2 +\| z_2\|_{L^2}^2+\| \partial_x z_2\|_{L^2}^2\right),
 \end{aligned}
\end{equation}
which proves \eqref{eq:Rz+Rv+DRv_M}. This ends the proof of Lemma \ref{lem:virial_M}.
\end{proof}

\subsection{Controlling nonlinear terms} Recall the second line of error terms in \eqref{dtM}:
\[
	\begin{aligned}
	\dfrac{1}{2}\int \phi_{A,B} V_0' (\partial_x v_2)^2
	+\int  \phi_{A,B}\big[ \partial_x{G} \partial_x v_2
	+\tilde{F} \partial_x v_1\big]
	 -\rho'  \int \phi_{A,B}' \partial_x v_1\partial_x v_2.
	\end{aligned}
\]
In order to control them in a well-ordered fashion, we set
\begin{equation}\label{eq:M0123}
\begin{aligned}
 \mathcal{M}_0= \frac12 \int \phi_{A,B} V_0' (\partial_x v_2)^2,
 \qquad 
 \mathcal{M}_2 =  \int \phi_{A,B} \partial_x G \partial_x v_2,\\
 \qquad 
	 \mathcal{M}_1 =\int \phi_{A,B} \partial_x F \partial_x v_1,
 \qquad 	 
 \mathcal{M}_3= -\rho'  \int \phi_{A,B}' \partial_x v_1\partial_x v_2.
 \end{aligned}
\end{equation}

\subsubsection{Control of $\mathcal{M}_0$.} Noticing that $V_0' = 6HH'$, $|V_0'| \lesssim H'$, $\phi_{A,B}=\chi_A^2 \varphi_B $, and that 
\[
|H'(y) \varphi_B(y)|\lesssim \left|y \sech^2\left( \frac{y}{\sqrt{2}}\right)\right|\lesssim e^{-|y|},
\] 
we get
\[
 \frac12\int \phi_{A,B} V_0' (\partial_x v_2)^2\lesssim \|e^{-|y|/2}\partial_x v_2\|_{L^2}^2.
\]
Following a decomposition similar to that in \eqref{eq:exp_vx}, using \eqref{eq:estimates_v2_w2}, and recalling the values for $B$ and $\gamma$ in \eqref{Fijacion2}, one gets
\begin{equation}\label{eq:m0}
\begin{aligned}
|\mathcal{M}_0| = &~{} \left|\frac12 \int \phi_{A,B} V_0' (\partial_x v_2)^2 \right| \\
\lesssim &~{} e^{-A/2} \|  \zeta_B \partial_x v_2\|_{L^2}^2 + \| \partial_x z_2\|_{L^2}^2  +B^{-1}\| z_2\|_{L^2}^2 + (Ae^{A/4})^{-1}\| \zeta_B v_2\|_{L^2}^2 \\
\lesssim &~{}  \|\partial_x z_2\|_{L^2}^2+\delta^{1/10} \|z_2\|_{L^2}^2+\delta \|w_2\|_{L^2}^2.
\end{aligned}
\end{equation}

\subsubsection{Control of $\mathcal{M}_1$.}

Recalling $\mathcal M_1$ in \eqref{eq:M0123} and that $\tilde{F}=\partial_x F$, with $F$ given by \eqref{eq:G_H}, we can say the following: 
\[
\begin{aligned}
|\mathcal{M}_1|
\lesssim & ~{} 
\| \chi_A\varphi_B \partial_x v_1\|_{L^2}
\| \chi_A(1-\gamma\partial_x ^2)^{-1} \partial_x^2 \left[ u_1^3+3H u_1^2 \right] \|_{L^2}.
\end{aligned}
\]
Using \eqref{rem:chiA_zetaA4},
\[
\begin{aligned}
|\mathcal{M}_1|
\lesssim & ~{} 
\| \chi_A\varphi_B \partial_x v_1\|_{L^2}
\| \zeta_A^2 (1-\gamma\partial_x ^2)^{-1} \partial_x^2 \left[ u_1^3+3H u_1^2 \right]\|_{L^2}.
\end{aligned}
\]
Now, using \eqref{cor:estimates_Sech_Iop_partial} and by \eqref{eq:LambaN}, we obtain
\begin{equation}\label{eq:LambdaP_N}
\begin{aligned}
\|\zeta_A^2 (1-\gamma\partial_x ^2)^{-1} \partial_x^2 \left[ u_1^3+3H u_1^2 \right] \|_{L^2}
\leq&~{} \gamma^{-1/2} \left\|\zeta_A^2 \partial_x \left[ u_1^3+3H u_1^2 \right]\right\|_{L^2}\\
\leq&~{} \gamma^{-1/2} \| u_1\|_{L^\infty} \left[ \left\|w_1\right\|_{L^2} +\left\|\zeta_A\partial_{x} u_{1}\right\|_{L^2} \right].
\end{aligned}
\end{equation}
Then, by \eqref{eq:estimates_psiB_v1}, \eqref{eq:estimates_v1_w1}  and the above estimates, we get
\[
\begin{aligned}
| \mathcal{M}_1|
\lesssim & ~{}   \gamma^{-1/2}B \delta \left\| \zeta_{A}^{2} \partial_x v_1\right\|_{L^2} \big[ \left\|w_1\right\|_{L^2} +\left\|\zeta_A\partial_{x} u_{1}\right\|_{L^2}\big]\\
\lesssim & ~{}   \gamma^{-3/2}B \delta  \left[ \left\|w_1\right\|_{L^2}^{2} +\left\|\partial_{x} w_{1}\right\|_{L^2}^{2}\right].
\end{aligned}
\]
Considering that $\delta \gamma^{-3/2}B=\delta^{3/10}$, by \eqref{Fijacion2}, we conclude
\begin{equation}\label{eq:boundM1}
\begin{aligned}
| \mathcal{M}_1|
\lesssim & ~{}   \delta^{3/10}  \left[ \left\|w_1\right\|_{L^2}^{2} +\left\|\partial_{x} w_{1}\right\|_{L^2}^{2}\right].
\end{aligned}
\end{equation}

\subsubsection{Control of $\mathcal{M}_2$} Recall that  $\tilde{G}=\partial_x G$, $G$ given by \eqref{eq:G_H}. First of all, we have
\begin{equation*}
\begin{aligned}
\mathcal{M}_2 =&~{} \int \phi_{A,B} \partial_x G \partial_x v_2\\
=&-\gamma \int ((\chi_A^2)'\zeta_B^2+\chi_A^2(\zeta_B^2)') \partial_x v_2   (1-\gamma\partial_x ^2)^{-1} \left[   V_0''\partial_x v_2+2 V_0'\partial_x^2 v_2 \right]\\
&-\gamma \int \chi_A^2\varphi_B \partial_x^2 v_2 (1-\gamma\partial_x ^2)^{-1} \left[  V_0''\partial_x v_2+2V_0'\partial_x^2 v_2 \right]\\
& -\rho' \int \chi_A^2\varphi_B \partial_x v_2  (1-\gamma \partial_x^2)^{-1}\partial_x (V_0' u_1)\\
=: &~{} M_{21}+M_{22}+M_{23}.
\end{aligned}
\end{equation*}
\medskip
First, we focus on $M_{21}$. Using \eqref{eq:bound_ZCV_B}, \eqref{eq:CZvx} in Remark \ref{rem:CZ_v2x} and \eqref{eq:estimates_v2_w2}, we have
\begin{equation*}
\begin{aligned}
	 \| ((\chi_A^2)' & \zeta_B^2+ \chi_A^2(\zeta_B^2)') \partial_x v_2 \|_{L^2}\\
	 = &~{} 2\| (\chi_A'\zeta_B+\chi_A \zeta_B') \chi_A \zeta_B \partial_x v_2 \|_{L^2}\\
	\lesssim&~{} (A^{-1}+B^{-1})\|   \chi_A\zeta_B \partial_x v_2 \|_{L^2}\\
	\lesssim&~{} B^{-1}( \| \partial_x z_2\|_{L^2}+B^{-1/2}\| z_2\|_{L^2}
	+ (AB)^{-1/2}\| \zeta_{B}v_2\|_{L^2})\\
		\lesssim&~{} B^{-1}( \| \partial_x z_2\|_{L^2}+B^{-1/2}\| z_2\|_{L^2}
	+ (AB)^{-1/2}\| w_2\|_{L^2}) \\
	\lesssim&~{} B^{-1}( \| \partial_x z_2\|_{L^2}+\| z_2\|_{L^2}
	+ \| w_2\|_{L^2}).
\end{aligned}
\end{equation*}
Additionally, by \eqref{eq:bound_f'x_v2x} and \eqref{eq:estimates_v2_w2}, and $e^{-A/4}\gamma^{-1} \ll 1$,
\[
\begin{aligned}
\left\|  (1-\gamma\partial_x ^2)^{-1} \left[   V_0''\partial_x v_2+2 V_0'\partial_x^2 v_2 \right] \right\|_{L^2} \lesssim &~{} \gamma^{-1/2} \left(\| \partial_x z_2\|_{L^2}+ B^{-1/2}\|z_2\|_{L^2}\right) \\
&~{}  +e^{-A/4}\gamma^{-1/2} \left(    \| \zeta_B v_2\|_{L^2} + \|\zeta_B \partial_x v_2\|_{L^2}\right) \\
\lesssim &~{} \gamma^{-1/2} \left(\| \partial_x z_2\|_{L^2}+ \|z_2\|_{L^2}\right) +\|w_2\|_{L^2}.
\end{aligned}
\]
After applying the Cauchy-Schwarz inequality on $M_{21}$, we conclude that
\begin{equation}\label{eq:M21}
\begin{aligned}
|M_{21}|
  \lesssim & ~{}   B^{-1}\gamma^{1/2}\big[  \| \partial_x z_2\|_{L^2}^{2}+\| z_2\|_{L^2}^{2} +\| w_2 \|_{L^2}^{2}\big].
\end{aligned}
\end{equation}

\medskip

Secondly, for $M_{22}$ we set $\varrho(y)=\sech(y/10)$,  we perform the following separation
\[
\begin{aligned}
|M_{22}|
\lesssim & ~{}\left| \gamma \int \chi_A^2\varphi_B \partial_x^2 v_2 (1-\gamma\partial_x ^2)^{-1} \left[  V_0''\partial_x v_2+2V_0'\partial_x^2 v_2 \right] \right| \\
\lesssim &~{}  \gamma  \|\varrho \chi_A^2\varphi_B  \partial_x^2 v_2\|_{L^2}   \left\| \varrho^{-1}(1-\gamma\partial_x ^2)^{-1} \left[   V''_0 \partial_x v_2+2 V'_0 \partial_x^2 v_2\right]\right\|_{L^2}.
\end{aligned}
\]
Since $|\varrho\varphi_B\zeta_B^{-1}|\lesssim B$, one gets
\[
\begin{aligned}
|M_{22}|
\lesssim&~{} \gamma B  \| \chi_A \zeta_B \partial_x^2 v_2\|_{L^2}  
\left\| \varrho^{-1}(1-\gamma\partial_x ^2)^{-1} \left[   V_0'' \partial_x v_2+2 V_0' \partial_x^2 v_2\right]\right\|_{L^2}.
\end{aligned}
\]
Now we focus on the second term on the RHS above. Using \eqref{eq:cosh_kink} from Lemma \ref{lem:cosh_kink}, we obtain
\begin{equation}\label{eq:rhoK_G}
\begin{aligned}
\big\| \varrho^{-1}(1-\gamma\partial_x ^2)^{-1} &\big[   V_0'' \partial_x v_2+2 V_0' \partial_x^2 v_2\big]\big\|_{L^2}\\
	\lesssim &~{} 
	\left\|(1-\gamma\partial_x ^2)^{-1} \big[ V''_0 \varrho^{-1}\partial_x v_2+2 V'_0\varrho^{-1}\partial_x^2 v_2\big]\right\|_{L^2}\\
		\lesssim &~{}
	\left\|  V_0'' \varrho^{-1}\partial_x v_2\right\|_{L^2}+\left\| V_0' \varrho ^{-1}\partial_x^2 v_2\right\|_{L^2}.
\end{aligned}
\end{equation}
Since $ |V_0'' \varrho^{-1}|\lesssim e^{-|y|}$, following a similar decomposition as in \eqref{eq:exp_vx}, we obtain
\begin{equation*}
\begin{aligned}
\|e^{-|y|}\partial_x v_2\|_{L^2}\lesssim & ~{}   \|\partial_x z_2\|_{L^2}+B^{-1/2} \|z_2\|_{L^2}+A^{-1}e^{-A/4} \|\zeta_B v_2\|_{L^2} +e^{-A}\|\zeta_A \partial_x v_2\|_{L^2}\\
\lesssim & ~{}  \|\partial_x z_2\|_{L^2}+B^{-1/2} \|z_2\|_{L^2}+A^{-1} e^{-A/4}\|w_2\|_{L^2}.
\end{aligned}
\end{equation*}
The case of the second term in the RHS of \eqref{eq:rhoK_G} requires more care. First of all, we note again that $V_0' \rho^{-1}\sim e^{-|y|}$ and repeating the same decomposition, we have
\[
\begin{aligned}
\|e^{-|y|} \partial_x^2 v_2\|_{L^2}
\leq&~{} e^{-A/2} \|\zeta_B \partial_x^2 v_2\|_{L^2} + \|e^{-|y|/2} \chi_A \zeta_B\partial_x^2 v_2\|_{L^2} .
\end{aligned}
\]
Applying Remark \ref{rem:CZ_v2xx} and Lemma \ref{lem:v_w} \eqref{eq:estimates_v2_w2}, we have
\[
\begin{aligned}
\|e^{-|y|} \partial_x^2 v_2\|_{L^2}
\leq &~{} 
\| \partial_x^2 z_2\|_{L^2}+B^{-1/2}\| \partial_x z_2\|_{L^2}+B^{-1/2}\| z_2\|_{L^2} \\
&+(AB)^{-1/2} e^{-A/4}\gamma^{-1/2}\| w_2\|_{L^2}
+e^{-A/2} \gamma^{-1}\| w_2\|_{L^2} .
\end{aligned}
\]
Finally, gathering the previous estimates, for $M_{22}$ we have
\begin{equation}\label{eq:M22}
\begin{aligned}
|M_{22}|\lesssim&~{}
\gamma B  \left( \| \partial_x^2 z_2\|_{L^2}^2+B^{-1}\| \partial_x z_2\|_{L^2}^2+B^{-1}\| z_2\|_{L^2}^2 +(AB)^{-1} \gamma^{-1/2}\| w_2\|_{L^2}^2 \right).
\end{aligned}
\end{equation}

\medskip

Third, we treat the term $M_{23}$. By H\"older's inequality and Lemma \ref{lem:estimates_IOp} (i), we get
\[
\begin{aligned}
|M_{23}|=&~{} |\rho'| \bigg|\int  \chi_{A} \varphi_B \partial_x v_2  (1-\gamma \partial_x^2)^{-1}\partial_x (V_0' u_1)\bigg|\\
\lesssim&~{} |\rho'| \| \chi_{A} \varphi_B \partial_x v_2 \|_{L^2} \|(1-\gamma \partial_x^2)^{-1}\partial_x (V_0' u_1)\|_{L^2}\\
\lesssim&~{} |\rho'| \| \chi_{A} \varphi_B \partial_x v_2 \|_{L^2} \| \partial_x (V_0' u_1)\|_{L^2}.
\end{aligned}
\]
Expanding the derivatives, by \eqref{eq:bound_ZCV_B} one gets
\[
\begin{aligned}
|M_{23}|\lesssim&~{} |\rho'| \| \chi_{A} \varphi_B \partial_x v_2 \|_{L^2} \| V_0'' u_1+V_0' \partial_x u_1\|_{L^2}\\
\lesssim&~{} B  |\rho'| \|  \partial_x v_2 \|_{L^2} \| V_0'' u_1+V_0' \partial_x u_1\|_{L^2},
\end{aligned}
\]
and by \eqref{eq:estimates_v2}, we obtain
\[
\begin{aligned}
|M_{23}|
\lesssim&~{} B \gamma^{-1/2} |\rho'| \|  u_2 \|_{L^2} \| V_0'' u_1+V_0' \partial_x u_1\|_{L^2}.
\end{aligned}
\]
Then having in mind \eqref{eq:wi} and \eqref{new cota}, using \eqref{eq:ineq_zK_u1x} (since $|V_0'|, |V_0''|\sim e^{-\sqrt{2}|x|}$), we obtain
\begin{equation}\label{eq:M23}
\begin{aligned}
|M_{23}|
\lesssim&~{} B\gamma^{-1/2}\delta  |\rho'| ( \| w_1\|_{L^2}+\|\partial_x w_1\|_{L^2}).
\end{aligned}
\end{equation}
Collecting  \eqref{eq:M21}, \eqref{eq:M22} and \eqref{eq:M23}, one obtains
\[
\begin{aligned}
|\mathcal{M}_2|
\lesssim & ~{}    B^{-1}\gamma^{1/2}\big[  \| \partial_x z_2\|_{L^2}^{2}+\| z_2\|_{L^2}^{2} +\| w_2 \|_{L^2}^{2}\big] \\
&~{}+
\gamma B \big[  \| \partial_x^2 z_2\|_{L^2}^2+B^{-1} \| \partial_x z_2\|_{L^2}^2+B^{-1}\| z_2\|_{L^2}^2]\\
&~{} +B\gamma^{-1/2}\delta  |\rho'| \left[ \| w_1\|_{L^2}+\|\partial_x w_1\|_{L^2}\right].
\end{aligned}
\]
Using \eqref{Fijacion2} and Cauchy-Schwarz inequality, we conclude
\begin{equation}\label{eq:boundM2}
\begin{aligned}
|\mathcal{M}_2|
\lesssim & ~{}   
\delta^{3/10}  \big[  \| \partial_x^2 z_2\|_{L^2}^2+\| \partial_x z_2\|_{L^2}^2+\| z_2\|_{L^2}^2]
\\
&~{}+\delta^{7/10} \left[ \| w_1\|_{L^2}^2+\|\partial_x w_1\|_{L^2}^2+\| w_2\|_{L^2}^2+ |\rho'|^2 \right].
\end{aligned}
\end{equation}

\subsubsection{Control of $\mathcal{M}_3$}
Replacing \eqref{eq:psi_deriv}, we get
\[
\begin{aligned}
\mathcal{M}_3=& -\rho'  \int (\chi_A^2 \zeta_B^2+(\chi_A^2)'\varphi_B)\partial_x v_1\partial_x v_2\\
=& -\rho'  \int \chi_A^2 \zeta_B^2 \partial_x v_1\partial_x v_2-2\rho'  \int \chi_A'  \varphi_B \chi_A \partial_x v_1\partial_x v_2=: M_{31}+M_{32}.
\end{aligned}
\]
Control of $M_{31}$. By H\"older's inequality, recalling that $v_1=(1-\gamma\partial_x^2)^{-1}\LL u_1$ (see \eqref{eq:change_variable}), and using \eqref{eq:estimates_v1}, we get 
\[
\begin{aligned}
|M_{31}|
\leq&~{} |\rho'  | \| \chi_A\zeta_B \partial_x v_1 \|_{L^2}\| \chi_A\zeta_B \partial_x v_2\|_{L^2}\\
\lesssim&~{} \gamma^{-1} |\rho'| \|u_1\|_{H^1}  \| \chi_A \zeta_B \partial_x v_2\|_{L^2}.
\end{aligned}
\]
From \eqref{new cota}, 
\begin{equation}\label{M31}
\begin{aligned}
|M_{31}|
\lesssim&~{} \gamma^{-1} \delta  |\rho'| \| \chi_A \zeta_B \partial_x v_2\|_{L^2}.
\end{aligned}
\end{equation}
Control of $M_{32}$: applying H\"older's inequality, we get
\[
\begin{aligned}
|M_{32}| 
\lesssim&~{} BA^{-1}|\rho' |  \| \partial_x v_2\|_{L^2 (A<|y|<2A)} \| \partial_x v_1 \|_{L^2} .
\end{aligned}
\]
Using \eqref{eq:estimates_v1} and \eqref{rem:chiA_zetaA4}, we obtain
\[
\begin{aligned}
|M_{32}| \lesssim&~{} \gamma^{-1} BA^{-1}|\rho' |  \| u_1 \|_{H^1} \| \zeta_A^{2}\partial_x v_2\|_{L^2},
\end{aligned}
\]
and
\begin{equation}\label{M32}
\begin{aligned}
|M_{32}| \lesssim&~{} \gamma^{-1} \delta BA^{-1}|\rho' |   \| \zeta_A^{2}\partial_x v_2\|_{L^2}.
\end{aligned}
\end{equation}
Collecting \eqref{M31} and \eqref{M32}, we get
\[
\begin{aligned}
\left| \mathcal{M}_3 \right| \lesssim&~{}
 \gamma^{-1} \delta |\rho'| \big[  \| \chi_A \zeta_B \partial_x v_2\|_{L^2}
+ BA^{-1}    \| \zeta_A^{2}\partial_x v_2\|_{L^2}\big].
\end{aligned}
\]
Then, by \eqref{eq:CZvx}, \eqref{eq:estimates_v2_w2} and \eqref{Fijacion2}, we obtain
\begin{equation*}
\begin{aligned}
& \big| \mathcal{M}_3 \big| \\
&~{} \lesssim
 \gamma^{-1}\delta  |\rho'| \big[  
  \| \partial_x z_2\|_{L^2}+B^{-1/2}\| z_2\|_{L^2}
	+ (AB)^{-1/2}\| w_2\|_{L^2}
+ BA^{-1}\gamma^{-1/2}    \| w_2\|_{L^2}\big]\\
&~{} \lesssim
 \delta^{3/5}  |\rho'| \left[  
  \| \partial_x z_2\|_{L^2}+ \delta^{1/20}\| z_2\|_{L^2}
	+  \delta^{11/20}\| w_2\|_{L^2}
\right].
\end{aligned}
\end{equation*}
Hence, using Cauchy-Schwarz inequality, we conclude
\begin{equation}\label{eq:M3}
\begin{aligned}
\left| \mathcal{M}_3 \right| 
 \lesssim&~{} \delta^{3/5}  \left[ |\rho'|^2+\| \partial_x z_2\|_{L^2}^2+\delta^{1/10}\| z_2\|_{L^2}^2
						 +\delta^{11/10} \|w_2\|^2_{L^2}
 						\right].
\end{aligned}
\end{equation}

\subsection{End of proof Proposition \ref{prop:ineq_dtM}}

Gathering  \eqref{eq:bound_errorM}, \eqref{eq:m0}, \eqref{eq:boundM1},  \eqref{eq:boundM2} and \eqref{eq:M3},   we obtain
 	\begin{equation*}
	\begin{aligned}
	\frac{d}{dt}\mathcal{M}
	\leq&
		-\dfrac{1}{2}\int 
				 \left[ (\partial_x z_1)^2 
					+ V_0 (\partial_x z_2)^2 
					+3 (\partial_x^2 z_2)^2
				\right]
				+C_3\delta^{3/10}   \| \partial_x^2 z_2\|_{L^2}^2
				\\
&~{} + C_3 \delta^{1/10}\big(\| z_1\|_{L^2}^2 +\| z_2\|_{L^2}^2 \big)
+C_3 \|\partial_x z_2\|_{L^2}^2\\
	&~{} +C_3 \delta^{1/10} \big(\|  w_1\|^2_{L^2} + \|   \partial_x w_1\|_{L^2}^2+\| w_2\|^2_{L^2} \big)
 +\delta^{7/10}  |\rho'|^2.
			\end{aligned}
	\end{equation*}
 	Finally,  for $\delta$ small enough,  we conclude that for some $C_3>0$ fixed,
	\begin{equation*}
	\begin{aligned}
	\frac{d}{dt}\mathcal{M}
\leq &~{}
	-\dfrac{1}{2}\int  \left[ (\partial_x z_1)^2 
	+2 (\partial_x^2 z_2)^2
	\right]+ C_3 (\| z_1\|_{L^2}^2 +\| z_2\|^2_{L^2}+\|\partial_x z_2\|^2_{L^2}) \\
	&+C_3\delta^{1/10} \big(\|  w_1\|^2_{L^2} + \|   \partial_x w_1\|_{L^2}^2+\| w_2\|^2_{L^2} \big)
+C_3\delta^{7/10}  |\rho'|^2
.
			\end{aligned}
	\end{equation*}
This proves \eqref{eq:dM_z}.

\section{A second transfer estimate}\label{sec:5}

The variation of the virial functional $\M$ in \eqref{eq:dM_z} involves the terms $\partial_x z_1$ and $\partial_x^2 z_2$, but these terms do not appear in the variation of the virial related to the dual problem \eqref{eq:dJ}.  Hence, we need to find a way to transfer information between  the terms $\partial_x z_1$ to $\partial_x^2 z_2$. Consider now the  virial functional $\mathcal N$ defined  as
\begin{equation}\label{eq:N_def}
\begin{gathered}
\mathcal{N}(t)=\int \rho_{A,B} (y)\tv_1 (t,x)v_2(t,x)dx=\int \rho_{A,B} (y)\partial_x v_1 (t,x) v_2(t,x)dx,
\end{gathered}
\end{equation}
where $\rho_{A,B}$ is a well-chosen localized weight depending on $A$ and $B$. A similar quantity was considered in \cite{KMMV,M_GB}. Note that the virial $\mathcal{N}$ considers at the same time the dynamics presented in \eqref{eq:syst_v} and \eqref{eq:syst_vx}.

\subsection{A virial identity for $\N$}

Let $\mathcal N$ be defined as \eqref{eq:N_def}.
 
\begin{lem}\label{lem:identity_dtN}
	Let $(v_1,v_2)\in H^1(\R)\times H^2(\R)$ a solution of \eqref{eq:syst_v}. Consider $ \rho_{A,B} =\rho_{A,B}(y) $ an even smooth bounded function to be a choose later. Then
	\begin{equation}
	\begin{aligned}\label{eq:N'}
	\frac{d}{dt}\mathcal{N}=&~{} 
		-\int \rho_{A,B}  \left( (\partial_x^2  v_2)^2+V_0 (\partial_x v_2)^2\right)
	+\int \rho_{A,B} (\partial_x v_1)^2
	+2\int \rho_{A,B}''   (\partial_x v_2)^2\\
	&-\frac12 \int \rho_{A,B}^{(4)} v_2^2 
	+\frac12 \int [\rho_{A,B} V_0]'' v_2^2    
	+ \int \rho_{A,B} v_2 \tilde{G}
	+  \int \rho_{A,B}  \partial_x v_1 F\\
	&  - \rho' \int \rho_{A,B}' \partial_x v_1 v_2.
	\end{aligned}
	\end{equation}
\end{lem}

\begin{proof}
	Computing the variation of $\mathcal{N}$, using \eqref{eq:syst_v} and \eqref{eq:syst_vx} and rewriting everything in terms of $(v_1,v_2)$, we obtain
	\begin{equation}\label{eq:dtn}
	\begin{aligned}
	\frac{d}{dt}\mathcal{N}=& \int \rho_{A,B} v_2 \partial_x \LL \tv_2+ \int \rho_{A,B} \partial_x v_1 \tv_1
	+ \int \rho_{A,B} v_2 \tilde{G}
	\\
	&~{} +  \int \rho_{A,B}  \tv_1 F- \rho' \int \rho_{A,B}' \partial_x v_1 v_2\\
	=& \int \rho_{A,B} v_2 \partial_x \LL \partial_x v_2+ \int \rho_{A,B} (\partial_x v_1)^2
	+ \int \rho_{A,B} v_2 \tilde{G}\\
	&
	+  \int \rho_{A,B}  \partial_x v_1 F - \rho' \int \rho_{A,B}' \partial_x v_1 v_2.
	\end{aligned}
	\end{equation}
	Applying \eqref{eq:PLP},
	\begin{equation}\label{eq:n_plp}
	\begin{aligned}
	\int \rho_{A,B} v_2 \partial_x  \LL\partial_x  v_2
	=&
	-\int \rho_{A,B}  \left[ (\partial^2_x  v_2)^2+V_0 (\partial_x v_2)^2 \right]
	+2\int \rho_{A,B}''   (\partial_x v_2)^2
	\\&
	+\frac12\int  [\rho_{A,B}'' V_0+\rho_{A,B}' V_0'] v_2^2
	-\int \rho_{A,B}^{(4)} v_2^2 .
	\end{aligned}
	\end{equation}
	Collecting \eqref{eq:dtn} and \eqref{eq:n_plp}, we obtain
	\begin{equation*}
	\begin{aligned}
	\frac{d}{dt}\mathcal{N}=&~{} 
	-\int \rho_{A,B}  \left( (\partial_x^2  v_2)^2+V_0 (\partial_x v_2)^2\right)
	+\int \rho_{A,B} (\partial_x v_1)^2
	+2\int \rho_{A,B}''   (\partial_x v_2)^2
	\\&
	-\frac12 \int \rho_{A,B}^{(4)} v_2^2 
	+\frac12\int  [\rho_{A,B}'' V_0+\rho_{A,B}' V_0'] v_2^2
	+ \int \rho_{A,B} v_2 \tilde{G}
	+  \int \rho_{A,B}  \partial_x v_1 F\\
	& - \rho' \int \rho_{A,B}' \partial_x v_1 v_2.
	\end{aligned}
	\end{equation*}
	This concludes the proof of \eqref{eq:N'}.
	\end{proof}

Now we choose the weight function $\rho_{A,B}$. Let
\begin{equation}\label{rho_AB}
\rho_{A,B}=\chi_A^2 \zeta_B^2,
\end{equation}
with $\chi_A$ and $\zeta_B$ introduced in \eqref{eq:psi_chiA} and  \eqref{eq:bound_phiA}.

\begin{prop}\label{prop:virial_N}
	Under \eqref{rho_AB}, there exist $C_4$ and $\delta_4>0$ such that, for any $0<\delta\leq\delta_4$, the following holds. Assume that for all $t\geq 0$ \eqref{eq:ineq_hip} holds, and $A$ and $B$ satisfy \eqref{Fijacion1} and \eqref{Fijacion2}, respectively. Then, for all $t\geq 0$, 	
	\begin{equation}\label{eq:dN_z} 
	\begin{aligned}
	\frac{d}{dt}\mathcal{N} (t) 
 	\geq &~{} 	\frac12\int (\partial_x z_1)^2 - C_{4}\int \left[ (\partial_x^2 z_2)^2  + (\partial_x z_2)^2 + z_2^2\right]
	-C_4 \delta^{1/10} \|z_1\|_{L^2}^2
	 \\
	&-C_{4} \delta^{1/5}(\| w_1\|_{L^2}^2+\|\partial_{x} w_{1}\|_{L^2}^{2}+ \| w_2\|_{L^2}^2)
         - C_4\delta |\rho'|^2.
	\end{aligned}
	\end{equation}
\end{prop}

\subsection{Start of proof of Proposition \ref{prop:virial_N}} The proof of this result is based in the following lemma, that  relates Lemma \ref{lem:identity_dtN} and the variables $z_i$.

\begin{lem}\label{lem:dtN}
	Assume that for all $t\geq 0$ \eqref{eq:ineq_hip} holds, and $A$ and $B$ satisfy \eqref{Fijacion1} and \eqref{Fijacion2}, respectively. Consider $ \rho_{A,B} $ as in \eqref{rho_AB}. Let $(v_1,v_2)\in H^1(\R)\times H^2(\R)$ be a solution of \eqref{eq:syst_v}.  Then
		\begin{equation}\label{dtN}
	\begin{aligned}
	\frac{d}{dt}\mathcal{N}=& 
\int  (\partial_x z_1)^2 
	-\int \left[ (\partial_x^2 z_2)^2
	+ V_0 (\partial_x z_2)^2 
	-\frac12  V_{0}''   z_2^2\right]\\
	&+\mathcal{RZ}+\mathcal{RV}
	+ \int \rho_{A,B} v_2 \tilde{G}
	+  \int \rho_{A,B}  \partial_x v_1 F - \rho' \int \rho_{A,B}' \partial_x v_1 v_2,
	\end{aligned}
	\end{equation}
	where $\mathcal{RZ}$ and $\mathcal{RV}$ are error terms that satisfy the following estimates
	\begin{equation*}
	\begin{aligned}
	|\mathcal{RZ}|\lesssim & ~{}  {\delta^{1/10}}(\| z_1\|_{L^2}^2+\| z_2\|_{L^2}^2+\| \partial_x z_2\|_{L^2}^2),\\
	|\mathcal{RV}|
	\lesssim & ~{}  \delta^{7/10}(\| w_1\|_{L^2}^2+\| \partial_{x} w_1\|_{L^2}^2+\| w_2\|_{L^2}^2).
	\end{aligned}
	\end{equation*}
\end{lem}
\begin{rem}
The terms $\mathcal{RZ}$ and $\mathcal{RV}$ are essentially appearing from the interaction among the cut-off functions at scales $A$ and $B$. They cannot be controlled in terms of positivity or negativity, only estimated. The worst terms are of order $B^{-1}$.
\end{rem}

\begin{proof}[Proof of Lemma \ref{lem:dtN}] 
	From \eqref{eq:N'}, we shall consider the following decomposition:
	\begin{equation}\label{eq:decomp_N'}
	\begin{aligned}
	\frac{d}{dt}\mathcal{N}=&~{} 
	\int \rho_{A,B} (\partial_x v_1)^2
	-\int \rho_{A,B}  (\partial_x^2  v_2)^2
	-\int \rho_{A,B} V_0 (\partial_x v_2)^2
	+2\int \rho_{A,B}''   (\partial_x v_2)^2\\
	&+\frac12 \int [\rho_{A,B} V_0]'' v_2^2 
	-\frac12 \int \rho_{A,B}^{(4)} v_2^2    
	+ \int \rho_{A,B} v_2 \tilde{G}
	+  \int \rho_{A,B}  \partial_x v_1 F - \rho' \int \rho_{A,B}' \partial_x v_1 v_2\\
	=:&~{} (N_1+N_2+N_3+N_4)+(N_5+N_6+\mathcal{N}_1+\mathcal{N}_2 +\mathcal N_3).
	\end{aligned}
	\end{equation}
	Secondly, from the definition of $\rho_{A,B}$ in \eqref{rho_AB}, one gets
	\begin{equation}\label{eq:rho_der}
	\begin{aligned}
			\rho_{A,B}'=&~{}(\chi_A^2)' \zeta_B^2+\chi_A^2 (\zeta_B^2)',\\
		\rho_{A,B}''=&~{} (\chi_A^2)'' \zeta_B^2+2(\chi_A^2)' (\zeta_B^2)'+\chi_A^2 (\zeta_B^2)'',\\
		\rho_{A,B}'''=&~{}(\chi_A^2)''' \zeta_B^2+3(\chi_A^2) ''(\zeta_B^2)'+3(\chi_A^2)' (\zeta_B^2)''+\chi_A^2 (\zeta_B^2)''',\\
		\rho_{A,B}^{(4)}=&~{}
		(\chi_A^2)^{(4)} \zeta_B^2+4(\chi_A^2)'''(\zeta_B^2)'
		+6(\chi_A^2)'' (\zeta_B^2)''+4(\chi_A^2)' (\zeta_B^2)'''+\chi_A^2 (\zeta_B^2)^{(4)} .
	\end{aligned}
	\end{equation}
	Now we estimate each term in \eqref{eq:decomp_N'}.

\medskip

First of all, from the definition of $z_i$, and \eqref{eq:1CZ_B_vix} in Remark \ref{rem:CZ_v2x},
\begin{equation}\label{eq:n1}
 \begin{aligned}
N_1 = &~{} \int \rho_{A,B} (\partial_x v_1)^2 \\
= &~{} \int (\partial_x z_1)^2 +
	\int     
	\frac{\zeta_B''}{\zeta_B} z_1^2
	+\int  \mathcal{E}_1 \zeta_B^2 v_1^2\\
=: &~{} \int (\partial_x z_1)^2 +\mathcal{RZ}_1 +\mathcal{RV}_1.
\end{aligned}
\end{equation}
Here, $\mathcal{E}_1 = \chi_A''\chi_A+(\chi_A^2)' \frac{\zeta_B'}{\zeta_B}.$ We easily have
\begin{equation}\label{eq:cotas1}
|\mathcal{RZ}_1| \lesssim \frac1{B} \|z_1\|_{L^2}^2, \quad |\mathcal{RV}_1| \lesssim \frac1{AB} \|\zeta_B v_1\|_{L^2}^2.
\end{equation}

\medskip

Similarly, from \eqref{eq:1_CZ_v2_xx} in Remark \ref{rem:CZ_v2xx}, 
\begin{equation}\label{eq:n2}
 \begin{aligned}
	N_2 = &~{}  -\int \rho_{A,B}  (\partial_x^2  v_2)^2 \\
	=&~{} 
	- \int (\partial_x^2 z_2)^2+\int \tilde{R}_1z_2^2
	-\int P_1 (\partial_x z_2)^2 
	-
	\int    \left[ P_1' \frac{\zeta_B'}{\zeta_B}+P_1 
	\frac{\zeta_B''}{\zeta_B}\right] z_2^2 \\
	&~{}  -\int \mathcal{E}_2(1) \zeta_B^2v_2^2
	-\int  \mathcal{E}_1(1)\zeta_B^2 v_2^2
	-\int \mathcal{E}_3 (1)\zeta_B^2(\partial_x v_2)^2\\
	= &~{} -\int (\partial_x^2 z_2)^2+ \mathcal{RZ}_2+ \mathcal{RV}_2,
	\end{aligned}
\end{equation}
with 
\[
\mathcal{RZ}_2 := - \int \tilde{R}_1z_2^2 -\int P_1 (\partial_x z_2)^2-
	\int    \left[ P_1' \frac{\zeta_B'}{\zeta_B}+P_1 
	\frac{\zeta_B''}{\zeta_B}\right] z_2^2;
\]
and
\[
\mathcal{RV}_2:=- \int \mathcal{E}_2(1) \zeta_B^2v_2^2
	-\int  \mathcal{E}_1(1)\zeta_B^2 v_2^2
	-\int \mathcal{E}_3(1) \zeta_B^2(\partial_x v_2)^2.
\]
In order to bound these two terms, we invoke \eqref{eq:1_CZ_v2_xx} and the estimates in \eqref{eq:PR_E23}:
\begin{equation}\label{eq:cotas2}
|\mathcal{RZ}_2 | \lesssim \frac1{B}\|z_2\|_{L^2}^2  + \frac1{B}\|\partial_x z_2\|_{L^2}^2;
\end{equation}
\[
|\mathcal{RV}_2| \lesssim \frac1{AB} \left( \|\zeta_B v_2\|_{L^2}^2 +\|\zeta_B \partial_x v_2\|_{L^2}^2 \right).
\]
As for the term $N_3$, applying Claim \ref{claim:P_CZ_vi_x} with $i=2$ and $P=V_0$, we have
	\begin{equation}\label{eq:n3}
	\begin{aligned}
	N_3= &~{} 
	-\int V_0 (\partial_x z_2)^2 -
	\int   \left[ V_{0}
	\frac{\zeta_B''}{\zeta_B} +V_0' \frac{\zeta_B'}{\zeta_B}\right] z_2^2
	-\int \mathcal{E}_1( V_0)\zeta_B^2 v_2^2\\
	=:&~{} -\int V_0 (\partial_x z_2)^2 + \mathcal{RZ}_3 + \mathcal{RV}_3,
	\end{aligned}
	\end{equation}
	where $\mathcal{E}_1$ is given by \eqref{eq:claimE1}. We get
	\begin{equation}\label{eq:cotas3}
	|\mathcal{RZ}_3 | \lesssim \frac1B\|z_2\|_{L^2}^2, \quad 
	|\mathcal{RV}_3 | \lesssim \frac1{AB} \| \zeta_B v_2 \|_{L^2}^2.
	\end{equation}
In the case of the term $N_4$ in \eqref{eq:rho_der}, we have
	\begin{equation*}
	\begin{aligned}
	 N_4
	=& 2 \int [(\chi_A^2)'' \zeta_B^2+2(\chi_A^2)'(\zeta_B^2)']   (\partial_x v_2)^2
	+2 \int  \frac{(\zeta_B^2)''}{\zeta_B^2}   \chi_A^2 \zeta_B^2(\partial_x v_2)^2  ,
	\end{aligned}
	\end{equation*}
	and using  Claim \ref{claim:P_CZ_vi_x}, with $i=2$ and $P=(\zeta_B^2)''/\zeta_B^2$, we get
		\begin{equation}\label{eq:n4}
	\begin{aligned}
	 N_4
=&~{} 2\int \frac{(\zeta_B^2)''}{\zeta_B^2}   (\partial_x z_2)^2 +
2\int   \left[  \left(\frac{(\zeta_B^2)''}{\zeta_B^2}  \right)' \frac{\zeta_B'}{\zeta_B}+
\frac{(\zeta_B^2)''}{\zeta_B^2}   
\frac{\zeta_B''}{\zeta_B}\right] z_2^2
\\&~{}+2\int \mathcal{E}_1\left( \frac{(\zeta_B^2)''}{\zeta_B^2}  \right)\zeta_B^2 v_2^2
+2\int \left[(\chi_A^2)'' +4(\chi_A^2)' \frac{\zeta_B'}{\zeta_B} \right]  \zeta_B^2(\partial_x v_2)^2\\
=:&~{} \mathcal{RZ}_4 + \mathcal{RV}_4 .
	\end{aligned}
	\end{equation}
Notice that 
\begin{equation}\label{eq:cotas4}
|\mathcal{RZ}_4| \lesssim \frac1{B} \|\partial_x z_2\|_{L^2}^2 +\frac1{B^2} \|z_2\|_{L^2}^2 ,
\end{equation}
and from \eqref{cota_final},
\[
|\mathcal{RV}_4| \lesssim \frac1{AB^2} \|\zeta_B v_2\|_{L^2}^2 + \frac1{AB}\|\zeta_B \partial_x v_2\|_{L^2}^2.
\]	
	
\medskip	
	
	Now, for $N_5$,  expanding the derivative, replacing  \eqref{eq:rho_der} and using definition of $z_2$, we have
	\begin{equation}\label{eq:n5}
	\begin{aligned}
	 N_5
	=& ~{} \frac12	\int  V_{0}''    z_2^2 + \frac12	\int  \left[  \frac{(\zeta_B^2)''}{\zeta_B^2}  V_{0}
	+ 2\frac{(\zeta_B^2)'}{\zeta_B^2} V_{0}' \right] z_2^2
	\\& ~{} 
	 +\frac12 \int \left[ \left( 2(\chi_A^2)' \frac{(\zeta_B^2)'}{\zeta_B^2}+(\chi_A^2)''\right) V_{0}+2(\chi_A^2)'  V_{0}'  \right] \zeta_B^2 v_2^2\\
	 = :&~{} \frac12	\int  V_{0}''    z_2^2+ \mathcal{RZ}_5 + \mathcal{RV}_5. 
	\end{aligned}
	\end{equation}
	It is not hard to see, using that $e^{-cB} \ll B^{-2}$, that
\begin{equation}\label{eq:cotas5}
|\mathcal{RZ}_5| \lesssim \frac1{B} \|z_2\|_{L^2}^2,\quad |\mathcal{RV}_5| \lesssim \frac1{AB} \|\zeta_B v_2\|_{L^2}^2.
\end{equation}
	Finally, for $N_6$, replacing \eqref{eq:rho_der}, we have
	\begin{equation}\label{eq:n6}
	\begin{aligned}
	N_6
	=&~{}
	-\frac12\int \frac{(\zeta_B^2)^{(4)}}{\zeta_B^2} z_2^2 \\
	&~{} -\frac12 \int \left[ 4(\chi_A^2)'\frac{(\zeta_B^2)'''}{\zeta_B^2}+ 6(\chi_A^2)''\frac{(\zeta_B^2)''}{\zeta_B^2}+ 8(\chi_A^2)''' \frac{\zeta_B'}{\zeta_B}+ (\chi_A^2)^{(4)}  \right] \zeta_B^2v_2^2   
	\\
	=: &~{} \mathcal{RZ}_6 + \mathcal {RV}_6. 
	\end{aligned}
	\end{equation}
	Notice that
	\begin{equation}\label{eq:cotas6}
	|\mathcal{RZ}_6| \lesssim \frac1{B} \|z_2\|_{L^2}^2, \quad |\mathcal{RV}_6| \lesssim \frac1{AB^3} \|\zeta_B v_2\|_{L^2}^2.
	\end{equation}
	Therefore, collecting \eqref{eq:n1}, \eqref{eq:n2}, \eqref{eq:n3}, \eqref{eq:n4}, \eqref{eq:n5} and \eqref{eq:n6}, and replacing in \eqref{eq:decomp_N'}, we obtain
	\begin{equation*}
	\begin{aligned}
	\frac{d}{dt}\mathcal{N}
	=&	
	\int  (\partial_x z_1)^2 
	-\int \left[ (\partial_x^2 z_2)^2
	+ V_0 (\partial_x z_2)^2 
	-\frac12  V_{0}''   z_2^2\right]\\
	&+\mathcal{RZ}+\mathcal{RV}
	+ \int \rho_{A,B} v_2 \tilde{G}
	+  \int \rho_{A,B}  \partial_x v_1F- \rho' \int \rho_{A,B}' \partial_x v_1 v_2,
		\end{aligned}
	\end{equation*}
		where the error term related to $z=(z_1,z_2)$ is (see \eqref{eq:cotas1}, \eqref{eq:cotas2}, \eqref{eq:cotas3}, \eqref{eq:cotas4}, \eqref{eq:cotas5} and \eqref{eq:cotas6})
\begin{equation*}
\begin{aligned}
\mathcal{RZ}:=&~{} \sum_{j=1}^6 \mathcal{RZ}_j \\
|\mathcal{RZ}| \lesssim &~{} \frac1{B} \left( \|z_2\|_{L^2}^2 +\|z_1\|_{L^2}^2 +\| \partial_x z_1\|_{L^2}^2 \right).
\end{aligned}
\end{equation*}
Similarly, the remainder related to $(v_1,v_2)$ is
			\begin{equation*}
\begin{aligned}
\mathcal{RV}	= &~{} \sum_{j=1}^6 \mathcal{RV}_j \\
|\mathcal{RV} |\lesssim &~{} \frac1{AB}\left(  \|\zeta_B v_1\|_{L^2}^2 + \|\zeta_B v_2\|_{L^2}^2 +\|\zeta_B \partial_x v_2\|_{L^2}^2 \right).
\end{aligned}
\end{equation*}

	Applying estimate \eqref{v1_mejorada} with $K=B\ll A$, estimate \eqref{eq:estimates_sech_u1_w1} and Lemma \ref{lem:v_w} \eqref{eq:estimates_v2_w2}, we obtain
	\[
	\begin{aligned}
	\left|\mathcal{RV}\right|
	\lesssim & ~{}  (AB)^{-1} \gamma^{-1}(\| w_1\|_{L^2}^2+\|\partial_{x} w_{1}\|_{L^2}^{2}+ \| w_2\|_{L^2}^2).
	\end{aligned}
	\]
	Finally, by \eqref{Fijacion2}, we get $(AB)^{-1}\gamma^{-1}=\delta^{7/10}$ and we conclude
	\begin{equation}\label{eq:bound_RV_RDV}
	\begin{aligned}
	|\mathcal{RV}|
	\lesssim & ~{} \delta^{7/10}(\| w_1\|_{L^2}^2+\|\partial_{x} w_{1}\|_{L^2}^{2}+ \| w_2\|_{L^2}^2),
	\end{aligned}
	\end{equation}
	and
	\begin{equation}\label{eq:bound_RZ}
	\begin{aligned}
	|\mathcal{RZ}|\lesssim & ~{}  \delta^{1/10}\big(\| z_1\|_{L^2}^2+\| z_2\|_{L^2}^2+\| \partial_x z_2\|_{L^2}^2\big).
	\end{aligned}
	\end{equation}
	This ends the proof of Lemma \ref{lem:dtN}.
\end{proof}

\subsection{Control of nonlinear terms} Following \eqref{eq:decomp_N'}, the nonlinear terms in \eqref{dtN} are denoted as
\[
\mathcal{N}_1=  \int \rho_{A,B}  \partial_x v_1 F, \quad \mathcal{N}_2 = \int \rho_{A,B} v_2 \tilde{G}, \quad \mathcal N_3=  - \rho' \int \rho_{A,B}' \partial_x v_1 v_2.
\]

\noindent
{\bf Control of $\mathcal{N}_1$.} We observe that
\[
	(\chi_A \zeta_B)^2 \partial_x v_1=	\chi_A \zeta_B\partial_x z_1 -(\chi_A \zeta_B)' z_1,
\]
then, we have
\[
\begin{aligned}
\mathcal{N}_1=&  \int \left[\chi_A \partial_x z_1 -\frac{(\chi_A \zeta_B)'}{\zeta_B} z_1\right] \zeta_B F.
\end{aligned}
\]
Recall that $F$ is given by \eqref{eq:G_H}. Using \eqref{eq:LambaN} and \eqref{eq:ineq_hip} we have
\begin{equation}\label{eq:N1}
\begin{aligned}
|\mathcal{N}_1|\leq &~{} \left\| \chi_A \partial_x z_1 -\frac{(\chi_A \zeta_B)'}{\zeta_B} z_1\right\|_{L^{2}}\| \zeta_B F\|_{L^{2}}\\
 \lesssim & ~{}  \delta \left[ \|  \partial_x z_1\|_{L^2}^2 +\| z_1\|_{L^2}^2
+  \| w_1\|_{L^2}^2+\| \partial_{x} w_1\|_{L^2}^2 \right].
\end{aligned}
\end{equation}

\medskip

\noindent
{\bf Control of $\mathcal{N}_2$.}
Recalling  that  $\tilde{G}=\partial_x G$ and  $G$ is given by \eqref{eq:G_H}, using definition of $z_2$ and \eqref{rho_AB}, we have
\[
\begin{aligned}
|\mathcal{N}_2|
\leq &~{} 
\left|\int ((\chi_A \zeta_B)' z_2+\chi_A \zeta_B \partial_x z_2)  G\right|
\\&~{}
+\bigg|  \rho' \int ((\chi_A \zeta_B)' z_2+\chi_A \zeta_B \partial_x z_2)(1-\gamma \partial_x^2)^{-1} (V_0' u_1)\bigg|
=:~{}\mathcal{N}_{21}+\mathcal{N}_{22}.
\end{aligned}
\]
Firstly, we focus on $\mathcal{N}_{21}$. We note that, 
\[
\begin{aligned}
\left|\int ((\chi_A \zeta_B)' z_2+\chi_A \zeta_B \partial_x z_2)  G\right|
\lesssim   (\|z_2 \|_{L^2}+\|\partial_x z_2 \|_{L^2}) 
\|  G\|_{L^2}.
\end{aligned}
\]
Cauchy-Schwarz inequality, \eqref{eq:bound_f'x_v2x} and Lemma \ref{lem:v_w} \eqref{eq:estimates_v2_w2} lead to conclude that
\[
\begin{aligned}
|\mathcal{N}_{21}|
\lesssim & ~{} \gamma^{1/2}[\|z_2 \|_{L^2}+\|\partial_x z_2 \|_{L^2}]\left[\| \partial_x z_2\|_{L^2}
+ \|z_2\|_{L^2} +e^{-A}  (  \| \zeta_B v_2\|_{L^2} + \|\zeta_B \partial_x v_2\|_{L^2}) \right] \\
\lesssim & ~{} \gamma^{1/2}[ \|z_2 \|_{L^2}+\|\partial_x z_2 \|_{L^2} ] \left[\| \partial_x z_2\|_{L^2}
+ \|z_2\|_{L^2} +e^{-A}  \gamma^{-1/2}  \|w_2\|_{L^2}  \right] \\
\lesssim & ~{} \gamma^{1/2}\left[\| \partial_x z_2\|_{L^2}^{2}
+ \|z_2\|_{L^2}^{2} +e^{-A}  \gamma^{-1/2}  \|w_2\|_{L^2}^{2}  \right].
\end{aligned}
\]
Estimate of $\mathcal{N}_{22}$. Applying H\"older's inequality and Lemma \ref{lem:estimates_IOp}, we get
\[
|\mathcal{N}_{22}|\lesssim  ~{} |  \rho'|  \big[ \|(\chi_A \zeta_B)' z_2\|_{L^2}+\|\chi_A \zeta_B \partial_x z_2\|_{L^2} \big]
\|u_1\|_{L^2}.
\]
Then, by \eqref{eq:ineq_hip} and using Cauchy-Schwarz inequality, we obtain
\[
|\mathcal{N}_{22}|
\lesssim  ~{} \delta  \big[  |  \rho'|^2+\|z_2\|_{L^2}^2+\|\partial_x z_2\|_{L^2}^2 \big].
\]
Finally, by \eqref{Fijacion2}, we conclude
\begin{equation}\label{eq:N2}
\begin{aligned}
|\mathcal{N}_2|
\lesssim & ~{} \delta^{1/5}\left(\| \partial_x z_2\|_{L^2}^{2}
+ \|z_2\|_{L^2}^{2} + \|w_2\|_{L^2}^{2}  \right)
+\delta  \left(  |  \rho'|^2+\|z_2\|_{L^2}^2+\|\partial_x z_2\|_{L^2}^2 \right)
.
\end{aligned}
\end{equation}

\medskip

\noindent
{\bf Control of $\mathcal{N}_3$.} Replacing $\rho_{A,B}'$ (see \eqref{eq:rho_der}), regrouping the terms and applying H\"older inequality, we obtain
\[
\begin{aligned}
|\mathcal{N}_3|= 2| \rho'| \bigg|\int \big(\chi_A \zeta_B'+ \chi_A'\zeta_B \big) v_2 \chi_A \zeta_B\partial_x v_1 \bigg|
\lesssim&~{} B^{-1} | \rho'| \|v_2\|_{L^2} \| \chi_A \zeta_B\partial_x v_1\|_{L^2}.
\end{aligned}
\]
Applying \eqref{eq:ineq_hip} and \eqref{eq:CZvx}, obtaining
\begin{equation}\label{eq:cotaN3}
\begin{aligned}
|\mathcal{N}_3|\lesssim&~{} \delta B^{-1} | \rho'|  \| \chi_A \zeta_B\partial_x v_1\|_{L^2}\\
\lesssim & ~{}  \delta B^{-1} | \rho'| ( \| \partial_x z_1\|_{L^2}+B^{-1/2}\| z_1\|_{L^2}
	+ (AB)^{-1/2}\| \zeta_{B}v_1\|_{L^2}).
\end{aligned}
\end{equation}
Finally, by \eqref{Fijacion2}, we get $(AB)^{-1/2}\gamma^{-1/2} \lesssim 1$, and using \eqref{eq:kv1_w1}, we obtain
\begin{equation}\label{eq:N3}
\begin{aligned}
|\mathcal{N}_3|\lesssim & ~{}  \delta^{11/10}  \left[ | \rho'|^2+\| z_1\|_{L^2}^2+\| \partial_x z_1\|_{L^2}^2
	+ \| w_{1}\|_{L^{2}}^2 +\|\partial_{x} w_{1}\|_{L^{2}}^2\right].
\end{aligned}
\end{equation}

\subsection{End of proof of  Proposition \ref{prop:virial_N}}
Noticing that
	\begin{equation*}
	\begin{aligned}
	\frac{d}{dt}\mathcal{N}  \geq &~{} 	\int (\partial_x z_1)^2 -\int \left[ (\partial_x^2 z_2)^2  + V_0 (\partial_x z_2)^2 
	-\frac12 V_0'' z_2^2\right]
	 \\
	&- \left|\mathcal{RZ}\right|-\left|\mathcal{RV}\right| -|\mathcal{N}_1| -|\mathcal{N}_2|-|\mathcal{N}_3|.
	\end{aligned}
	\end{equation*}
Collecting \eqref{eq:bound_RV_RDV}, \eqref{eq:bound_RZ},  \eqref{eq:N1}, \eqref{eq:N2} and \eqref{eq:N3}, we obtain for some $C_4>0$ fixed in \eqref{dtN}
	\begin{equation*}
	\begin{aligned}
	\frac{d}{dt}\mathcal{N}  \geq &~{} 	\frac12\int (\partial_x z_1)^2 - C_4\int \left[ (\partial_x^2 z_2)^2  +  (\partial_x z_2)^2 
	+ z_2^2\right]
	 \\
	&
	-C_4 \delta^{1/5}(\| w_1\|_{L^2}^2+\|\partial_{x} w_{1}\|_{L^2}^{2}+ \| w_2\|_{L^2}^2)
\\&
	-C_4 \delta^{1/10}\big(\| z_1\|_{L^2}^2 +\| z_2\|_{L^2}^2+\| \partial_x z_2\|_{L^2}^2\big)
	-C_4 \delta    |  \rho'|^2
	 .
	\end{aligned}
	\end{equation*}
	From the above inequality, we concludes the proof of \eqref{eq:dN_z}.

\section{Coercivity estimates}\label{sec:5bis}

Before starting the proof of Theorem \ref{thm:asymptotic}, we need coercivity results to deal with the terms
\begin{equation}\label{last_estimates}
\int  \vA  \left(  V_0' +2H' u_1\right) u_1^2, \quad  \delta |\rho'|^2, \quad \rho' \int \vA H' u_2,
\end{equation}
that appear in the virial estimates of $\I$ (see \eqref{eq:dI_w}). We will decompose this term in terms of the variables $(w_1,w_2)$ and $(z_1,z_2)$. The last ones involve the variables $(v_1,v_2)$; then we should be able to reconstruct the operator $\LL$ from our computations.

\subsection{First coercivity estimate} The key element of the proof of Theorem \ref{thm:asymptotic} is the following transfer estimate.

\begin{lem}\label{cor:sech_u1_proof2}
	Let $u_1$ be in $H^1$, $(w_1,w_2)$ be as in \eqref{eq:wi}, and $(z_1,z_2)$ as in \eqref{eq:def_psiB}. Then
	\begin{equation}\label{eq:w_sech}
	\begin{aligned}
	\left| \int  \vA  \left(  V_0' +2H' u_1\right) u_1^2 \right|
\lesssim & ~{} \delta^{1/20} \left( \| w_1\|_{L^2}^2 +\|\partial_x w_1\|_{L^2}^2 \right)+\frac{1}{\delta^{1/20}} \|z_1\|_{L^2}^2 +\delta^{2/5}\| \partial_x z_1\|_{L^2}^2. 
	\end{aligned}
	\end{equation}
\end{lem}
\begin{proof}
First, we observe that  $\vA (y) \lesssim |y|$, and 
\begin{equation}\label{clever1}
|\vA  \left(  V_0' +2H' u_1\right)| 
\lesssim  |y| \sech^2 \left(\frac{y}{\sqrt{2}}\right)\lesssim \sech^2\left(\frac{y}2\right).
\end{equation}
	Set $\frac2B<\ell<\min\{\frac12, \frac14\sqrt{\lambda}\} \leq \frac12$. We note that
	\[
	\int  \sech^2 \left(\frac{y}2\right) u_1^2\lesssim \int  \sech^2\left(\ell y\right)u_1^2.
	\]
	Now, we focus on the term on the RHS of the last inequality. 
	 Applying Lemma \ref{lem:coerc_weight} with $\phi_L=\sech^2\left(\ell y \right)$, and using that  $|\phi'|\leq C\ell \phi$ and that $\langle u_1, H' \rangle=0$, we obtain for some $\lambda>0$,
	\begin{equation*}
	\begin{aligned}
	\int  \sech^2\left(\ell y \right)  u_1^2
	\leq&  \int  \sech^2\left(\ell y \right) \left[ u_1^2 +(\partial_x  u_1)^2\right]\\
	\leq& ~{}\frac1\lambda \int  \sech^2\left(\ell y \right) \left[ (\partial_x  u_1)^2+V_0 u_1^2\right].
	\end{aligned}
	\end{equation*}
	Now, integrating by parts one gets
	\begin{equation*}
	\begin{aligned}
		\int  \sech ( \ell y) (\partial_x  u_1)^2
					=& -\int  \sech^2 ( \ell y)u_1\partial_x^2 u_1+\frac12\int ( \sech^2 ( \ell y))'' u_1^2.
	\end{aligned}
	\end{equation*}
	Using that 
	\[
	|( \sech^2 ( \ell y))''|\leq \ell^2 \sech^2 ( \ell y),
	\]
	and choosing $\ell $ small enough ($0< \ell \leq \frac{\sqrt{\lambda}}{4}$), one gets
		\begin{equation*}
		\begin{aligned}
			\int  \sech^2 ( \ell y)u_1^2
			\lesssim & ~{}  
			 \int  \sech^2 \left( \ell y\right)\LL( u_1) u_1.
		\end{aligned}
	\end{equation*}
	Now, using definition of $v_1$, we obtain	
	\begin{equation}\label{eq:ineq_Pv1}
	\begin{aligned}
	\int  \sech^2 \left( \ell y\right)\LL( u_1) u_1
	\lesssim & ~{}  
	\int  \sech^2 \left( \ell y\right) u_1 v_1-\gamma\int  \sech^2 \left( \ell y\right) u_1 \partial_x^2 v_1.
	\end{aligned}
	\end{equation}
	For the first integral in the RHS of  \eqref{eq:ineq_Pv1}, using definition of $z_1$ and $w_1$, one can see that
	\begin{equation}\label{eq:CA3_sec}
	\begin{aligned}
		&	\int \sech^2 \left( \ell y\right)  u_1 v_1\\
&=\int \chi_A^3 \sech^2 \left( \ell y\right)  u_1 v_1+	\int (1-\chi_A^3) \sech^2 \left( \ell y\right)  u_1 v_1\\
&=\int \chi_A^2 \sech^2 \left( \ell y\right)(\zeta_A \zeta_B)^{-1}  w_1 z_1
+	\int (1-\chi_A^3) \sech^2 \left( \ell y\right)\zeta_A^{-2}  w_1 (\zeta_A v_1)\\
& \lesssim \max_{|y|<2A}\left\{\sech^2 \left( \ell y\right)(\zeta_A \zeta_B)^{-1}\right\}  \|w_1\|_{L^2} \|z_1\|_{L^2}
+	\max_{|y|>A}\left\{ \sech^2 \left( \ell y\right)\zeta_A^{-2} \right\} \|w_1\|_{L^2} \| \zeta_A v_1\|_{L^2}\\
& \lesssim\max_{|y|<2A}\left\{\sech^2 \left( \ell y\right)(\zeta_A \zeta_B)^{-1}\right\}  \|w_1\|_{L^2} \|z_1\|_{L^2}
+	\gamma^{-1} \max_{|y|>A}\left\{ \sech^2 \left( \ell y\right)\zeta_A^{-2} \right\} \|w_1\|_{L^2}^2\\
& \lesssim \epsilon \|w_1\|_{L^2}^2+\epsilon^{-1} \|z_1\|_{L^2}^2
+	\gamma^{-1} e^{-\frac{A}{4B}} \|w_1\|_{L^2}^2.
	\end{aligned}
	\end{equation}
	Note that the last inequality holds if $2B^{-1}<\ell$.	Now,  for  the second integral on the RHS of  \eqref{eq:ineq_Pv1}, integrating by parts we obtain the following expression 
	\begin{equation}\label{eq:decomposition_part_v1}
	\begin{aligned}
	\int \partial_x \big[ \sech^2 &\left( \ell y\right)  u_1\big] \partial_x v_1\\
	=& ~
	\int \big[ 
	(\sech^2 \left( \ell y\right))'  u_1
	+ \sech^2 \left( \ell y\right)  \partial_x u_1
	\big] \partial_x v_1\\
	=& ~
	\int  (\sech^2 \left( \ell y\right))' \chi_A^2 u_1 \partial_x v_1+\int (1-\chi_A^2) (\sech^2 \left( \ell y\right))'  u_1 \partial_x v_1\\
	&+ \int \sech^2 \left( \ell y\right) \chi_A^2 \partial_x u_1 \partial_x v_1+\int  (1- \chi_A^2)\sech^2 \left( \ell y\right) \partial_x u_1 \partial_x v_1\\
	=: &~{} \ell_1 +\ell_2+\ell_3+ \ell_4.
	\end{aligned}
	\end{equation}
	We treat each term $\ell_i$ in \eqref{eq:decomposition_part_v1}, starting with $\ell_1$. Using the following decomposition and by H\"older inequality, we get
	\[
	\begin{aligned}
	& \left| \int  (\sech^2 \left( \ell y\right))' \chi_A^2 u_1 \partial_x v_1\right|\\
	&~{} \lesssim \left| \int  (\sech^2 \left( \ell y\right))' \chi_A^3 u_1 \partial_x v_1 \right| + \left| \int  (\sech^2 \left( \ell y\right))' (1-\chi_A^3) u_1 \partial_x v_1 \right|\\
	&~{} \lesssim\left| \int  (\sech^2 \left( \ell y\right))' \chi_A^3 u_1 \partial_x v_1 \right| + \left| \int  (\sech^2 \left( \ell y\right))'\zeta_A^{-2} (1-\chi_A^3) w_1  (\zeta_A\partial_x v_1) \right|\\
	&~{} \lesssim \ell \|\chi_A u_1\|_{L^2} \|  \zeta_B \chi_A^2 \partial_x v_1 \|_{L^2} + \ell \max_{|x|>A}\left\{\sech^2(\ell y)\zeta_A^{-2}\right\}
	\|  w_1\|_{L^2}  \|\zeta_A \partial_x v_1\|_{L^2}.
	\end{aligned}
	\]
	Furthermore, by the definition of $z_1$, we can check
	\begin{equation}\label{eq:C^3Zv1_x}
	\begin{aligned}
			\chi_A^2 \zeta_B \partial_x v_1= \chi_A\partial_x z_1-\chi_A\frac{\zeta_B'}{\zeta_B} z_1- \chi_A' z_1;
	\end{aligned}
	\end{equation}
	and by Lemma \ref{lem:v_w} \eqref{eq:estimates_v1_w1} and Remark \ref{rem:chiA_zetaA4}, we obtain
		\begin{equation}\label{eq:coer_1}
	\begin{aligned}
	|\ell_1|=\left| \int  (\sech^2 \left( \ell y\right))' \chi_A^2 u_1 \partial_x v_1\right|
	\lesssim&~{}  \ell \| w_1\|_{L^2} (\|  \partial_x z_1 \|_{L^2}
	+B^{-1}\| z_1 \|_{L^2}) \\
	&+ \ell\gamma^{-1} \max_{|y|>A}\{\sech^2(\ell y)\zeta_A^{-2}\}
	(\|  \partial_x w_1\|_{L^2}^2+\|  w_1\|_{L^2}^2).
	\end{aligned}
	\end{equation}
	In similar way, we obtain for $\ell_3$
	\[
	\begin{aligned}
	& \left| \int  \sech^2 \left( \ell y\right) \chi_A^2\partial_x u_1 \partial_x v_1\right| \\
	&~{} \lesssim \left| \int  \sech^2 \left( \ell y\right) \chi_A^3 \partial_x u_1 \partial_x v_1 \right| + \left| \int  \sech^2 \left( \ell y\right) (1-\chi_A^3) \partial_x u_1 \partial_x v_1 \right|\\
	&~{} \lesssim  \|\chi_A \partial_x u_1\|_{L^2} \|  \zeta_B \chi_A^2 \partial_x v_1 \|_{L^2} 
	+  \max_{|y|>A}\left\{\sech^2(\ell y)\zeta_A^{-2}\right\}
	\|  \zeta_A \partial_x u_1\|_{L^2}  \|\zeta_A \partial_x v_1\|_{L^2}.
	\end{aligned}
	\]
	By \eqref{eq:C^3Zv1_x}, Lemma \ref{lem:v_w} and Remark \ref{rem:chiA_zetaA4}, we get
		\[
	\begin{aligned}
	\left| \int  \sech^2 \left( \ell y\right) \partial_x u_1 \partial_x v_1\right|
	\lesssim&~{}  
	 \|\zeta_A \partial_x u_1\|_{L^2} (\| \partial_x z_1 \|_{L^2} +\| z_1 \|_{L^2})
	+ \gamma^{-1} \max_{|y|>A}\{\sech^2(\ell y)\zeta_A^{-2}\}
	\|  \zeta_A \partial_x u_1\|_{L^2}^2  .
	\end{aligned}
	\]
	We conclude using \eqref{lem:cosh_kink} with $K=A$. We obtain
	\begin{equation}\label{eq:coer_2}
	\begin{aligned}
	\left| \int  \sech^2 \left( \ell y\right) \partial_x u_1 \partial_x v_1\right|
	\lesssim&~{}  
	( \| w_1\|_{L^2}+\|\partial_x w_1\|_{L^2} ) (\| \partial_x z_1 \|_{L^2} +\| z_1 \|_{L^2})\\
	&+ \gamma^{-1} \max_{|y|>A}\{\sech^2(\ell y)\zeta_A^{-2}\}
	(\| w_1\|_{L^2}^2+\|\partial_x w_1\|_{L^2}^2 ) .
	\end{aligned}
	\end{equation}
	Collecting \eqref{eq:CA3_sec}, \eqref{eq:coer_1}, \eqref{eq:coer_2} and by Cauchy-Schwarz inequality, we obtain
	\[
	\begin{aligned}
	\int \sech (y)  u_1^2
	\lesssim & ~{}  \epsilon \| w_1\|_{L^2}+\frac1\epsilon \| z_1\|_{L^2}\\
	&+\gamma \left(\| w_1\|_{L^2}^2+\|\partial_x w_1\|_{L^2}^2 \right)+\gamma\left(\| \partial_x z_1\|_{L^2}^2+B^{-2}\| z_1\|_{L^2}^2\right)\\
	\lesssim & ~{} \max\{ \epsilon,A^{-1},\gamma\} \| w_1\|_{L^2}
	+\gamma \|\partial_x w_1\|_{L^2}^2 \\
	&+\max \{\epsilon^{-1},B^{-2} \} \|z_1\|_{L^2}^2 
	+\gamma\| \partial_x z_1\|_{L^2}^2.
	\end{aligned}
	\]
	Finally, by \eqref{Fijacion2} and choosing $\epsilon=\delta^{1/20}$, we conclude
	\[
	\begin{aligned}
	\int \sech (y)  u_1^2
		\lesssim & ~{}  B^{-1/2} \left( \| w_1\|_{L^2}^2 +\|\partial_x w_1\|_{L^2}^2 \right)+B^{1/2} \|z_1\|_{L^2}^2 +\gamma\| \partial_x z_1\|_{L^2}^2\\
	\lesssim & ~{} \color{black} \delta^{1/20} \left( \| w_1\|_{L^2}^2 +\|\partial_x w_1\|_{L^2}^2 \right)+\frac{1}{\delta^{1/20}} \|z_1\|_{L^2}^2 +\delta^{2/5}\| \partial_x z_1\|_{L^2}^2.
	\end{aligned}
	\]
	This ends the proof of Lemma \ref{cor:sech_u1_proof2}.
\end{proof}

\subsection{Second coercivity estimate}

Now we consider the second term in \eqref{last_estimates}. Recall the estimate for the shift $\rho$ obtained in \eqref{eq:rho_p}, where
\[
 |\rho'|^2  \lesssim  \int e^{-\sqrt{2}|y|}u_2^2 .
\] 
In order to manage this term, we will decompose this term in localized terms for $w_2$ and $z_2$.

\begin{lem}
	Let $u_2$ be in $L^2$, $(w_1,w_2)$ be as in \eqref{eq:wi}, and $(z_1,z_2)$ as in \eqref{eq:def_psiB}. Then 	\begin{equation}\label{eq:u2_sech}
	\begin{aligned}
	\int  e^{-\sqrt{2}|y|} u_2^2
 \lesssim&~{} \delta^{1/20}\|w_2\|_{L^2}^2 
+\frac{1}{\delta^{1/20}} \|z_2\|_{L^2}^2
+\delta^{2/5} \|  \partial_x^2 z_2\|_{L^2}^2
+\delta^{3/5}\|\partial_x z_2\|_{L^2}^2 .
	\end{aligned}
	\end{equation}
\end{lem} 
\begin{proof}
Firstly, we consider the decomposition
\begin{equation}\label{eq:deco_u2}
\int  e^{-\sqrt{2}|y|} u_2^2= \int  \chi_A^3 e^{-\sqrt{2}|y|} u_2^2+\int (1- \chi_A^3) e^{-\sqrt{2}|y|} u_2^2.
\end{equation}
For the second integral on the RHS it holds that
\begin{equation}\label{eq:casicasi}
 \int (1- \chi_A^3) e^{-\sqrt{2}|y|} u_2^2\lesssim \sup_{|y|>A} \left\{e^{-\sqrt{2}|y|} \zeta_A^{-2}\right\} \int  w_2^2 \ll \gamma \|w_2\|_{L^2}^2.
\end{equation}
Now, we focus on the first integral in the RHS of \eqref{eq:deco_u2}. We observe that 
\[
\chi_A^2\zeta_B \partial_x v_2= \chi_A\partial_x z_2-(\chi_A \zeta_B)' \zeta_B^{-1} z_2,
\]
and
\begin{equation}\label{eq:chia3_v2}
\begin{aligned}
\chi_A^3 \zeta_B \partial_x^2 v_2
=& \chi_A^2 \partial_x^2 z_2-(\chi_A\zeta_B)'' \zeta_B^{-1} \chi_A z_2-2(\chi_A \zeta_B)'\zeta_B^{-1} \chi_A^2 \zeta_B\partial_x v_2\\
= &\chi_A^2 \partial_x^2 z_2-(\chi_A\zeta_B)'' \zeta_B^{-1} \chi_A z_2
-2(\chi_A \zeta_B)'\zeta_B^{-1} \left( \chi_A\partial_x z_2-(\chi_A \zeta_B)' \zeta_B^{-1} z_2 \right).
\end{aligned}
\end{equation}
Recalling that $v_2=(1-\gamma\partial_x^2)^{-1} u_2$ (see \eqref{eq:change_variable}), \eqref{eq:def_psiB} and using \eqref{eq:chia3_v2}, we obtain
\[
\begin{aligned}
\left| \int  \chi_A^3 e^{-\sqrt{2}|y|} u_2^2 \right|
=&~{}  \left| \int   e^{-\sqrt{2}|y|} (\zeta_A \zeta_B)^{-1}w_2  \chi_A^{3}\zeta_B(1-\gamma\partial_x^2) v_2 \right|\\
\lesssim &~{} \int   |w_2|  \chi_A^{3}\zeta_B | (1-\gamma\partial_x^2) v_2 |\\
\lesssim&~{} \|w_2\|_{L^2} \|z_2\|_{L^2}+\gamma \|w_2\|_{L^2} \|\chi_A^{3}\zeta_B \partial_x^2 v_2\|_{L^2}\\
\lesssim&~{} \|w_2\|_{L^2} \|z_2\|_{L^2}
+\gamma \|w_2\|_{L^2} (
\|  \partial_x^2 z_2\|_{L^2}
+B^{-1}\|\partial_x z_2\|_{L^2}
+B^{-1}\|z_2\|_{L^2}).
 \end{aligned}
\]
Using the Cauchy inequality and \eqref{eq:casicasi}, we get for $\epsilon>0$,
\[
\begin{aligned}
\int   e^{-\sqrt{2}|y|} u_2^2
\lesssim&~{} (\epsilon+\gamma)\|w_2\|_{L^2}^2 
+\frac{1}{\epsilon}\|z_2\|_{L^2}^2
+\gamma\|  \partial_x^2 z_2\|_{L^2}^2
+\gamma B^{-1}\|\partial_x z_2\|_{L^2}^2,
 \end{aligned}
\]
and, by \eqref{Fijacion2}, choosing $\epsilon=\delta^{1/20}$; we get
\[
\begin{aligned}
\int   e^{-\sqrt{2}|y|} u_2^2
\lesssim&~{}  \delta^{1/20}\|w_2\|_{L^2}^2 
+\frac{1}{ \delta^{1/20}}\|z_2\|_{L^2}^2
+\delta^{2/5}\|  \partial_x^2 z_2\|_{L^2}^2
+\delta^{1/2}\|\partial_x z_2\|_{L^2}^2,
 \end{aligned}
\]
which proves \eqref{eq:u2_sech}.
\end{proof}

\subsection{Third coercivity estimate}

Notice that on the variation of the functional $\mathcal{I}$ (see \eqref{eq:dI_w}), the integral term
\[
\left(\int \varphi_A H' u_2  \right)^2
\]
can be treated on a similar way as \eqref{clever1} on the above lemma, since
\[
|\varphi_A H'|\lesssim \sech^2\left(\frac{y}{2}\right),
\]
and by H\"older inequality, we get
\[
\left( \int \sech^2\left(\frac{y}{2}\right) u_2 \right)^2  \lesssim \int \sech^2\left(\frac{y}{2}\right) u_2^2.
\]
By \eqref{eq:u2_sech},
\[
 \left(\int \varphi_A H' u_2  \right)^2 \lesssim \delta^{1/20}\|w_2\|_{L^2}^2 
+\frac{1}{ \delta^{1/20}}\|z_2\|_{L^2}^2
+\delta^{2/5}\|  \partial_x^2 z_2\|_{L^2}^2
+\delta^{1/2}\|\partial_x z_2\|_{L^2}^2.
\]
We conclude the following 
\begin{cor}
Under the assumptions in Proposition \ref{prop:virial_I}, \eqref{Fijacion1} and \eqref{Fijacion2}, one has for some $C_1\geq C_0>0$
\begin{equation}\label{eq:dI_w_new}
\begin{aligned}
\dfrac{d}{dt}\mathcal I
		 \leq &~{} -\frac12 \int  \left( w_2^2  + 3(\partial_x w_1)^2  
		+\left(2 -3C_1\delta \right)w_1^2 \right)\\
		&~{} +C_1\delta^{1/20} \left( \| w_1\|_{L^2}^2 +\|\partial_x w_1\|_{L^2}^2+\|w_2\|_{L^2}^2 \right)+\frac{C_1}{\delta^{1/20}} \left( \|z_1\|_{L^2}^2 +\|z_2\|_{L^2}^2 \right)\\
		&~{}  +C_1 \delta^{2/5} \left( \| \partial_x z_1\|_{L^2}^2 +\|\partial_x z_2\|_{L^2}^2+\|  \partial_x^2 z_2\|_{L^2}^2 \right).
\end{aligned}
\end{equation}
\end{cor}

\subsection{Final coercivity estimate}

Lastly, we need the following coercivity property for $z_2$.
\begin{lem}\label{coercivity_z2}
There exists $m_0>0$ such that for any $z\in H^1$ holds
\begin{equation}\label{coercoer}
\int \left( 2 (\partial_x z)^2 + V_0 z^2\right)\geq m_0 \|z\|^2_{H^1} . 
\end{equation}
\end{lem}

\begin{rem}
Notice that no orthogonality condition is needed in Lemma \ref{coercivity_z2}. Indeed, the term  $2 (\partial_x z)^2$ is strong enough to avoid a very complicated to ensure orthogonality condition on $z_2$ in Proposition \ref{prop:virial_J}.
\end{rem}

\begin{proof}
Define $\mathcal L_\# := -2 \partial_x^2+V_0= -2 \partial_x^2 +2 -3 \sech^2\left( \frac{x}{\sqrt{2}}\right)$. Since $\mathcal L_\# \geq \mathcal L$, we easily have $\mathcal L_\# \geq 0$. Let us write
\[
\mathcal L_\# = \frac15 (-\partial_x^2 + 1) + \mathcal L_{\#\#}, \quad \mathcal L_{\#\#}:= -\frac{9}5(-\partial_x^2 + 1)-3 \sech^2\left( \frac{x}{\sqrt{2}}\right).
\]
It is not difficult to check that $\mathcal L_{\#\#}$ is a classical selfadjoint operator, with first eigenfunction given by
\[
\sech^m \left(  \frac{x}{\sqrt{2}}\right), \quad m=\frac{1}{6} \left(\sqrt{129}-3\right)\sim 1.39;
\]
and eigenvalue $\ell =\frac{3}{20} \left(\sqrt{129}-11\right)\sim 0.054>0$. Therefore, since $\mathcal L_{\#\#}\geq 0$, we conclude \eqref{coercoer} with $m_0= \frac15$. 
\end{proof}

From the previous result, we obtain the following corollary.

\begin{cor}
Under the hypotheses in Proposition \ref{prop:virial_J}, the following is satisfied. There exists $C_2>0$ and $m_0>0$ such that, for all $t\geq 0$, 
\begin{equation}\label{eq:dJ_new}
\begin{aligned}
\frac{d}{dt}\J \leq & ~{} -\frac{m_0}{4} \left( \| z_1\|_{L^2}^2+\| z_2\|_{L^2}^2 +\|\partial_x z_2\|_{L^2}^2\right) \\
&+C_2 \delta^{1/10} \left(\| w_1\|_{L^2}^2+\| \partial_{x} w_1\|_{L^2}^2+\| w_2\|_{L^2}^2 \right)
+ C_2\delta^{7/5}  \|  \partial_x^2 z_2\|_{L^2}^2.
\end{aligned}
\end{equation}    
\end{cor}
\begin{proof}
Direct from Proposition \ref{prop:virial_J} and Corollary \ref{coercivity_z2}.
\end{proof}

\section{Proof of Theorem \ref{thm:asymptotic}}\label{sec:6}

Now we are ready to conclude the proof of Theorem \ref{thm:asymptotic}.  Recall the constants $\delta_i>0$ for $i=1,\ldots, 4$, defined in Propositions \ref{prop:virial_I} \ref{prop:virial_J}, \ref{prop:ineq_dtM}, \ref{prop:virial_N}.

\begin{prop}\label{prop:virial_I+J+M+N}
	There exist $C_5>0$ and $0<\delta_5 \leq \min\{\delta_1,\delta_2,\delta_3,\delta_4\}$ such that for any $0<\delta\leq \delta_5$, the following is satisfied. Assume that for all $t\geq0$, \eqref{eq:ineq_hip} holds. Let
	\begin{equation*}
\mathcal H:= \mathcal{J}+ 16C_2\delta^{1/10}\mathcal{I} -32C_2C_1\delta^{1/2}\mathcal{N}+80C_1C_2C_4\delta^{1/5}\mathcal{M}.
\end{equation*}	
	Then, under \eqref{Fijacion1} and \eqref{Fijacion2}, one has for all $t\geq 0$,
	\begin{equation}\label{eq:dtH_inequality}
	\begin{aligned}
	\dfrac{d}{dt}\mathcal{H}
			\lesssim 		
	&- \delta^{1/10} \left(\|w_1\|_{L^2}^2+\|\partial_x w_1\|_{L^2}^2 
	+\| w_2\|_{L^2}^2  \right).
	\end{aligned}
	\end{equation}
\end{prop}
\begin{proof}
First,  from \eqref{eq:dI_w_new}, we obtain for $\delta>0$ small and  some $C_{1}$ fixed

\begin{equation}\label{8p14}
\begin{aligned}
\dfrac{d}{dt}\mathcal I
 \leq &-\frac14 \int  \left( w_2^2  + 3(\partial_x w_1)^2 +w_1^2 \right)
		\ +C_1 \delta^{1/20} \left( \| w_1\|_{L^2}^2 +\|\partial_x w_1\|_{L^2}^2 +\|w_2\|_{L^2}^2 \right)
		\\&+\frac{C_1}{\delta^{1/20}}(\|z_1\|_{L^2}^2+\|z_2\|_{L^2}^2) 
+C_1\delta^{2/5}( \| \partial_x z_1\|_{L^2}^2+\|\partial_x z_2\|_{L^2}^2+\| \partial_x^2 z_2\|_{L^2}^2).
\end{aligned}
\end{equation}
From \eqref{eq:dN_z} and \eqref{eq:u2_sech}, we get
\begin{equation}\label{8p15}
\begin{aligned}
		\frac12\int (\partial_x z_1)^2 
\leq &~{} \frac{d}{dt}\mathcal{N}  + 2C_{4}\int \left[ (\partial_x^2 z_2)^2  + (\partial_x z_2)^2 + z_2^2\right]
	+C_4 \delta^{1/10} \|z_1\|_{L^2}^2
	 \\
	&+C_{4} \delta^{1/5}(\| w_1\|_{L^2}^2+\|\partial_{x} w_{1}\|_{L^2}^{2}+ \| w_2\|_{L^2}^2).
\end{aligned}
\end{equation}         
Then, from \eqref{8p14} and \eqref{8p15}, we obtain
\begin{equation}\label{8p16}
\begin{aligned}
\dfrac{d}{dt} \left( \mathcal I-2C_1\delta^{2/5}\mathcal{N} \right)
			 \leq &-\frac14 \int  \left( w_2^2  + 3(\partial_x w_1)^2  +w_1^2 \right)
		+\frac{2C_1}{\delta^{1/20}}(\|z_1\|_{L^2}^2+\|z_2\|_{L^2}^2) 
\\&		+5C_4C_1\delta^{2/5}(\| \partial_x^2 z_2\|_{L^2}^2+ \|\partial_x z_2\|_{L^2}^2),
\end{aligned}
\end{equation}
since, without loss of generality, we can assume $C_4>1$.
\medskip

From \eqref{eq:dJ_new} we also have
\begin{equation}\label{8p17}
\begin{aligned}
\frac{d}{dt}\J \leq & ~{} -\frac{m_0}{4} \left( \| z_1\|_{L^2}^2+\| z_2\|_{L^2}^2 +\|\partial_x z_2\|_{L^2}^2\right) \\
&+C_2 \delta^{1/10} \left(\| w_1\|_{L^2}^2+\| \partial_{x} w_1\|_{L^2}^2+\| w_2\|_{L^2}^2 \right)
+ C_2\delta^{7/5}  \|  \partial_x^2 z_2\|_{L^2}^2.
\end{aligned}
\end{equation}

Gathering \eqref{8p16} and \eqref{8p17}, we conclude that

\begin{equation}\label{8p18}
\begin{aligned}
& \frac{d}{dt}\left( \mathcal{J}+ 16C_2\delta^{1/10}\mathcal{I} -32C_1 C_2\delta^{1/2}\mathcal{N} \right) \\
& ~{}  \leq -\frac{m_0}{8} \left( \| z_1\|_{L^2}^2+\| z_2\|_{L^2}^2 +\|\partial_x z_2\|_{L^2}^2\right) 
-3C_2\delta^{1/10} \left(\| w_1\|_{L^2}^2+\| \partial_{x} w_1\|_{L^2}^2+\| w_2\|_{L^2}^2 \right)
\\
&\quad		+80C_1C_2 C_4\delta^{1/2}\| \partial_x^2 z_2\|_{L^2}^2.
\end{aligned}
\end{equation}
On the other hand, from \eqref{eq:dM_z} and \eqref{eq:u2_sech}, we get
	 	\begin{equation*}
	\begin{aligned}
	\frac{d}{dt}\mathcal{M}
			\leq &~{}
	-\dfrac{1}{2}\int  \left( (\partial_x z_1)^2 + (\partial_x^2 z_2)^2 \right)
	+2C_3 (\| z_1\|_{L^2}^2 +\| z_2\|^2_{L^2}+\|\partial_x z_2\|^2_{L^2}) \\
	&+2C_3\delta^{1/10} \big(\|  w_1\|^2_{L^2} + \|   \partial_x w_1\|_{L^2}^2+\| w_2\|^2_{L^2} \big)
	.
			\end{aligned}
	\end{equation*}
Finally, define
\begin{equation}\label{eq:mH}
\mathcal H:= \mathcal{J}+ 16C_2\delta^{1/10}\mathcal{I} -32C_2C_1\delta^{1/2}\mathcal{N}+80C_1C_2C_4\delta^{1/5}\mathcal{M}.
\end{equation}	
We conclude from the last estimate and \eqref{8p18},
\begin{equation}\label{8p19}
\begin{aligned}
& \frac{d}{dt}\mathcal H\\
&\leq -\frac{m_0}{8} \left( \| z_1\|_{L^2}^2+\| z_2\|_{L^2}^2 +\|\partial_x z_2\|_{L^2}^2\right)
-3C_2\delta^{1/10} \left(\| w_1\|_{L^2}^2+\| \partial_{x} w_1\|_{L^2}^2+\| w_2\|_{L^2}^2 \right)
\\&~{}\quad		+80C_4C_1C_2\delta^{5/10}\| \partial_x^2 z_2\|_{L^2}^2
\\
 &~{} \quad
	-40C_1C_2C_4\delta^{1/5}\left( \|\partial_x z_1\|_{L^2}^2 + \|\partial_x^2 z_2\|_{L^2}^2 \right)
	+C \delta^{1/5} (\| z_1\|_{L^2}^2 +\| z_2\|^2_{L^2}+\|\partial_x z_2\|^2_{L^2}) \\
&~{} \quad +C\delta^{3/10} \big(\|  w_1\|^2_{L^2} + \|   \partial_x w_1\|_{L^2}^2+\| w_2\|^2_{L^2} \big)\\
&\leq  -\frac{m_0}{32} \left( \| z_1\|_{L^2}^2+\| z_2\|_{L^2}^2 +\|\partial_x z_2\|_{L^2}^2\right)
-C_2\delta^{1/10} \left(\| w_1\|_{L^2}^2+\| \partial_{x} w_1\|_{L^2}^2+\| w_2\|_{L^2}^2 \right)
\\
 &~{}\quad
	-20C_1C_2C_4\delta^{1/5}\left( \|\partial_x z_1\|_{L^2}^2 + \|\partial_x^2 z_2\|_{L^2}^2 \right),
\end{aligned}
\end{equation}
where $C=160 C_1C_2C_4C_3$. From \eqref{8p19} we obtain \eqref{eq:dtH_inequality}. This ends the proof of the proposition.
\end{proof}

We will use now Proposition \ref{prop:virial_I+J+M+N} as follows. 

\begin{lem}
One has
\begin{equation*}
\int_{0}^\infty \left[\| w_2\|_{L^2}^2  +\|\partial_x w_1\|_{L^2}^2 
+\|w_1\|_{L^2}^2  \right]dt\lesssim  \delta^{11/10}.
\end{equation*}
\end{lem}

\begin{proof}
Integrating the estimate \eqref{eq:dtH_inequality} on $[0,t]$, we get
\[
\int_{0}^t \left[\| w_2\|_{L^2}^2  +\|\partial_x w_1\|_{L^2}^2 
+\|w_1\|_{L^2}^2  \right]dt\lesssim 2\sup_{0\leq s\leq t} |\mathcal H(s)| \leq 2\sup_{s\geq 0} |\mathcal H(s)|.
\]
No we estimate the functional $\mathcal H(t)$. Indeed, from \eqref{eq:mH},
\[
\begin{aligned}
|\mathcal H| = &~{} \left| \mathcal J+ 16C_2\delta^{1/10} \mathcal I  -32C_2 C_1\delta^{1/2}\mathcal{N} +80C_1C_2C_4\delta^{1/5}\mathcal{M} \right|\\
\lesssim &~{} |\mathcal J|+ \delta^{1/10}|\mathcal I | + \delta^{1/2}|\mathcal{N}|+ \delta^{1/5}|\mathcal{M}|\\
\lesssim &~{} \bigg| \int   \psi_{A,B} v_1v_2 \bigg|+ \delta^{1/10}\bigg|  \Int{\R} \varphi_A u_1 u_2 \bigg|\\
&+\delta^{1/2}\bigg|\int \rho_{A,B} \partial_x v_1 v_2\bigg|+ \delta^{1/5}\bigg| \int \psi_{A,B} \partial_x v_1 \partial_x  v_2 \bigg|\\
\lesssim &~{} B\big\| v_1v_2 \big\|_{L^1}+ A\delta^{1/10}\big\| u_1 u_2 \big\|_{L^1}+\delta^{1/2}\big\|  \partial_x v_1 v_2\big\|_{L^1}+B \delta^{1/5}\big\| \partial_x v_1 \partial_x  v_2 \big\|_{L^1}.
\end{aligned}
\]
Now, from estimates \eqref{eq:estimates_v1}, \eqref{eq:estimates_v2} and \eqref{v1_mejorada}, we conclude
\[
\begin{aligned}
|\mathcal H| \lesssim &~{} B\delta^2 \gamma^{-1/2}+ A\delta^{2+ 1/10} +\delta^{1/2}\gamma^{-1}\delta^2+B \delta^{1/5}\gamma^{-1}\delta\gamma^{-1/2}\delta.
\end{aligned}
\]

Finally, using \eqref{Fijacion1}, \eqref{Fijacion2}
\[
\begin{aligned}
|\mathcal H| \lesssim &~{} \delta^{1+1/10}(\delta^{6/10}  + 1 + \delta^{10/10}+ \delta^{4/10}) \lesssim \delta^{11/10}.
\end{aligned}
\]
Passing the limit as $t\to \infty$, we conclude.
\end{proof}

By Lemma \ref{lem:v_w}, estimates \eqref{eq:estimates_sech_u1_w1}-\eqref{eq:sech_u2_w2}, one obtains
\begin{equation}\label{eq:int_a1_a2_norm}
\int_{0}^\infty \left(
\int (u_1^2+(\partial_x u_1)^2+u_2^2 )\sech(y) \right) dt\leq \delta.
\end{equation}
Using the above estimate, we will conclude the proof of Theorem \ref{thm:asymptotic}. Let
\begin{equation*}
\begin{gathered}
\mathcal{K}(t)=\int \sech(y) u_1^2+\int \sech(y) (\IOpg \partial_x u_2)^2=:\mathcal{K}_1(t)+\mathcal{K}_2(t).
\end{gathered}
\end{equation*}
For $\mathcal{K}_1$, using \eqref{eq:u_linear} and integrating by parts, we have
\[
\begin{aligned}
\frac{d \mathcal{K}_1}{dt}
=&~{} 2\int \sech(y)(u_1 \dot{u_1}) -\rho' \int \sech'(y)u_1^2\\
= &~{} 2\int \sech(y)u_1 (\partial_x u_2+\rho' H') -\rho' \int \sech'(y)u_1^2 \\
=&~{} -2\int (\sech'(y)u_1+\sech(y)\partial_x u_1) u_2+  2 \rho'  \int \sech(y) H'(y) u_1-\rho' \int \sech'(y)u_1^2.
\end{aligned}
\]
Then, applying H\"older inequality, Cauchy-Schwarz inequality and \eqref{eq:rho_p}, we get
\[
\begin{aligned}
\left|\frac{d}{dt}\mathcal{K}_1(t)\right|
\lesssim & ~{}\int \sech(y)(u_1^2+(\partial_x u_1)^2+ u_2^2) + \left(\int e^{-\sqrt{2}|y|} u_2^2  \right)^{1/2}\left( \int e^{-\sqrt{2}|y|} u_1^2 \right)^{1/2} \\
\lesssim & ~{}\int \sech(y)(u_1^2+(\partial_x u_1)^2+ u_2^2) .
\end{aligned}
\]
For $\mathcal{K}_2$, passing to the variables $(v_1,v_2)$ (see \eqref{eq:change_variable}) 
\begin{equation*}
\begin{gathered}
\mathcal{K}_2=\int \sech(y) ( \partial_x v_2)^2,
\end{gathered}
\end{equation*}
 and using \eqref{eq:syst_v}, we get
\[
\begin{aligned}
\frac{d}{dt}{\mathcal{K}}_2
= &~{} 2\int \sech(y)\partial_x v_2   \partial_x^2  v_1+2\int \sech(y)\partial_x v_2  \partial_x F -\rho' \int \sech'(y)( \partial_x v_2)^2 \\
=:&~{} K_{21}+K_{22} +K_{23}.
\end{aligned}
\]
Integrating by parts in $K_{21}$, we have
\[
\begin{aligned}
K_{21}
= -2\int (\sech'(y) \partial_x v_2+\sech(y)  \partial_x^2 v_2 ) \partial_x v_1.
\end{aligned}
\]
Using \eqref{eq:change_variable} we obtain
\[
\begin{aligned}
|K_{21}|
\lesssim & ~{}\int  \sech (y)( (\partial_x  (1-\gamma \partial_x^2)^{-1}  u_2)^2+(\partial_x^2  (1-\gamma \partial_x^2)^{-1}  u_2)^2+ ( \partial_x  (1-\gamma \partial_x^2)^{-1}  \mathcal L u_1)^2)\\
\lesssim &~{} \left\| \sech^{1/2}(y) (1-\gamma \partial_x^2)^{-1}  \partial_x u_2 \right\|^2_{L^2} +\left\| \sech^{1/2}(y) (1-\gamma \partial_x^2)^{-1} (\partial_x^2 -1+1)u_2 \right\|^2_{L^2} \\
&~{} +\left\|  \sech^{1/2}(y)   (1-\gamma \partial_x^2)^{-1} \partial_x( \mathcal L u_1 )\right\|^2_{L^2}.
\end{aligned}
\]
Using Lemma \ref{eq:lem_sec_Iop g}, more precisely \eqref{cor:estimates_Sech_Iop_partial} and \eqref{lem:estimates_Sech_Iop_Op}, \eqref{eq:LL}, and \eqref{eq:estimates_v1}, we obtain
\[
\begin{aligned}
|K_{21}|
\lesssim & ~{} \frac1{\gamma^2}  \left\| \sech^{1/2}(y)  u_2 \right\|^2_{L^2} +\left\|  \sech^{1/2}(y)   (1-\gamma \partial_x^2)^{-1} ( \mathcal L \partial_x u_1 +\partial_x V_0 u_1)\right\|^2_{L^2} \\
\lesssim &~{} \gamma^{-2} \int  \sech (y)( u_2^2+ ( \partial_x u_1)^2 + u_1^2).
\end{aligned}
\]
For $K_{22}$, we use Cauchy-Schwartz inequality,  \eqref{cor:estimates_Sech_Iop_partial} and a similar computation of  \eqref{eq:LambdaP_N}, then
\[
\begin{aligned}
|K_{22}| \lesssim & \int \sech(y)[ (\IOpg \partial_x u_2)^2 + (\IOpg \partial_x ((u_1^3+3H u_1^2)))^2]\\
 \lesssim_{\gamma}& ~{}    \int \sech(y)[ u_2^2 + u_1^2+(\partial_x u_1)^2].
\end{aligned}
\]
The term $K_{23}$ is bounded as follows: using \eqref{eq:change_variable}
\[
\begin{aligned}
|K_{23}| \lesssim &~{}  \left(\int e^{-\sqrt{2}|y| }u_2^2\right)^{1/2} \left(\int \sech(y)( \partial_x v_2)^2 \right) \\
\lesssim &~{}  \gamma^{-2} \int \sech(y) u_2^2 + \gamma^2 \left(\int \sech(y)( \partial_x (1-\gamma \partial_x^2)^{-1}  u_2)^2 \right)^2 \\
\lesssim &~{}  \gamma^{-2} \int \sech(y) u_2^2 +  \left(\int \sech(y) u_2^2 \right)^2.
\end{aligned}
\]
Then, we conclude
\[
\begin{aligned}
\left|\frac{d}{dt}{\mathcal{K}}_2(t)\right|
\lesssim_{\gamma}& ~{}  \int  \sech(y)(u_1^2+(\partial_x u_1)^2+ u_2^2).
\end{aligned}
\]
By \eqref{eq:int_a1_a2_norm}, there exists and increasing sequence $t_n\to \infty$ such that
\[
\lim_{n\to \infty} \left[\mathcal{K}_1(t_n)+\mathcal{K}_2(t_n)\right]=0.
\]
For $t\geq 0$, integrating on $[t,t_n]$, and passing to the limit as $n\to \infty$, we obtain
\[
\mathcal{K}(t)\lesssim \int_{t}^\infty \left[ \int \sech(y)(u_1^2+(\partial_x u_1)^2+ u_2^2)\right]dt.
\]
By \eqref{eq:int_a1_a2_norm}, we deduce 
\[
\lim_{t\to\infty} \mathcal{K}(t)=0.
\]

By the decomposition of solution \eqref{eq:decomposition} and the boundedness in $H^1$ of $u_1$, this implies \eqref{eq:local_stability}. This ends the proof of Theorem \ref{thm:asymptotic}.

\appendix
\section{Proof Lemma \ref{lem:PLL}} \label{app:Null_spaces}

The first step is understand the generalized null space of this operator. By Coppel \cite{coppel}, we know that the asymptotic behavior of the solutions of  fourth-order differential equation 
\begin{equation}\label{eq:PL=0}
-\PL\phi=0,
\end{equation}
 are determined by the asymptotic behavior  of solutions 
\[
-\partial_x^2 \LL_0 u=0, \quad \mbox{where} \quad  \LL_0=-\partial_x^2+2.
\]
One can see, analyzing the coefficient matrix  of the first-order system associated to above equation, that the eigenvalues are simple and given by $\pm\sqrt{2}$ and $0$ is an eigenvalue of multiplicity two.  Then, by Theorem 4 on Chapter 4 of \cite{coppel}, the asymptotic behavior of the solutions of \ref{eq:PL=0} are, modulo scaling, as $x\to +\infty$,
\[
e^{-\sqrt{2}x},\quad e^{\sqrt{2}x}, \quad 1,\quad x.
\]
Using classical techniques of ODE and  the principle of superposition, one obtains that the solution of $-\PL \phi=0$, is a linear combination of the linearly independent functions
\[
\begin{aligned}
& u_0(x)=H'(x),
\quad u_1(x)= H' (x)\int_{0}^{x} (H'(y))^{-2} dy,
\\
& \quad u_2(x)=-2H' (x)\int_{0}^{x} (H'(y))^{-2}\int_{-\infty}^{y} H'(s)ds dy,
 \\
&  u_3(x)=H' (x)\int_{0}^{x} (H'(y))^{-2}\int_{-\infty}^{y} s H'(s)ds dy.
\end{aligned}
\]
We also have $\lim_{x\to -\infty} u_0(x) =0$, but the others do not belong to $L^2$:
\[\displaystyle
\begin{aligned}
\lim_{x\to -\infty} u_1(x)
 =&\lim_{x\to -\infty}  \dfrac{\displaystyle\int_{0}^{x} (H'(y))^{-2} dy}{(H' (x))^{-1}}
  =\lim_{x\to -\infty}  \frac{ (H'(x))^{-2}}{-(H' (x))^{-2}H''(x)}=\lim_{x\to -\infty}  \frac{1 }{-H (H^2-1)}=+\infty.
\end{aligned}
\]
also,
\[
\begin{aligned}
\lim_{x\to -\infty} u_2(x)
& =\lim_{x\to -\infty} -2 \frac{\displaystyle\int_{0}^{x} (H'(y))^{-2} (H(y)+1)dy}{(H' (x))^{-1}}
 \\
  &=\lim_{x\to -\infty}  -2\frac{ (H'(x))^{-2}(H(x))+1)}{-(H' (x))^{-2}H''(x)}
  \\&
    =\lim_{x\to -\infty}  -2\frac{H(x)+1 }{-H (H-1)(H+1)}=1.
\end{aligned}
\]
Finally,
 \[
\begin{aligned}
\lim_{x\to -\infty} u_3(x)
& =\lim_{x\to -\infty}  \dfrac{\displaystyle\int_{0}^{x} (H'(y))^{-2} \int_{-\infty}^{y} s H'(s)dsdy}{(H' (x))^{-1}}
  \\&
  =\lim_{x\to -\infty}  \frac{ \displaystyle(H'(x))^{-2}\int_{-\infty}^{x} s H'(s)ds}{-(H' (x))^{-2}H''(x)}
  =\lim_{x\to -\infty}  \frac{\displaystyle \int_{-\infty}^{x} s H'(s)ds}{-H''(x)}
    \\&      =\lim_{x\to -\infty}  \frac{ x H'(x)}{-H'''(x)}
           =\lim_{x\to -\infty}  \frac{ x H'(x)}{-H'(x)(3H^2-1)}
             \\&
           =\lim_{x\to -\infty}  \frac{ x }{1-3H^2}
           =-\infty.
\end{aligned}
\]

This proves Lemma \ref{lem:PLL}.

\section{Proof of Lemma \ref{lem:LL}}\label{app:LL_idnt}

Replacing $\LL$ (see \eqref{eq:LL}) and integrating by parts, we get
\[
\Jap{\eta \LL(f)}{g}=\int\eta [\partial_x f \partial_x g+V_0 fg]+\Jap{\eta' \partial_x f}{g}.
\]
In particular, we get
\[
\begin{aligned}
\Jap{\eta \partial_x \LL(f)}{f}
=&\int\eta [-\partial_x^3 f  f+\partial_x (V_0 f) f]
=\int\eta [\partial_x^2 f  \partial_x f+V_0' f^2 +V_0 \partial_x f f]+ \int\eta' \partial_x^2 f  f\\
=&\int\eta \left(\frac12 \partial_x  [ (\partial_x f)^2]+V_0' f^2 +\frac12 V_0 \partial_x [f^2] \right)- \int \partial_x (\eta' f)\partial_x f
\\
=&-\frac12  \int\eta' [  (\partial_x f)^2 +V_0 f^2 ]+\int \eta V_0' f^2-\frac12 \int \eta V_0'  f^2- \int (\eta'' f+\eta' \partial_x f)\partial_x f    \\
=&-\frac12  \int\eta' [ 3 (\partial_x f)^2 +V_0 f^2 ]+\frac12 \int \eta V_0'  f^2+\frac12 \int \eta''' f^2.
\end{aligned}
\]
Lastly, integrating by parts, we have
\[
\begin{aligned}
\Jap{\eta \LL( \partial_x f)}{f}
=&\int\eta [-\partial_x^3 f  f+V_0 \partial_x f f]\\
=&-\frac12  \int\eta' [ 3 (\partial_x f)^2 +V_0 f^2 ]-\frac12 \int \eta V_0'  f^2+\frac12 \int \eta''' f^2.
\end{aligned}
\]
Note that
\[
\begin{aligned}
\Jap{\eta \partial_x \LL(\partial_x f)}{f}=&-\Jap{\eta \LL(\partial_x f)}{\partial_x f}-\Jap{\eta' \LL(\partial_x f)}{f},
\end{aligned}
\]
then, by \eqref{eq:L} and \eqref{eq:LP}, we get
\[
\begin{aligned}
\Jap{\eta \partial_x \LL(\partial_x f)}{f}=&-\int\eta [(\partial_x^2 f)^2 +V_0 (\partial_x f)^2]+\frac12 \int \eta'' (\partial_x f)^2 
\\
&+\frac12  \int\eta'' [ 3 (\partial_x f)^2 +V_0 f^2 ]+\frac12 \int \eta' V_0'  f^2-\frac12 \int \eta^{(4)} f^2\\
=&-\int\eta [(\partial_x^2 f)^2 +V_0 (\partial_x f)^2]+2 \int \eta'' (\partial_x f)^2 
\\
&+\frac12  \int\eta''  V_0 f^2 +\frac12 \int \eta' V_0'  f^2-\frac12 \int \eta^{(4)} f^2.
\end{aligned}
\]
The proof is complete.

\section{Local well-posedness} \label{app:WP_kink}

In this section, we will prove that the small perturbation around the static kink are in fact locally well-posed on the energy spaces associated, i.e. on $H^{1}\times L^{2}$. It has been proved that the linear part of \eqref{eq:CP+1} is well posed on $H^1\times L^2$ (see \cite{Notes_linares}). Now we will focus on the nonlinear equation associated to \eqref{eq:CP+1} is locally well posed. Notice that \eqref{eq:CP+1} is equivalent to
\[
\partial_{t}^2 u_{1}-2\partial_{x}^{2}u_{1}+\partial_{x}^{4} u_{1}-\partial_{x}^{2} F(t,x,u_{1})=0,
\]
where $F(t,x,u_1)=3u_1(H^2-1)+u_1^2(u_1+3H)$, satisfies \eqref{eq:F_bound}.
Let the operator  $\Phi$  given by 
\begin{equation}\label{eq:operador_phi}
\Phi[u](t)=  \mathcal{G}(t) u_1^0(x)+\mathcal{K}(t) u_2^0(x) +\int_{0}^t \mathcal{K}(t-s) \partial_x^2 F(x,s,u)ds,
\end{equation}
where
\[
\mathcal{G}(t)= \mathcal{F}^{-1} G(t,\xi) \mathcal{F}, \quad \mathcal{K}(t)= \mathcal{F}^{-1} K(t,\xi) \mathcal{F};
\]
 $\mathcal{F}$ and  $\mathcal{F}^{-1}$ represents the Fourier transform and its inverse, respectively. The Fourier multipliers are given by
\[
\begin{aligned}
G(t,\xi)=  \cos(\omega(\xi)t), 
\quad 
K(t,\xi)= \frac{\sin(\omega(\xi) t)}{\omega(\xi)}, \mbox{ and }  \omega(\xi)=|\xi|\sqrt{\xi^2+2}.
\end{aligned}
\]
Following the ideas in the proof of \cite{Notes_linares}, we will use a contraction mapping argument and analogous estimates in the linear case. 
\medskip

We will invoke some linear estimates obtained in \cite{Linares_93}, this estimates are valid in our case since $\frac12 \omega(\xi)=\phi(\tilde{\xi})=|\tilde{\xi}|\sqrt{1+\tilde{\xi}^2}$, while the last one is the same symbol studied by Linares for the Good-Boussinesq equation.

\begin{rem}\label{rem:w_phi}
One can notice that $\frac12 \omega(\xi)=\phi(\tilde{\xi})=|\tilde{\xi}|\sqrt{1+\tilde{\xi}^2}$, where the last one is the same symbol studied by Linares for the Good-Boussinesq equation. In fact, let $v\in L^2$ and considering $\xi=\tau/\sqrt{2}$, we get
\[
\begin{aligned}
V_1(2t) v(\sqrt{2}x)
=&\int e^{i(2t\phi(\xi)+\sqrt{2}x\xi)}\widehat{v(\sqrt{2}x)}(\xi)d\xi
=\frac{1}{\sqrt{2}}\int e^{i(2t\phi(\xi)+\sqrt{2}x\xi)}\widehat{v(x)}\left(\frac{\xi}{\sqrt{2}}\right)d\xi\\
=&\frac{1}{2}\int e^{i\big(tw(\tau)+x\tau \big)}\hat{v}\left(\frac{\tau}{2}\right)d\tau\\
\end{aligned}
\]
Let  $u$ a function such that $\widehat{(u(x))}(\tau)=\widehat{(v(x))}(\tau/2)$, then 
\[
\begin{aligned}
V_1(2t) v(x)
=&\frac{1}{2}\int e^{i\big(t w(\tau)+x\tau\big)}\widehat{u(y)}(\tau)d\tau
=\tilde{V}_1(t)u(y),
\end{aligned}
\]
where $\widetilde{V}_1=\mathcal{F}^{-1}[e^{iw(\tau)t}\mathcal{F}]$ is related to the Fourier multiplier $\mathcal{G}$. It is analogous for $V_2$ and the Fourier multiplier $\mathcal{K}$.
\end{rem}
The following results have been proved in \cite{Notes_linares,Linares_93}.
\begin{lem}[{\cite[Lemma 2.7]{Notes_linares}}] \label{lem:V1}
Let $f\in L^2$ and
\[
V_1(t) f(x) =\int e^{i(t \phi(\xi)+x\xi)}\hat{f}(\xi) d\xi,
\]
where $\phi(\xi)=|\xi|(1+\xi^2)^{1/2}$, then
\begin{equation*}
\| V_1(t)f\|_{L^2}\leq \| f \|_{L^2}
\end{equation*}
and 
\begin{equation*}
\left( \int_{0}^{T}\|V_1(t)f\|^4_{L^{\infty}} dt \right)^{1/4}\leq c (1+T^{1/4})\|f\|_{L^2}.
\end{equation*}
\end{lem}

\begin{lem}[{\cite[Lemma 2.8]{Notes_linares}}] \label{lem:V2}
Let $g=h' \in L^2$ and
\[
V_2(t) h'(x) =\int e^{i(t \phi(\xi)+x\xi)}\frac{\sgn(\xi) \hat{h}(\xi) }{(1+\xi^2)^{1/2}}d\xi.
\] 
Then
\begin{equation*}
\| V_2(t)h'\|_{L^2}\leq c \| h \|_{H^{-1}},
\end{equation*}
and 
\begin{equation*}
\left( \int_{0}^{T}\|V_2(t)h'\|^4_{L^{\infty}}dt \right)^{1/4}\leq c (1+T^{1/4})\|h\|_{H^{-1}}.
\end{equation*}
\end{lem}

\begin{lem}[{\cite[Lemma 2.8]{Notes_linares}}]\label{lem:V2_xx}
Let $g=p'' \in L^2$ and
\[
V_2(t) p''(x) =\int e^{i(t \phi(\xi)+x\xi)}\frac{\sgn(\xi) \xi \hat{h}(\xi) }{(1+\xi^2)^{1/2}}d\xi.
\] 
Then
\begin{equation*}
\| V_2(t)p''\|_{L^2}\leq c \| p \|_{L^2}
\end{equation*}
and 
\begin{equation*}
\left( \int_{0}^{T}\|V_2(t)p''\|^4_{L^{\infty}}dt \right)^{1/4}\leq c\|p\|_{L^2}.
\end{equation*}
\end{lem}

Finally,
\begin{lem}[{\cite[Proposition 2.12]{Notes_linares}}] \label{lem:dx(-)dt}
Let 
\[\begin{aligned}
(-\partial_x^2)^{-1/2}\partial_t V_1(t) f(x) &=\int i|\xi|^{-1}\phi(\xi) e^{i(t \phi(\xi)+x\xi)}  \hat{f}(\xi) ~{}d\xi,\\
(-\partial_x^2)^{-1/2}\partial_t V_2(t) h'(x) &=\int i|\xi|^{-1} e^{i(t \phi(\xi)+x\xi)}\frac{\sgn(\xi) \hat{h}(\xi) }{(1+\xi^2)^{1/2}}~{}d\xi,
\end{aligned}
\] 
and
\[
(-\partial_x^2)^{-1/2}\partial_t V_2(t) p''(x) =\int i|\xi|^{-1} e^{i(t \phi(\xi)+x\xi)}\frac{ \xi \hat{p}(\xi) }{(1+\xi^2)^{1/2}}~{}d\xi.
\] 
Then
\begin{equation*}
\begin{aligned}
\| (-\partial_x^2)^{-1/2}\partial_t V_1(t) f\|_{L^2}\leq C \| f \|_{H^{1}}, \\
\| (-\partial_x^2)^{-1/2}\partial_t V_2(t) h'\|_{L^2}\leq C \| h \|_{L^2},
\end{aligned}
\end{equation*}
and 
\begin{equation*}
\| (-\partial_x^2)^{-1/2}\partial_t V_2(t) p''\|_{L^2}\leq C \| p \|_{H^1}.
\end{equation*}
\end{lem}

\medskip

Let $\Phi$ operator given by \eqref{eq:operador_phi}
 and, similarly as in \cite{Notes_linares}, 
\[
\begin{aligned}
Y_{T}^{a}
=&\bigg\{ u\in C([0,T]:H^{1}(\R))\cap L^{4}([0,T];L_{1}^{\infty}(\R)) \quad \bigg| \quad (-\partial_{x}^{2})^{-1/2}\partial_{t} u \in C([0,T],L^{2}(\R)),
 \\ &~{}\quad \sup_{[0,T]} \| (1-\partial_{x}^{2})^{1/2}u(t)\|_{L^{2}}\leq a, \quad \sup_{[0,T]} \| (-\partial_{x}^{2})^{-1/2}\partial_{t} u(t)\|_{L^{2}}\leq a \mbox{ and }\|| u \||_{1}\leq a \bigg\},
\end{aligned}
\]
where
\[
|\| u\||_{1}=\left( \int_{0}^{T} \|u (t)\|_{L^{\infty}}^{4}dt\right)^{1/4}+\left( \int_{0}^{T} \|\partial_{x}u (t)\|_{L^{\infty}}^{4}dt\right)^{1/4}.
\]

\begin{rem}
The multiplier $\mathcal{G}$ and $\mathcal{K}$ are directly related to $V_1$ and $V_2$, respectively. This means that, considering Remark \ref{rem:w_phi}, Lemmas \ref{lem:V1}, \ref{lem:V2}, \ref{lem:V2_xx} and \ref{lem:dx(-)dt} are valid for  $\mathcal{G}$ and $\mathcal{K}$, respectively.
\end{rem}

First, we will show that 
\begin{prop}\label{prop:Phi_YaT}
Let $u_1^{0}\in H^{1}(\R)$ and $u_{2}^{0}=\partial_{x} \tilde{u}_{0}^{2}\in L^{2}$, define $\Phi[u]$ as in \eqref{eq:operador_phi}. Then 
\[
\Phi:Y^{a}_{T}\to Y^{a}_{T}
\]
for some $T=T(\delta)$, where $\delta$ such that $\|u_1^{0}\|_{H^{1}}, \|u_{2}^{0}\|_{L^{2}}\leq\delta$.
\end{prop}

We will need the following technical lemma to manage the nonlinearity $F$.
\begin{lem}
Let $F$ as in \eqref{nueva_F}, then the following inequalities holds
\begin{equation}\label{eq:Fx_L2}
\begin{aligned}
\|F(x,t,u)\|_{L^{2}}\lesssim& \|u(H^2-1)\|_{L^2}+(\|u\|_{L^\infty }^2+\|u\|_{L^\infty })\|u\|_{L^2},\\
\|\partial_x F\|_{L^2}\lesssim& (1+\|u\|_{L^\infty})\|u H'\|_{L^2}+\|\partial_{x}u (H^{2}-1)\|_{L^2}
\\&+(\|u \|_{L^\infty}^2+\|u \|_{L^\infty})\|\partial_{x} u\|_{L^2}\\ 
\end{aligned}
\end{equation}
and 
\begin{equation}\label{eq:Fu-Fv_L2}
\begin{aligned}
\| F(x,t,u)-&F(x,t,v)\|_{L^{2}}\lesssim   \|u-v\|_{L^2}(\|u\|_{L^\infty}^2+\|w\|_{L^\infty}^2+\|u\|_{L^\infty}+\|v\|_{L^\infty})
\\&\quad\quad\quad \quad\quad\quad\quad+ \|(u-v)(H^2-1)\|_{L^2},
\\ 
\|\partial_{x} F(x,s,u)&-\partial_x F(x,s,v)\|_{L^{2}} \\
\lesssim&~{}   \|\partial_{x}u-\partial_{x}v\|\big(\|u\|_{L^\infty}+\|v\|_{L^\infty}+ \|u\|_{L^\infty}^2+\|v\|_{L^\infty}^{2}\big)\\
&+\|u-v\|_{L^2} \big( \|u\|_{L^\infty}+\|v\|_{L^\infty}+\|u\|_{L^\infty}^{2}+\|v\|_{L^\infty}^{2}+ \|\partial_{x} u\|_{L^\infty}^{2}+\|\partial_{x}v\|_{L^\infty}^{2}\big)\\
&+\|(H^2-1) (\partial_{x}u-\partial_{x}v)\|_{L^2}+\|6HH' (u-v) \|_{L^2}
.
\end{aligned}
\end{equation}
\end{lem}
\begin{proof}
The first inequality follows directly from the definition of $F$. 
For the second inequality, we notice that
\[
\begin{aligned}
\partial_{x} F(x,t,u)
=&~{}3 u H'(2H+u)+\partial_{x}u(H^{2}-1)+(3u^2+6Hu) \partial_{x} u.
\end{aligned}
\]
Then,  we get
\[\begin{aligned}
\|\partial_{x} F(x,t,u)\|_{L_{x}^{2}}
\lesssim&~{}  (1+\|u\|_{L^\infty})\|u H'\|_{L^2}+\|\partial_{x}u (H^{2}-1)\|_{L^2}
\\&~{}+(\|u\|_{L^\infty}^2+\|u \|_{L^\infty})\|\partial_{x} u\|_{L^2}.\\
\end{aligned}
\]
Secondly, using \eqref{eq:Fu-Fv}, we get
\[
\begin{aligned}
\|F(t,x,u)-F(t,x,v)\|_{L^2}
\leq&~{} 3 \|(u-v)(H^2-1)\|_{L^2}
\\&~{}+\|u-v\|_{L^2}(\|u\|_{L^\infty}^2+\|v\|_{L^\infty}^2+\|u\|_{L^\infty}+\|v\|_{L^\infty}).
\end{aligned}
\]
For the last inequality, we notice that
\[
\begin{aligned}
|\partial_{x}(F(t,x,u)-F(t,x,v))|
\lesssim&~{} |\partial_{x}u-\partial_{x}v||(H^2-1)+u^2+v^{2}+u+v|\\
&+|u-v||6HH'+u+v+u^{2}+v^{2}+ (\partial_{x} u)^{2}+(\partial_{x}v)^{2}|.
\end{aligned}
\]
Then,
\[
\begin{aligned}
\|\partial_x F(t,x,u)-\partial_x &F(t,x,v)\|_{L^2}\\
\leq&~{} \|\partial_{x}u-\partial_{x}v\|\big(\|u\|_{L^\infty}+\|v\|_{L^\infty}+ \|u\|_{L^\infty}^2+\|v\|_{L^\infty}^{2}\big)\\
&~{} +\|u-v\|_{L^2} \big( \|u\|_{L^\infty}+\|v\|_{L^\infty}+\|u\|_{L^\infty}^{2}+\|v\|_{L^\infty}^{2}+ \|\partial_{x} u\|_{L^\infty}^{2}+\|\partial_{x}v\|_{L^\infty}^{2}\big)\\
&~{}+\|(H^2-1) (\partial_{x}u-\partial_{x}v)\|_{L^2}+\|6HH' (u-v) \|_{L^2}.
\end{aligned}
\]
This concludes the proof.
\end{proof}
Now we prove Proposition \ref{prop:Phi_YaT}.
\begin{proof}[Proof \ref{prop:Phi_YaT}]
First, we will focus on the norms $\|\cdot\|_{L_t^{\infty}L^2_x}$. By Remark \ref{rem:w_phi}, Lemma \ref{lem:V1}, Lemma \ref{lem:V2} and H\"older's inequality, we get
\[
\begin{aligned}
\sup_{[0,T]}\|\Phi[u]\|_{L^{2}}
\leq& ~{} \sup_{[0,T]}\|  \mathcal{G}(t) u_1^0(x)\|_{L^{2}}+\sup_{[0,T]}\|\mathcal{K}(t) u_2^0(x) \|_{L^{2}}
+\sup_{[0,T]} \left\|\int_{0}^t \mathcal{K}(t-s) \partial_x^2 F(x,s,u)ds\right\|_{L^{2}}\\
\leq& ~{}c\|  u_1^0(x)\|_{L^{2}}+\| \tilde{u}_2^0(x) \|_{H^{-1}}
+ \int_{0}^T \left\| F(x,s,u)\right\|_{L^{2}} ds.
\end{aligned}
\]
Using \eqref{eq:Fx_L2} and the Sobolev embedding, one has
\[
\begin{aligned}
\sup_{[0,T]}\|\Phi[u]\|_{L^{2}}
\leq& ~{}c\|  u_1^0(x)\|_{H^{1}}+\| \tilde{u}_2^0(x) \|_{L^2}
+ cT \sup_{[0,T]}(  \|u\|_{H^1}+\|u\|_{H^1}^3+\|u\|_{H^1}^2).
\end{aligned}
\]
Similarly, for the term $\partial_x \Phi$, using Lemmas \ref{lem:V1}  and \ref{lem:V2} it is obtained that
\[
\begin{aligned}
\sup_{[0,T]}\|\partial_{x} \Phi[u]\|_{L^{2}}
\leq &~{} \|  \partial_{x} u_1^0(x)\|_{L^{2}}+\| u_2^0(x) \|_{L^{2}}
+\int_{0}^T  \left\| \partial_x F(x,s,u)ds\right\|_{L^{2}}.
\end{aligned}
\]
From \eqref{eq:Fx_L2} and the Sobolev embedding, we get
\[
\begin{aligned}
\sup_{[0,T]}\|\partial_{x} \Phi[u]\|_{L^{2}}
\leq &~{} c\left(  \|  u_1^0(x)\|_{H^{1}}+\| u_2^0(x) \|_{L^{2}} \right)
+cT  \sup_{[0,T]}  \left(\|u\|_{H^1}+\|u\|_{H^1}^2+\|u \|_{H^1}^3 \right).
\end{aligned}
\]
Then, we arrive to 
\[
\begin{aligned}
\sup_{[0,T]}\|\Phi[u]\|_{H^1}\leq&~{} \sup_{[0,T]}\|\Phi[u]\|_{L^2}+\sup_{[0,T]}\|\partial_x \Phi[u]\|_{L^2}\\
\leq&~{} 
2c( \|  u_1^0(x)\|_{H^{1}}+\| u_2^0(x) \|_{L^{2}})
+2cT  \sup_{[0,T]}  (\|u\|_{H^1}+\|u\|_{H^1}^2+\|u \|_{H^1}^3).
\end{aligned}
\]
Hence, if
\[
a=8c\delta \mbox{ and } T=\delta^4/4c,
\]
we obtain
\[
\color{black}
4c\delta+2cTa(1+a+a^2)=4c\delta(1+4cT(1+a+a^2))<8c\delta.
\]
Now, to estimate $(-\partial_{x}^{2})^{{-1/2}}\partial_{t} \Phi[u]$, we will use Lemma \ref{lem:dx(-)dt} and \eqref{eq:Fx_L2}; we obtain
\[
\begin{aligned}
\sup_{[0,T]}\left\|(-\partial_{x}^{2})^{{-1/2}}\partial_{t} \Phi[u]\right\|_{L^{2}}
\leq&~{}
c(\| u_1^0(x)\|_{H^1}+\| u_2^0(x)\|_{L^2} )
+cT\sup_{[0,T]} (
\|u\|_{H^1}+\|u\|_{H^1}^2+\|u\|_{H^1}^3).
\end{aligned}
\]
As before,
\[\color{black}
2c \delta+acT(1+a+a^2)=2c\delta(1+2cT(1+a+a^2))<8c\delta.
\]
Finally, we will estimate the $L^4L^\infty$ norm. Applying Lemmas \ref{lem:V2_xx}, \ref{lem:V2} and \ref{lem:V1}; and using \eqref{eq:Fx_L2}, we get
\[
\begin{aligned}
\left\| \Phi[u] \right\|_{L_{T}^{4}L_{x}^{\infty}}
\leq&~{} \|\mathcal{G}(t) u_1^0(x)\|_{L_{T}^{4}L_{x}^{\infty}}
+\|\mathcal{K}(t) u_2^0(x)\|_{L_{T}^{4}L_{x}^{\infty}} +\left\| \int_{0}^t \mathcal{K}(t-s) \partial_x^2 F(x,s,u)ds \right\|_{L_{T}^{4}L_{x}^{\infty}}\\
\leq&~{} (1+T^{1/4})\|u_1^0\|_{L^2}
+\| \tilde{u}_2^0\|_{L^2} +cT\sup_{[0,T]}(
\|u\|_{H^1}+\|u\|_{H^1}^2+\|u\|_{H^1}^3).
\end{aligned}
\]
Similarly, as the above estimates we obtain
\[
\begin{aligned}
\left\| \partial_{x} \Phi[u] \right\|_{L_{T}^{4}L_{x}^{\infty}}
\leq&~{} (1+T^{1/4})(\|u_1^0\|_{L^2}
+\| \tilde{u}_2^0\|_{L^2}) +cT\sup_{[0,T]}(
\|u\|_{H^1}+\|u\|_{H^1}^2+\|u\|_{H^1}^3).
\end{aligned}
\]
Finally,
\[
\begin{aligned}
\||\Phi[u]\||_{1}=&~{}\| \Phi[u]\|_{L_{T}^{4}L_{x}^{\infty}}+\| \partial_{x} \Phi[u]\|_{L_{T}^{4}L_{x}^{\infty}}\\
\leq& 2c\delta (1+T^{1/4}) +2cTa(1+a+a^2),
\end{aligned}
\]
and we get
\[\color{black}
 2c\delta[ (1+T^{1/4}) +2cTa(1+a+a^2)]=2c\delta [ (1+T^{1/4}) +2cT(1+a+a^2)]<8c\delta.
\]
This ends the proof of the proposition.
\end{proof}

Secondly, we will prove that $\Phi$ is a contraction mapping.
\begin{thm}
Let $u_1^{0}\in H^{1}(\R)$ and $u_{2}^{0}=\partial_{x} \tilde{u}_{0}^{2}\in L^{2}$. Then there exists $T=T(\delta)$ and a unique solution of the integral equation \eqref{eq:operador_phi} in $[0,T]$ with
\[
u\in C([0,T];H^{1}(\R))\cap L^{4}([0,T];L^{\infty}),
\]
and 
\[
(-\partial_{x}^{2})^{-1/2}\partial_{t} u\in C([0,T];L^{2}(\R)).
\]
\end{thm}
\begin{proof}
Proposition \ref{prop:Phi_YaT} ensures that $\Phi[u]:Y_{T}^{a}\to Y^{a}_{T} $, so we only need to show that $\Phi$ is a contraction.
Firstly, we notice 
\[
(\Phi[u]-\Phi[v])(t)=\int_{0}^t \mathcal{K}(t-s) \partial_x^2 \big(F(x,s,u)- F(x,s,v)\big)ds.
\]
Now, we will focus on the $L_{t}^{\infty} L_{x}^{2}$ norm. By  H\"older inequality and Lemma \ref{lem:V2_xx}, we obtain 
\begin{equation*}
\begin{aligned}
\sup_{[0,T]}\|\Phi[u]&-\Phi[v]\|_{L^{2}}(t)\\
\leq&~{} C \int_{0}^T  \left\|  u^2+(u+v)(v+3H) \right\|_{L^{\infty}}  \left\|  u-v \right\|_{L^{2}}ds+C \int_{0}^T   \left\|  |u-v||3(H^2-1)| \right\|_{L^{2}}ds\\
\leq&~{} C\sup_{{[0,T]}} \left\|  u-v \right\|_{L^{2}} \left(\int_{0}^T ( \left\|  u\right\|_{L^{\infty}}^2+\left\|  v\right\|_{L^{\infty}}^2 +\left\|  u\right\|_{L^{\infty}}+\left\|  v\right\|_{L^{\infty}})  ds+ \int_{0}^T 1ds\right).
\end{aligned}
\end{equation*}
Applying H\"older inequality,
\[
\begin{aligned}
\sup_{[0,T]}\|\Phi[u]-\Phi[v]\|_{L^{2}}(t)
\leq&~{} C T^{1/2}
\bigg[
 2a^{2}+2T^{1/4} a+T^{1/2}
 \bigg]
  \sup_{{[0,T]}} \left\|  u-v \right\|_{L^{2}} .
\end{aligned}
\]
Then, by \eqref{eq:Fu-Fv_L2},
\begin{equation}\label{eq:phix}
\begin{aligned}
& \sup_{[0,T]}\|\partial_{x}\Phi[u]-\partial_{x}\Phi[v]\|_{L^{2}}(t) \\
 &~{}\leq C \int_{0}^T \| \partial_{x}u-\partial_{x}w\|_{L^{2}}\|u^2+v^{2}+u+v\|_{L^{\infty}}\\
&~{}\quad  + C \int_{0}^T \|u-w\|_{L^{2}} \|u+v+u^{2}+v^{2}+ (\partial_{x} u)^{2}+(\partial_{x}v)^{2}\|_{L^{\infty}}ds\\
&~{}\quad  + C \int_{0}^T \| (\partial_{x}u-\partial_{x}w)(H^2-1)\|_{L^{2}} + C \int_{0}^T \|(u-w)6HH'\|_{L^{2}}ds.
\end{aligned}
\end{equation}
Hence, by Lemma \ref{lem:V2_xx}, we get
\[
\begin{aligned}
& \sup_{[0,T]}\|\partial_{x}\Phi[u]-\partial_{x}\Phi[v]\|_{L^{2}}(t)\\
&~{} \leq C \sup_{[0,T]}\| \partial_{x}u-\partial_{x}w\|_{L^{2}}\bigg[ 2a^{2}T^{1/2}+2T^{3/4} a+T \bigg]\\
&~{} \quad +C\sup_{[0,T]} \|u-w\|_{L^{2}}  \bigg[ 2T^{3/4}a+4T^{1/2}a^{2} +T\bigg]\\
&~{} \leq C \big( \sup_{[0,T]}\| \partial_{x}u-\partial_{x}w\|_{L^{2}} +\sup_{[0,T]} \|u-w\|_{L^{2}} \big)\bigg[ 4a^{2}T^{1/2}+2T^{3/4} a+T \bigg].
\end{aligned}
\]
Now, as in \eqref{eq:phix}, 
\[
\begin{aligned}
\sup_{[0,T]}\|(-\partial_{x}^{2})^{-1/2}\partial_{t}\Phi[u]- &(-\partial_{x}^{2})^{-1/2}\partial_{t}\Phi[v]\|_{L_{x}^{2}}\\
\leq&~{} C \left( \sup_{[0,T]}\| \partial_{x}u-\partial_{x}w\|_{L^{2}} +\sup_{[0,T]} \|u-w\|_{L^{2}} \right)\bigg[ 4a^{2}T^{1/2}+2T^{3/4} a+T \bigg].
\end{aligned}
\]
Finally, we will estimates the $L_{T}^{4}L_{x}^{\infty}$ norm, using H\"older's inequality, we get
\[
\begin{aligned}
\|\Phi[u]-\Phi[v]\|_{L_{t}^{4}L_{x}^{\infty}}
\leq&~{} C T^{1/2} \bigg[ 2a^{2}+2T^{1/4} a+T^{1/2} \bigg] \sup_{{[0,T]}} \left\|  u-v \right\|_{L^{2}} .
\end{aligned}
\]
Furthermore,
\[
\begin{aligned}
\|\partial_{x}\Phi[u]-\partial_{x}\Phi[v]\|_{L_{T}^{4}L_{x}^{\infty}}
\leq&~{} C \left( \sup_{[0,T]}\| \partial_{x}u-\partial_{x}w\|_{L^{2}} +\sup_{[0,T]} \|u-w\|_{L^{2}} \right)\bigg[ 4a^{2}T^{1/2}+2T^{3/4} a+T \bigg].
\end{aligned}
\]
Then, for $T=\delta^{4}/4c<1$ and $a=8c\delta$, we get that 
\[
\begin{aligned}
4a^{2}T^{1/2}+2T^{3/4} a+T 
\leq&~{} \delta^{4}( 3(8c)^{2}+(8c+1)^{2}),
 \\
 2a^{2}T^{1/2} +2T^{3/4} a+T 
 \leq&~{} \delta^{4}((8c)^{2} +(8c+1)^{2}),
\end{aligned}
\]
and for $\delta$ small enough we obtain that $\Phi$ is a contraction.
\end{proof}

\end{document}